\documentclass[11pt]{amsart}

\usepackage{epsf,amssymb,times,overpic,amscd}
\usepackage[usenames]{color}
\usepackage[all]{xy}

\usepackage{amsfonts}
\usepackage{amsmath}
\usepackage{graphicx}
\usepackage{epsfig}
\usepackage{epstopdf}

\newcommand{\ba}{\begin{array}}
\newcommand{\ea}{\end{array}}
\newcommand{\be}{\begin{enumerate}}
\newcommand{\ee}{\end{enumerate}}

\newtheorem{thm}{Theorem}[section]
\newtheorem{prop}[thm]{Proposition}
\newtheorem{lemma}[thm]{Lemma}
\newtheorem{cor}[thm]{Corollary}
\newtheorem{conj}[thm]{Conjecture}

\theoremstyle{definition}
\newtheorem{defn}[thm]{Definition}

\theoremstyle{remark}

\newtheorem{rmk}[thm]{Remark}
\newtheorem{example}[thm]{Example}

\newcommand{\C}{\mathcal{C}}
\newcommand{\E}{\mathcal{E}}
\newcommand{\Ccl}{\mathcal{C}_{\infty}}
\newcommand{\Cv}{\mathcal{C}_{\infty1}}
\newcommand{\R}{\mathbb{R}}
\newcommand{\Z}{\mathbb{Z}}

\newcommand{\N}{\mathbb{N}}
\newcommand{\bdry}{\partial}
\newcommand{\Cl}{Cl_{\mathbb{Z}}}
\newcommand{\cl}{\mathcal{CL}}
\newcommand{\cv}{\mathcal{V}}
\newcommand{\cb}{\mathcal{B}}
\newcommand{\hb}{\mathcal{HB}}
\newcommand{\bs}{\backslash}
\newcommand{\Hom}{\operatorname{Hom}}
\newcommand{\n}{\noindent}
\newcommand{\F}{\mathbb{F}_2}
\newcommand{\ot}{\otimes}
\newcommand{\we}{\wedge}
\newcommand{\ra}{\rightarrow}
\newcommand{\xra}{\xrightarrow}
\newcommand{\g}{\Gamma}
\newcommand{\mf}{\mathbf}
\newcommand{\op}{\operatorname}
\newcommand{\cal}{\mathcal}
\newcommand{\es}{\emptyset}
\newcommand{\lan}{\langle}
\newcommand{\ran}{\rangle}
\newcommand{\uq}{\mathbf{U}_q(\mathfrak{sl}(1|1))}

\begin{document}
\title{A diagrammatic categorification of a Clifford algebra}

\author{Yin Tian}
\address{University of Southern California, Los Angeles, CA 90089}
\email{yintian@usc.edu}

\begin{abstract}
We give a graphical calculus for a categorification of a Clifford algebra and its Fock space representation via differential graded categories.
The categorical action is motivated by the gluing action between the {\em contact categories} of infinite strips.
\end{abstract}

\maketitle

\section{Introduction}
This paper is a sequel to \cite{Tian1} in which the author categorified a finite-dimensional Clifford algebra.
The goal of this paper is to categorify an infinite-dimensional Clifford algebra $Cl_{\Z}$ via a diagrammatic monoidal DG category $\cl$ and give a categorical action of $\cl$ on a category $DGP(R)$ which lifts the Fock space representation $V$ of $Cl_{\Z}$.
(See Section 1.1 for precise definitions of $Cl_{\Z}$ and $V$.)
The main improvement over \cite{Tian1} is that the categorification is done diagrammatically.
In particular, the category $\cl$ is strict monoidal.

\vspace{.2cm}
The Clifford algebra $Cl$ and its cousin, the Heisenberg algebra $H$, arise from the study of commutation relations in fermionic and bosonic quantum mechanics, respectively.
Here $H$ has generators $p_i, q_i$, for $i$ in some infinite set $I$ and commutator relations
$$[p_i,q_j]=p_iq_j-q_jp_i=\delta_{i,j}1, \quad [p_i,p_j]=0, \quad [q_i,q_j]=0.$$
The symmetric algebra of infinitely many variables indexed by $I$ admits a natural action of $H$.
Khovanov \cite{Kh3} provided a graphical calculus for a categorification of $H$.
More generally, Cautis and Licata \cite{CL1} presented a categorification $\cal{H}_{\g}$ of the Heisenberg algebra $H_{\g}$ associated to a finite subgroup $\g \subset SL_2(\mathbb{C})$.
In a later paper \cite{CL2}, they provided 2-representations of quantum affine algebras via complexes in the 2-representation of quantum Heisenberg algebra.
The diagrammatic approach to categorification was pioneered by Lauda in \cite{La} where he obtained a categorification of the quantum group $\mathbf{U}_q(\mathfrak{sl}_2)$ .

\vspace{.2cm}
The Clifford algebra $Cl$ is a $\mathbb{C}$-algebra with generators $\psi_j, \psi_j^*$ for $j\in \Z$ and relations:
$$\{\psi_i, \psi_j^*\}=\psi_i\psi_j^*+\psi_j^*\psi_i=\delta_{ij}, \quad \{\psi_i, \psi_j\}=0, \quad \{\psi_i^*, \psi_j^*\}=0,$$
for $i,j \in \Z$.
It is the analogue of $H$ obtained by replacing the commutation relations by the anticommutation relations.
The Clifford algebra $Cl$ acts on the exterior algebra of infinitely many variables, which is called the Fock space of free fermions.
Note that this Fock space also admits an action of the Heisenberg algebra $H$ under the boson-fermion correspondence \cite[Section 14.9]{Kac}.

The antisymmetric property of fermions in physics can be understood in the framework of differential graded categories in mathematics.
Lipshitz, Ozsv\'ath and Thurston in \cite{LOT} defined a diagrammatic differential graded algebra, called the {\em strands algebra} in the context of bordered Heegaard Floer homology.
Motivated by the strands algebra, Khovanov \cite{Kh2} gave a diagrammatic categorification of the positive part of $\mathbf{U}_q(\mathfrak{gl}(1|2))$.
A connection between Heegaard Floer homology and $3$-dimensional contact topology on the categorical level was observed by Zarev in \cite{Za}.
Based on this connection the author categorified $\mathbf{U}_t(\mathfrak{sl}(1|1))$ and its tensor product representations in \cite{Tian2}.
Our graphical calculus in this paper is another attempt to pursue these diagrammatic methods for the Clifford algebra.

\subsection{Main results}
Let $F$ be a complex vector space with a basis:
$$\{i_1\we i_2 \we \cdots ~|~ i_1>i_2>\cdots ~\mbox{are integers}, i_n=i_{n-1}-1 ~\mbox{for}~ n\gg0\}.$$
Elements in the basis are called {\em semi-infinite monomials} and $|k\ran=k \we (k-1) \we \cdots$ is called the {\em vacuum state} of {\em charge} $k$.
The generators $\psi_j$ and $\psi_j^*$ act on $F$ via {\em wedging} and {\em contracting} operators:
\begin{align*}
\psi_j(i_1 \we i_2 \we \cdots)=& \left\{
\begin{array}{ll}
0 & \quad\mbox{if} \hspace{0.3cm} j=i_k ~\mbox{for some}~ k; \\
(-1)^{k}i_1 \we \cdots \we i_k \we j \we i_{k+1} \we \cdots & \quad\mbox{if} \hspace{0.3cm} i_k > j > i_{k+1}.
\end{array}\right. \\
\psi_j^*(i_1 \we i_2 \we \cdots)=& \left\{
\begin{array}{ll}
0 & \quad \mbox{if} \hspace{0.3cm} j\neq i_k ~\mbox{for all}~ k; \\
(-1)^{k-1}i_1 \we \cdots \we i_{k-1} \we i_{k+1} \we \cdots & \quad \mbox{if} \hspace{0.3cm} j= i_k.
\end{array}\right. \\
\end{align*}
The vector space $F$ is called the Fock space representation of $Cl$.
We refer to \cite[Section 14.9]{Kac} for more details.

For the purpose of categorification, we define the integral Clifford algebra $\Cl$ as a $\Z$-algebra with generators $a_i$ for $i \in \Z$ and relations:
\begin{gather*}
a_i^2=0;\\
a_ia_j=-a_j a_i \hspace{.2cm} \mbox{if} \hspace{.2cm} |i-j|>1;\\
a_i a_{i+1} + a_{i+1}a_i=1.
\end{gather*}
There is a $\mathbb{C}$-algebra homomorphism $\theta$:
$$\begin{array}{cccc}
\theta: & Cl_{\Z} \ot_{\Z}  \mathbb{C} & \ra & Cl \\
 & a_{2i} & \mapsto & \psi_i^* \\
 & a_{2i-1} & \mapsto & \psi_i+\psi_{i-1},
\end{array}$$
which is an isomorphism between suitable completions of the algebras.
Let $V$ be a free abelian group whose basis is the set of semi-infinite monomials.
Then $V$ is a representation of $Cl_{\Z}$ which is an integral version of the pullback of the representation $F$ of $Cl$ under the homomorphism $\theta$.
The goal of this paper is to categorify $Cl_{\Z}$ and $V$ via DG categories.

Let $H(\cal{A})$ denote the $0$th homology category of a DG category $\cal{A}$.
Let $K_0(\cal{A})$ denote the Grothendieck group of $H(\cal{A})$ if $H(\cal{A})$ is triangulated.
Our main results are the followings.

\begin{thm} \label{thmcl}
There exists a DG category $\cl$ such that $H(\cl)$ is triangulated.
There is a monoidal functor $\ot: \cl \times \cl \ra \cl$ whose decategorification $K_0(\ot): K_0(\cl) \times K_0(\cl) \ra K_0(\cl)$ makes $K_0(\cl)$ isomorphic to the Clifford algebra $Cl_{\Z}$.
\end{thm}

\begin{thm} \label{thmv}
There exists a DG category $DGP(R)$ generated by some projective DG modules over a DG algebra $R$ such that $H(DGP(R))$ is triangulated.
Moreover, there is a categorical action
$\cal{F}: \cl \times DGP(R) \ra DGP(R)$ whose decategorification $K_0(\cal{F}): K_0(\cl) \times K_0(DGP(R)) \ra K_0(DGP(R))$ makes $K_0(DGP(R))$ isomorphic to the representation $V$ of $Cl_{\Z}$.
\end{thm}

\begin{rmk}
The main improvement over \cite{Tian1} is the monoidal functor on $\cl$ which is easily defined diagrammatically.
\end{rmk}

\subsection{Motivation from the contact category}
The motivation is from the {\em contact category} introduced by Honda \cite{Honda1} which presents an algebraic way to study contact topology in dimension $3$.
The contact category $\C(\Sigma, F)$ of $(\Sigma, F)$ is an additive category associated to a compact oriented surface $\Sigma$ with a subset of marked points $F$ of $\bdry\Sigma$.
The objects of $\C(\Sigma, F)$ are isotopy classes of {\em dividing sets} on $\Sigma$ with a homotopy grading whose restrictions to $\bdry\Sigma$ agree with $F$.
The morphisms are $\F$-vector spaces spanned by isotopy classes of {\em tight} contact structures $[\xi]$ on $\Sigma \times [0,1]$.
More precisely, a dividing set on $\Sigma$ is a properly embedded 1-manifold, possibly disconnected, which divides $\Sigma$ into positive and negative regions.
Any dividing set with a contractible component is defined as the zero object since there is no tight contact structure in a neighborhood of the dividing set by a criterion of Giroux \cite{Gi}.
As basic blocks of morphisms, {\em bypass attachments} introduced by Honda \cite{Honda2} give elementary pieces of contact structures on $\Sigma \times [0,1]$.
The effect of a bypass attachment locally changes dividing sets as in Figure \ref{in0}.
We refer to \cite{Tian1} for a more detailed introduction to the contact category.
\begin{figure}[h]
\begin{overpic}
[scale=0.25]{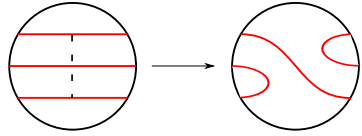}
\end{overpic}
\caption{A basic morphism in the contact category given by a bypass attachment.}
\label{in0}
\end{figure}

We are interested in the contact category $\C_{\infty}$ of an infinite strip with infinitely many marked points indexed by $\Z$ on both boundary components.
A monoidal structure on $\Ccl$ is given by horizontally stacking two dividing sets along their common boundaries for the objects and sideways stacking two contact structures for the morphisms.
There is a collection of distinguished dividing sets $[\es]$ and $[i]$'s for $i \in \Z$ as in Figure \ref{in1}.
Note that indices in $\Z$ on the left boundary of the strip is increasing from bottom to top.
This assignment is different from the one in \cite{Tian1}.
A key feature of $\Ccl$ is the existence of distinguished triangles given by three bypass attachments $[i+1]\cdot[i] \ra [\es] \ra [i]\cdot[i+1] \ra [i+1]\cdot[i]$.
These triangles descend to relations $a_{i}a_{i+1}+a_{i+1}a_{i}=1$ in the Clifford algebra $Cl_{\Z}$.
Our diagrammatic DG category $\cl$ can be viewed as a DG realization of the contact category $\Ccl$.
\begin{figure}[h]
\begin{overpic}
[scale=0.25]{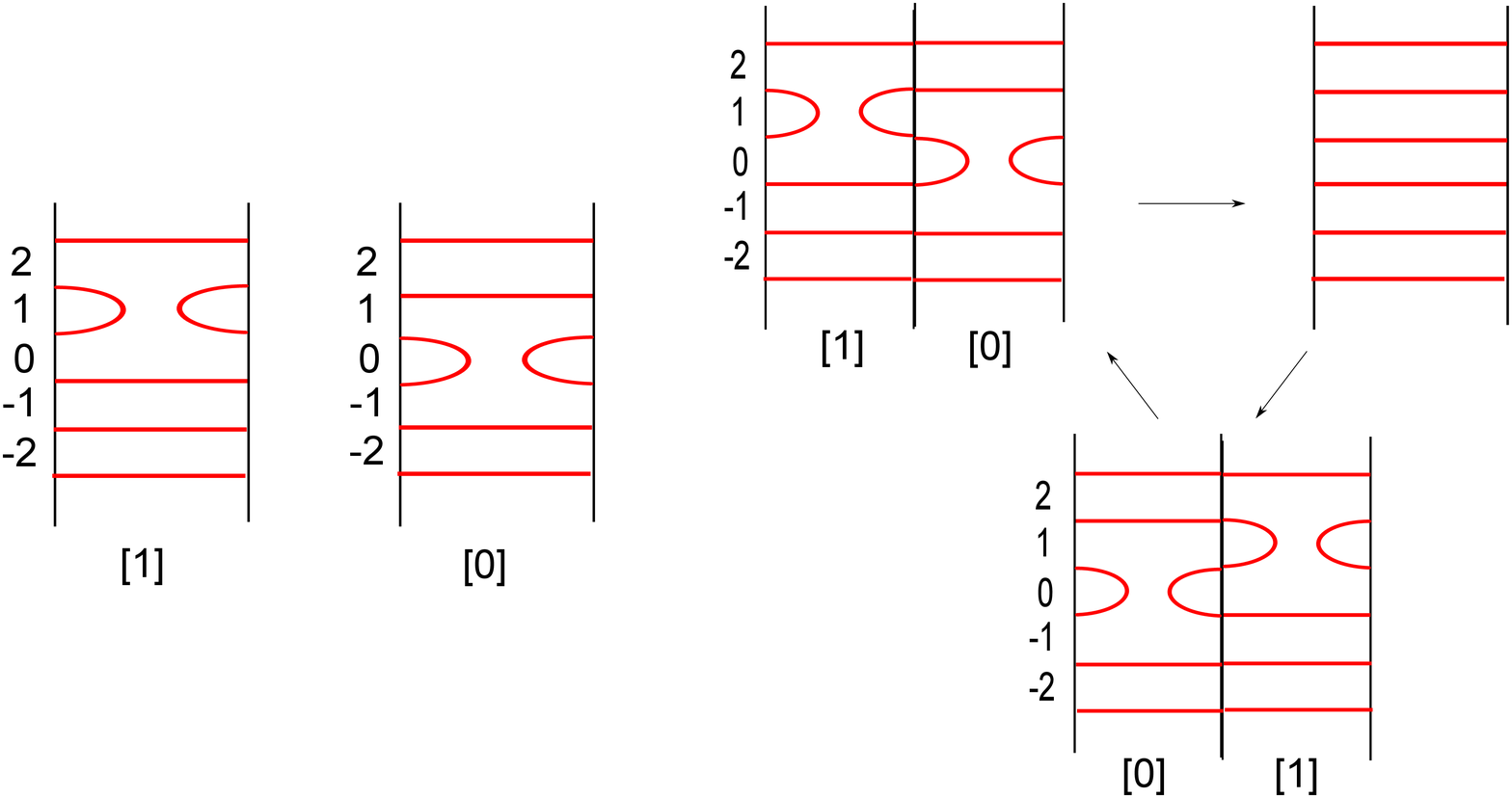}
\put(92,28){$[\es]$}
\end{overpic}
\caption{Distinguished objects $[0],[1]$ are on the left; a distinguished triangle in $\Ccl$ is on the right.}
\label{in1}
\end{figure}

To categorify the Fock space representation $V$, consider the contact category $\Cv$ of an infinite strip with infinitely many marked points indexed by $\Z$ on the left boundary and only one marked point on the right.
The semi-infinite monomials in the basis of $V$ can be lifted to distinguished objects of $\Cv$ as in Figure \ref{in2}.
For instance, each factor $i$ for $i \leq 0$ in a semi-infinite monomial $|0\ran=0 \we -1 \we \cdots$ corresponds to an arc $L_i$ on the infinite strip which connects the label $2i$ on the left boundary to the right boundary.
The collection of $L_i$'s for $i \leq 0$ determines a dividing set $\g_{|0\ran}$ of $\Cv$ whose negative region retracts
onto the union of $L_i$'s.
Under this lifting procedure, the {\em charge} of a semi-infinite monomial is related to the {\em Euler number} of the corresponding dividing set.
The action of $\Cl$ on $V$ is lifted to a categorical action of $\Ccl$ on $\Cv$ given by horizontally stacking morphisms $[\xi_0]$ of $\Ccl$ with morphisms $[\xi_1]$ of $\Cv$, where $[\xi_0]$ is to the left and $[\xi_1]$ is to the right.
See Figure \ref{in2}.

In general, the contact category is intimately related to bordered Heegaard Floer homology as well as to the {\em wrapped Fukaya category} interpretation of Heegaard Floer homology due to Auroux \cite{Au}.
The semi-infinite monomials in $V$ correspond to wedge products of arcs which are Lagrangians of an infinite symmetric product of the strip.

\begin{figure}[h]
\begin{overpic}
[scale=0.20]{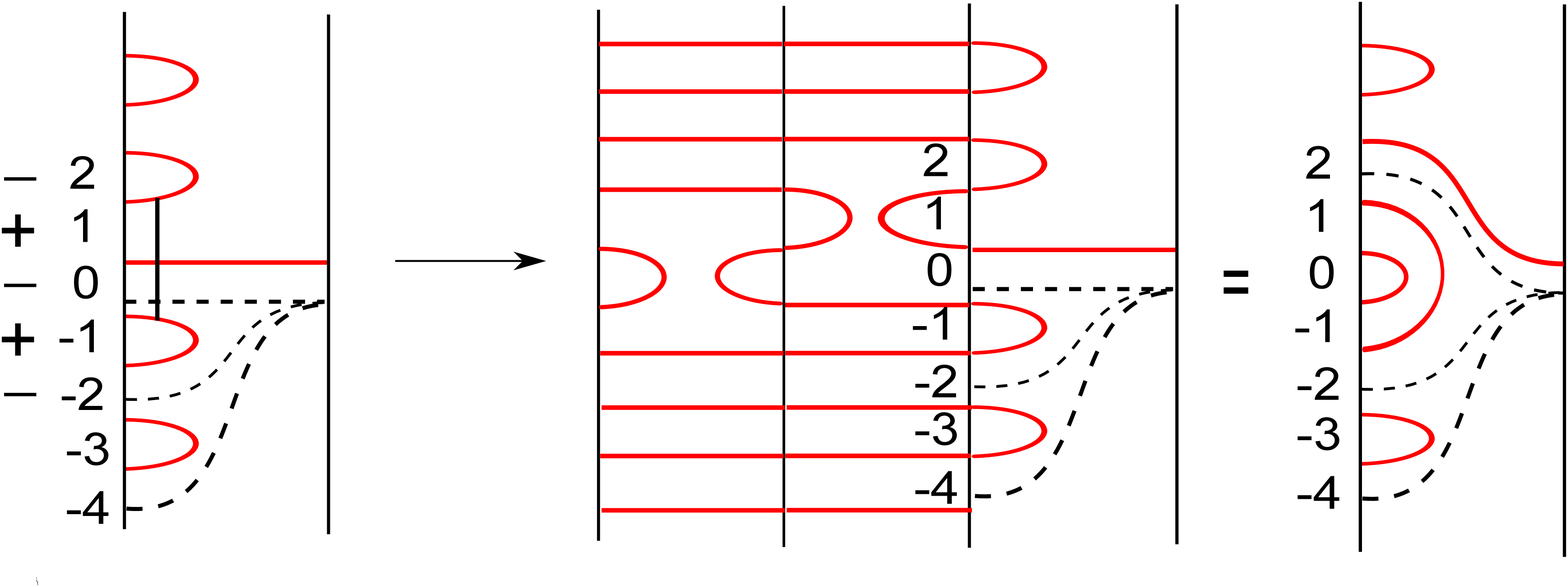}
\put(12,18){${\scriptstyle L_0}$}
\put(10,12){${\scriptstyle L_{-1}}$}
\put(10,5){${\scriptstyle L_{-2}}$}
\put(0,0){$|0\ran=0 \we -1 \we -2 \we \cdots$}
\put(68,0){$|0\ran$}
\put(44,0){$[0]$}
\put(55,0){$[1]$}
\put(82,0){$1 \we -1 \we -2 \we \cdots$}
\end{overpic}
\caption{Semi-monomials in $V$ are lifted to objects of $\Cv$, where dashed arcs are Lagrangians and the vertical solid arc denotes the bypass from $|0\ran$ to $[0][1]|0\ran$.}
\label{in2}
\end{figure}

\subsection{Organization of the paper}
\begin{itemize}
\item In Section 2 we define the diagrammatic DG category $\cl$ and give a surjective homomorphism $\Cl \twoheadrightarrow K_0(\cl)$.
\item In Section 3 we review the Fock space representation $V$ of $Cl_{\Z}$ and show that it is faithful.
\item In Section 4 we define a diagrammatic DG category $\cv$ aiming to categorify $V$.
The graphical calculus for $\cv$ is much easier than that for $\cl$ so that it enable us to give bases of morphism sets of $\cv$.
\item In Section 5 we give an algebraic formulation $DGP(R)$ of $\cv$ in terms of projective DG $R$-modules, where $R$ is a DG algebra generated by diagrams in $\cv$.
    We show that the categories $\cv$ and $DGP(R)$ are equivalent and both Grothendieck groups are isomorphic to $V$.
\item In Section 6 we construct a categorical action of $\cl$ on $DGP(R)$ via DG $R$-bimodules.
The graphical calculus in $\cv$ is the main tool to check that the relations in $\cl$ are preserved under the action.
\end{itemize}

\vspace{.2cm}
\noindent {\bf Acknowledgements:}
I am very grateful to Ko Honda for many ideas and suggestions and introducing me to the contact category.
I would like to thank Mikhail Khovanov for suggesting this project and many illuminating conversations.
I would also like to thank Aaron Lauda for suggestions to look for representations of the Clifford algebra, and Anthony Licata for pointing out the faithfulness of the Fock space representation.

\section{The diagrammatic DG category $\cl$}
In this section we define the diagrammatic monoidal DG category $\cl$ in two steps:
\be
\item define the elementary objects and morphisms between them;
\item enlarge to {\em one-sided twisted complexes} of elementary objects.
\ee
(1), carried out in Section 2.1, is diagrammatic and gives essential ingredients of $\cl$.
(2), carried out in Section 2.2, is the analogue of passing from modules to complexes of modules.
In Section 2.3 we discuss the Grothendieck group $K_0(\cl)$ and prove the following.

\begin{prop}\label{K0clsur}
There exists a monoidal DG category $\cl$ such that $H(\cl)$ is triangulated and $K_0(\cl)$ is a quotient of $Cl_{\Z}$ under a surjective algebra homomorphism $\gamma: \Cl \twoheadrightarrow K_0(\cl)$.
\end{prop}
We will show that the map $\gamma$ is actually an isomorphism via the categorical action of $\cl$ in Section 6.

\subsection{Elementary objects and morphisms}
In this subsection we define the elementary objects of $\cl$ and the morphism sets as cochain complexes of $\F$-vector spaces.
We fix $\F$ as the ground field for morphisms throughout the paper to simplify computations.

\vspace{.2cm}
\n $\bullet$ {\bf Elementary objects:} the set of elementary objects $\E(\cl)$ of $\cl$ consists of sequences of integers $\{\mf{a}=(a_1, \dots, a_n) ~|~ a_i \in \Z, ~n \in \Z_{\geq 0}\}$.
The empty sequence corresponding to $n=0$ will be denoted by $(\es)$.

\vspace{.2cm}
\n $\bullet$ {\bf Morphisms:} the morphism set $\Hom(\mf{a}, \mf{b})$ for $\mf{a}, \mf{b} \in \E(\cl)$ is a cochain complex of $\F$-vector spaces generated by a set $D(\mf{a},\mf{b})$ of dotted labeled planar diagrams from the label $\mf{a}$ to the label $\mf{b}$, modulo local relations $\cal{L}$.
The composition of morphisms is given by stacking diagrams vertically.
A vertical stacking of two diagrams is defined to be zero if the labels at their endpoints do not match.

\vspace{.2cm}
\n $\bullet$ {\bf $D(\mf{a},\mf{b})$:}
for $\mf{a}=(a_1, \dots, a_n)$ and $\mf{b}=(b_1, \dots, b_m)$ in $\E(\cl)$, any diagram $f \in D(\mf{a},\mf{b})$ is obtained by vertically stacking finitely many {\em generating diagrams} in the strip $\R \times [0,1]$ such that endpoints of $f$ are $\{1, \dots, n\} \times \{0\}$ with labels $a_i$ for $1\leq i \leq n$ and $\{1, \dots, m\} \times \{1\}$ with labels $b_j$ for $1\leq j \leq m$.
All diagrams are read from bottom to top as morphisms.
Each generating diagram is a horizontal stacking of one {\em elementary diagram} with some trivial vertical strands.

\vspace{.2cm}
\n $\bullet$ {\bf Elementary diagrams:}
the elementary diagrams consist of $5$ types as shown in Figure \ref{1}:
\be
\item a vertical strand $id_i \in \Hom((i),(i))$;
\item a cup $cup_{i,i+1} \in \Hom((\es), (i,i+1))$;
\item a cap $cap_{i+1,i} \in \Hom((i+1,i), (\es))$;
\item a crossing $cr_{i,j} \in \Hom((j,i), (i,j))$;
\item a dotted strand $dot_{e} \in \Hom((e+2), (e))$ or $dot_{e+1} \in \Hom((e-1), (e+1))$, where $e$ denotes an even integer throughout the paper.
\ee
\begin{figure}[h]
\begin{overpic}
[scale=0.27]{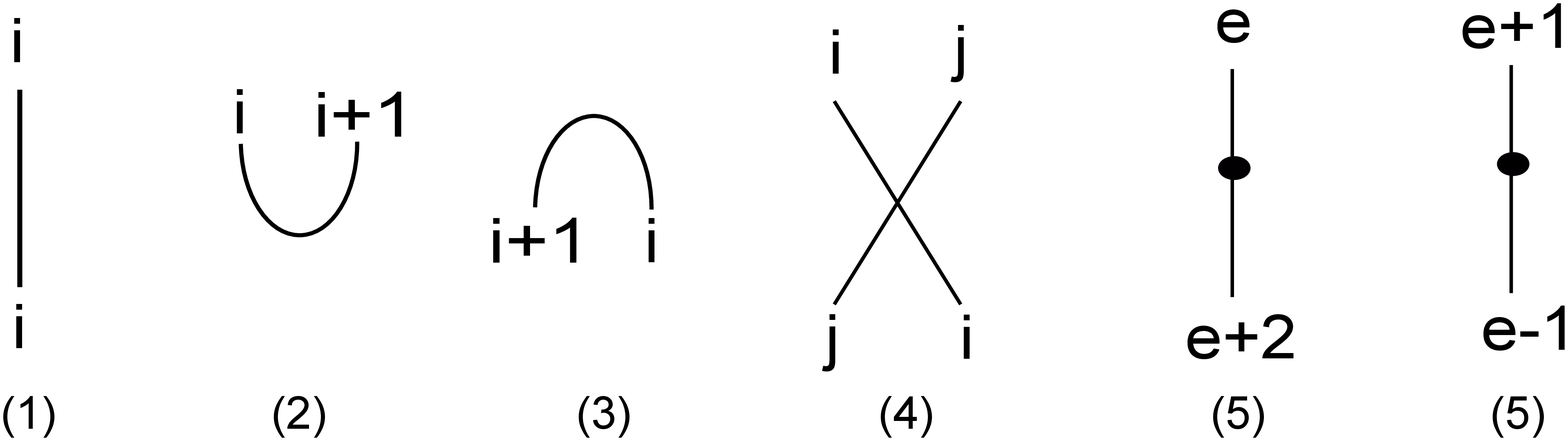}
\end{overpic}
\caption{The elementary diagrams of $\cl$}
\label{1}
\end{figure}

\vspace{.2cm}
\n $\bullet$ {\bf Local relations:}
the relations $\cal{R}$ consist of 6 groups:

\vspace{.1cm}
\n {\bf (R1)} Isotopy relation:
\begin{description}
\item[(R1-a)] vertical strands as idempotents;
\item[(R1-b)] isotopy of a single strand;
\item[(R1-c)] isotopy of a crossing;
\item[(R1-d)] isotopy of a dot through a cup or cap, where the isotopy class is denoted by a dotted cup $dup_{e+1,e} \in \Hom((\es),(e+1,e))$ or a dotted cap $dap_{e-1,e} \in \Hom((e-1,e),(\es))$ for simplicity;
\item[(R1-e)] isotopy of disjoint diagrams.
\end{description}
\begin{figure}[h]
\begin{overpic}
[scale=0.22]{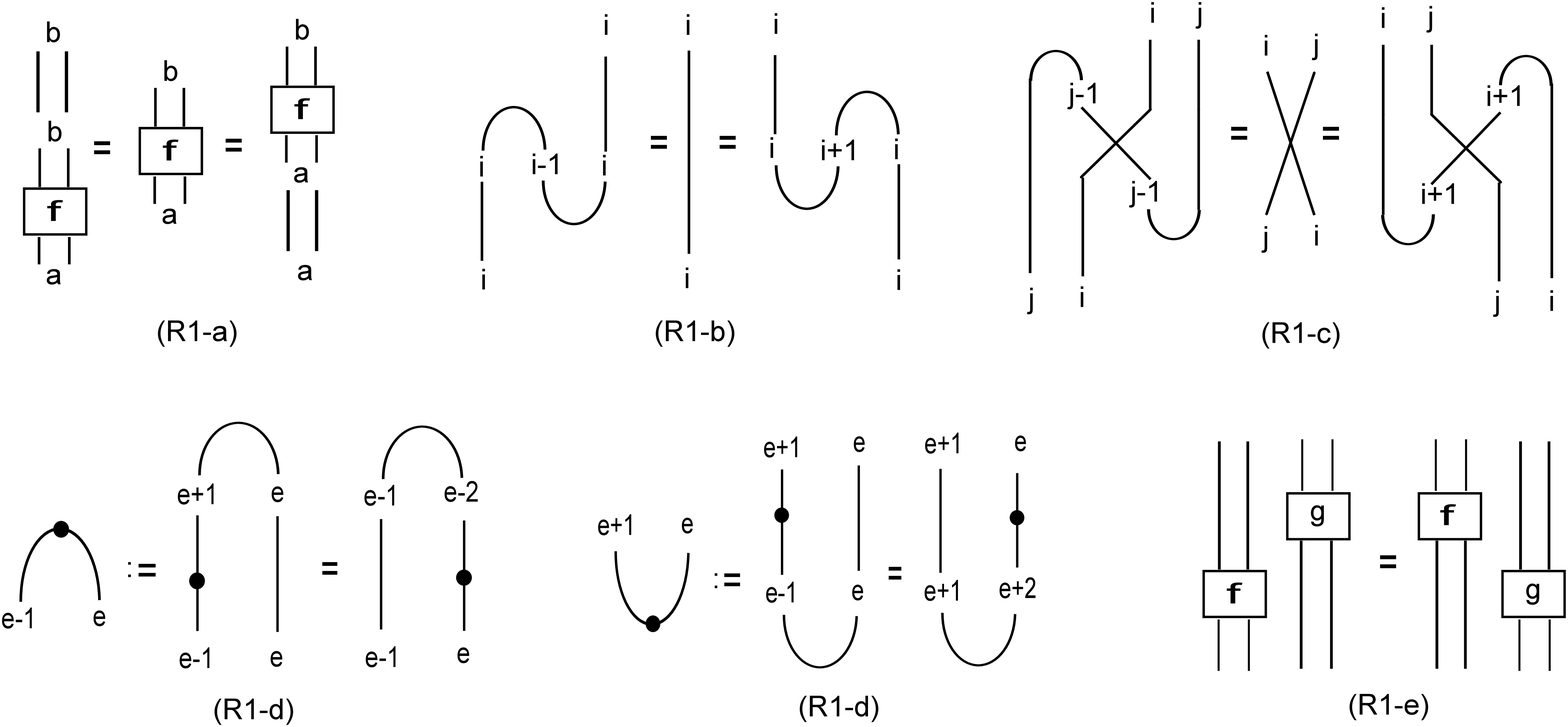}
\end{overpic}
\caption{Isotopies}
\label{2}
\end{figure}

\vspace{.1cm}
\n {\bf (R2)} Double dot relation: a strand with a double dot is zero.

\vspace{.1cm}
\n {\bf (R3)} Loop relation: a loop with one dot is the empty diagram, i.e., the identity $id_{\es} \in \Hom((\es),(\es))$.
\begin{figure}[h]
\begin{overpic}
[scale=0.22]{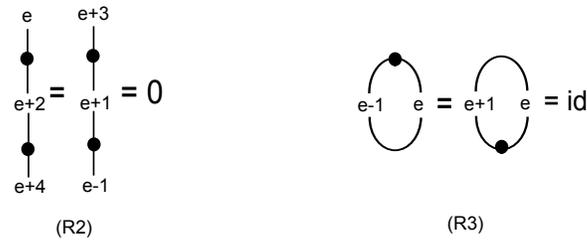}
\end{overpic}
\caption{Double dots and loops}
\label{3}
\end{figure}

\vspace{.1cm}
\n {\bf (R4)} Double crossing relation:
\begin{description}
\item[(R4-a)] a double crossing in $\Hom((e,e+1),(e,e+1))$ or $\Hom((e,e-1),(e,e-1))$ is the identity.
\item[(R4-b)] a double crossing in $\Hom((e+1,e),(e+1,e))$ or $\Hom((e-1,e),(e-1,e))$ is a sum of two terms.
\item[(R4-c)] a double crossing in $\Hom((i,i),(i,i))$ is zero.
\item[(R4-d)] a double crossing in $\Hom((i,j),(i,j))$ is the identity if $|i-j|>1$.
\end{description}

\n {\bf (R5)} Triple intersection moves for all labels:
\begin{figure}[h]
\begin{overpic}
[scale=0.2]{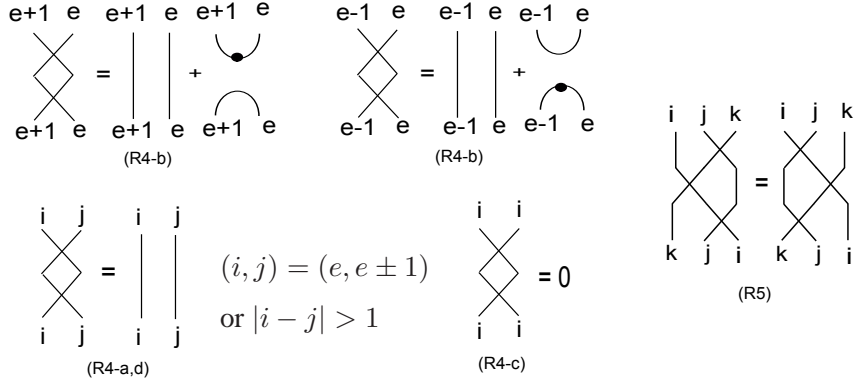}
\put(25,12){$(i,j)=(e,e\pm1)$}
\put(25,6){or $|i-j|>1$}
\end{overpic}
\caption{Double crossing and triple intersection moves}
\label{4}
\end{figure}

\n {\bf (R6)} Dot slide through a crossing.
\begin{figure}[h]
\begin{overpic}
[scale=0.22]{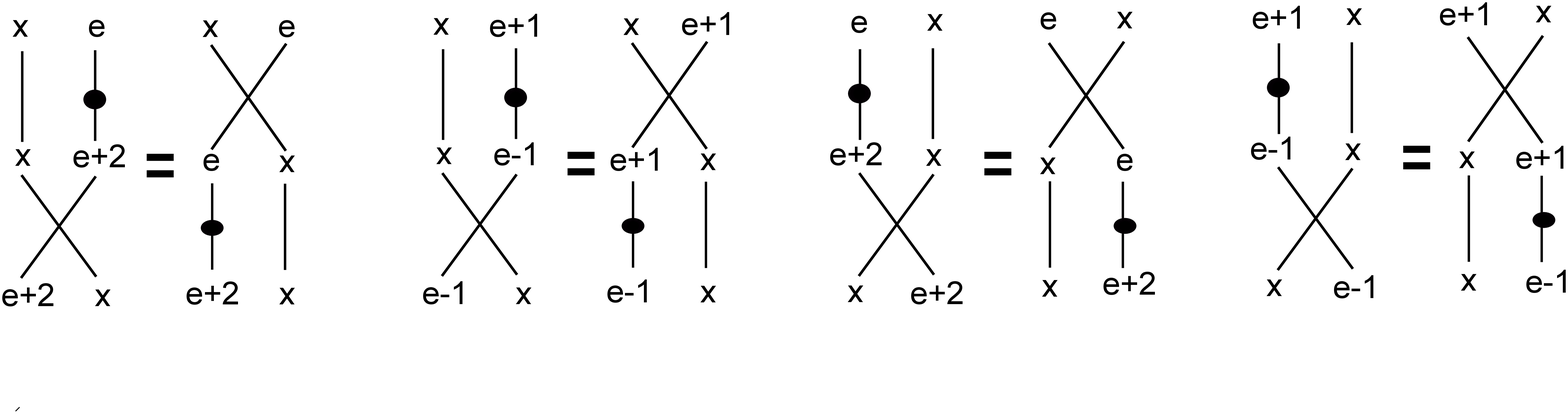}
\put(30,0){$x \neq e,e-1$}
\put(55,0){$x \neq e+1,e+2$}
\put(3,0){$x \neq e,e+1$}
\put(83,0){$x \neq e,e+1$}
\end{overpic}
\caption{Dot slide}
\label{5}
\end{figure}

\n $\bullet$ {\bf Monoidal structure:}
The monoidal functor on the elementary objects and morphisms is given by the horizontal stacking.
Let $\mf{a}\ot \mf{b}$ denote a horizontal stacking of $\mf{a}, \mf{b} \in \E(\cv)$.

\vspace{.2cm}
\n $\bullet$ {\bf Cohomological grading:}
A cohomological grading $\op{gr}$ is defined on the elementary diagrams as zero except for crossings $cr_{i,j}$:
$$ \op{gr}(cr_{i,j})= \left\{
\begin{array}{cl}
(-1)^{i-j} & \mbox{if} \hspace{0.3cm} i<j; \\
(-1)^{i-j+1} & \mbox{if} \hspace{0.3cm} i\geq j.
\end{array}\right.
$$
The grading is extended to general diagrams additively with respect to the vertical and horizontal stacking.
Two sides of any relation have the same grading.
Hence $\op{gr}$ defines a grading on the morphism sets.
Let $\Hom(\mf{a},\mf{b})=\oplus_i\Hom^i(\mf{a},\mf{b})$ be the decomposition according to the grading.

\begin{rmk}
\n(1) Morphisms are generated by diagrams up to isotopy relative to boundary by (R1).

\n(2) The object $(i)$ is left adjoint to $(i-1)$ and right adjoint to $(i+1)$ by (R1-a), (R1-b) and (R1-e):
\begin{gather*}
\op{Hom}((i)\ot\mf{a}, \mf{b})\cong \op{Hom}(\mf{a},(i-1)\ot\mf{b}),\\
\op{Hom}(\mf{a}\ot(i), \mf{b})\cong \op{Hom}(\mf{a},\mf{b}\ot(i+1)).
\end{gather*}

\n(3) A dotted strand either increases odd labels or decreases even labels by $2$.

\n(4) Double crossing relations (R4-a,b) depend on the parity of labels.
There are similar relations in the graphical calculus for the Heisenberg algebra \cite{Kh3}.
The main difference from there is that the double crossing in $\Hom((i,i),(i,i))$ is zero in our case.

\n(5) It is helpful to check the grading for (R6).
For instance, two crossings $cr_{x,e+2}$ and $cr_{x,e}$ in the first relation have the same grading if and only if $x \neq e, e+1$.
\end{rmk}

\n $\bullet$ {\bf More relations:} we deduce more relations which will be useful later from the defining relations.
\begin{lemma} \label{curl}
The following relations hold in $\cl$:
\begin{figure}[h]
\begin{overpic}
[scale=0.2]{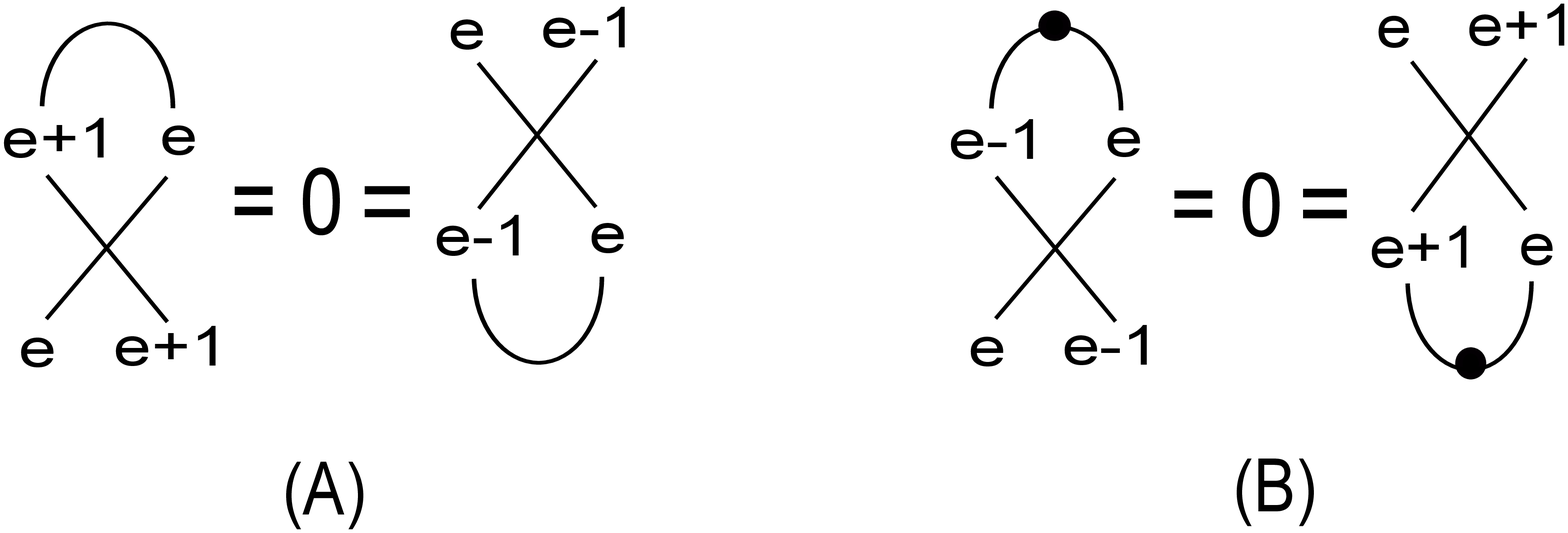}
\end{overpic}
\caption{More relations}
\label{7}
\end{figure}
\end{lemma}

\begin{proof}
The proof of the first equation in (A) is given in Figure \ref{8} and uses (R3) and (R4-a,b).
\begin{figure}[h]
\begin{overpic}
[scale=0.20]{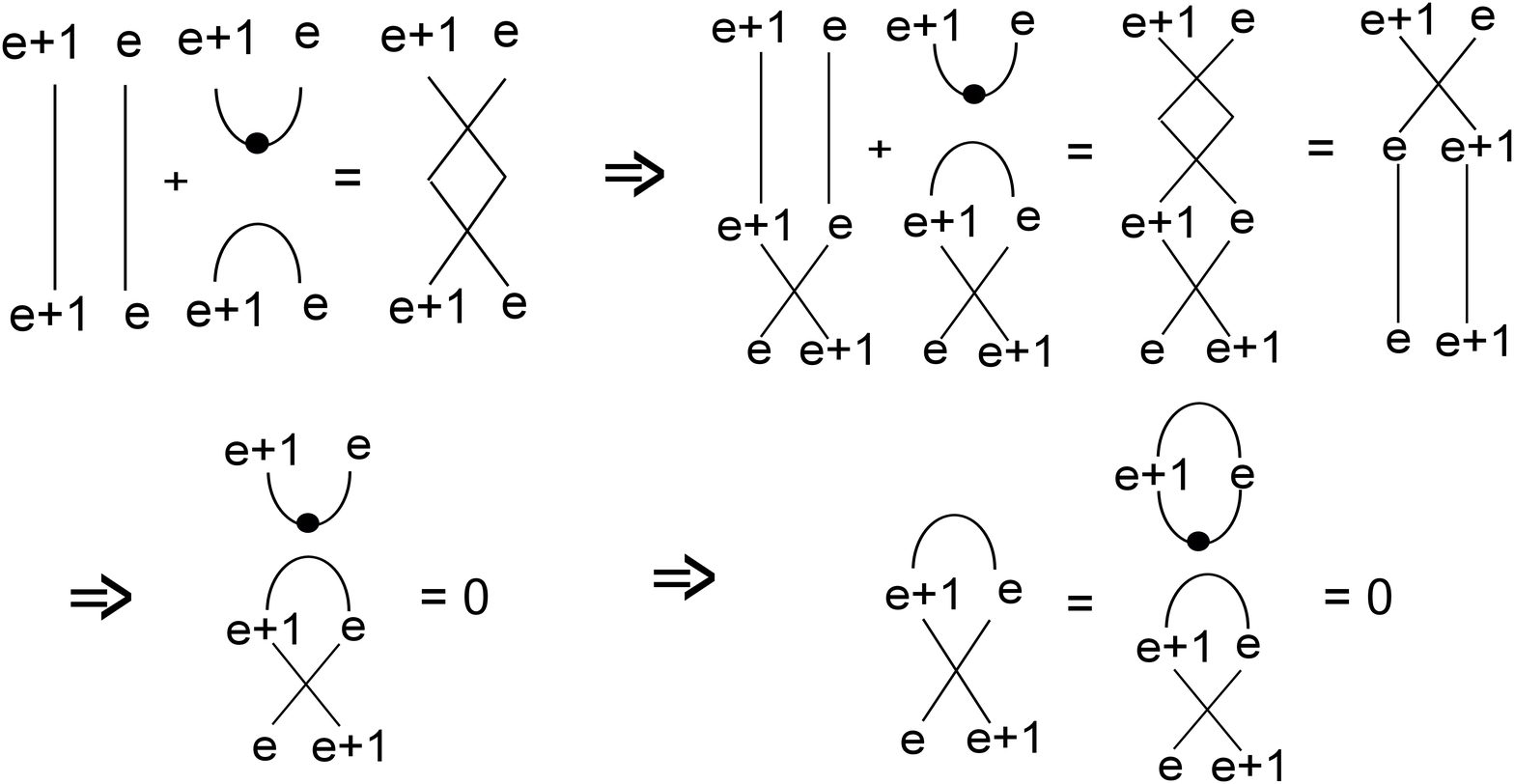}
\end{overpic}
\caption{}
\label{8}
\end{figure}

The verification of the other relations is left to the reader.
\end{proof}

\n $\bullet$ {\bf Differential:}
A differential $d$ is defined on the elementary diagrams in Figure \ref{6} and extended to general diagrams by Leibniz's rule with respect to vertical and horizontal stackings.
It is easy to see that $d$ increases the grading by $1$.
Note that the differential of a crossing $cr_{i,j}$ depends on $i-j$.
\begin{figure}[h]
\begin{overpic}
[scale=0.2]{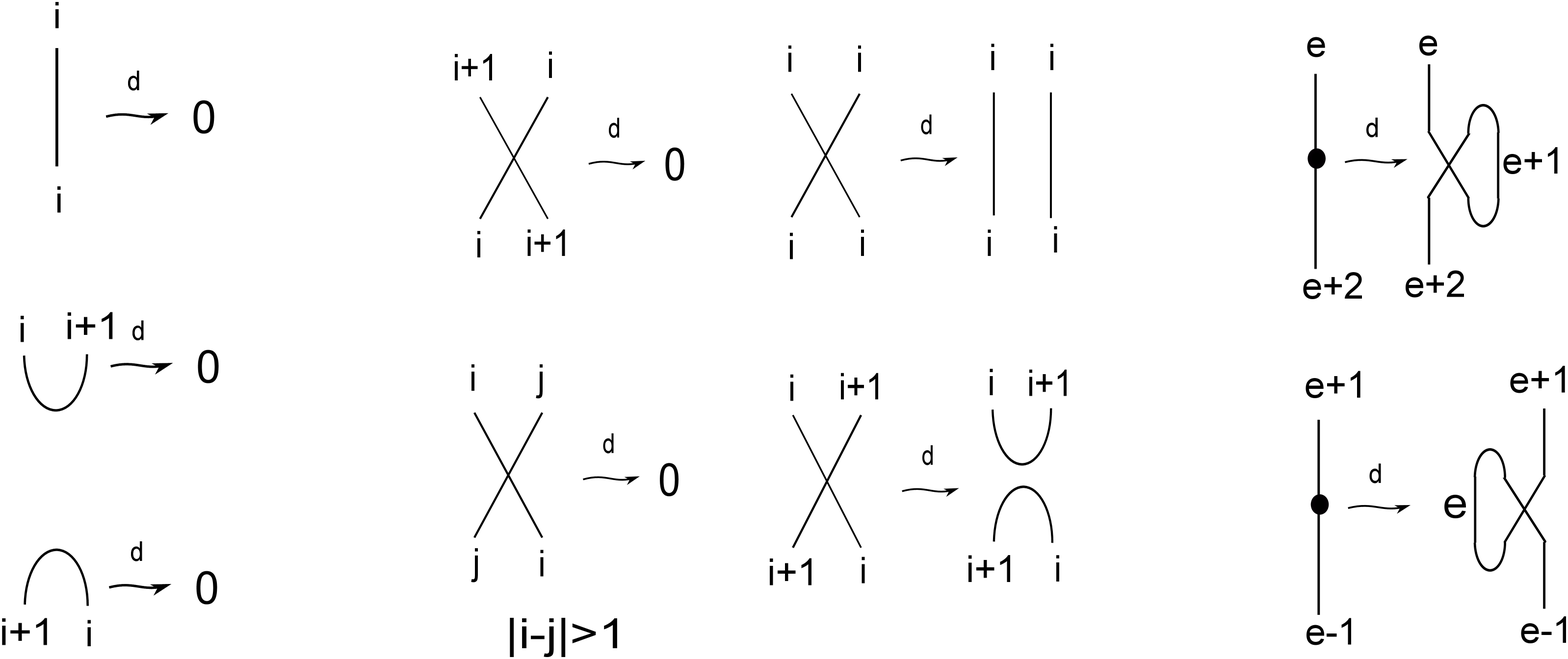}
\end{overpic}
\caption{The differential on the elementary diagrams}
\label{6}
\end{figure}

\begin{lemma}
The relations $\cal{R}$ are preserved under $d$.
\end{lemma}
\begin{proof}
We prove the lemma for (R1-c,d), (R2) and (R4-b) in Figure \ref{9} and leave the others to the reader.
\begin{figure}[h]
\begin{overpic}
[scale=0.2]{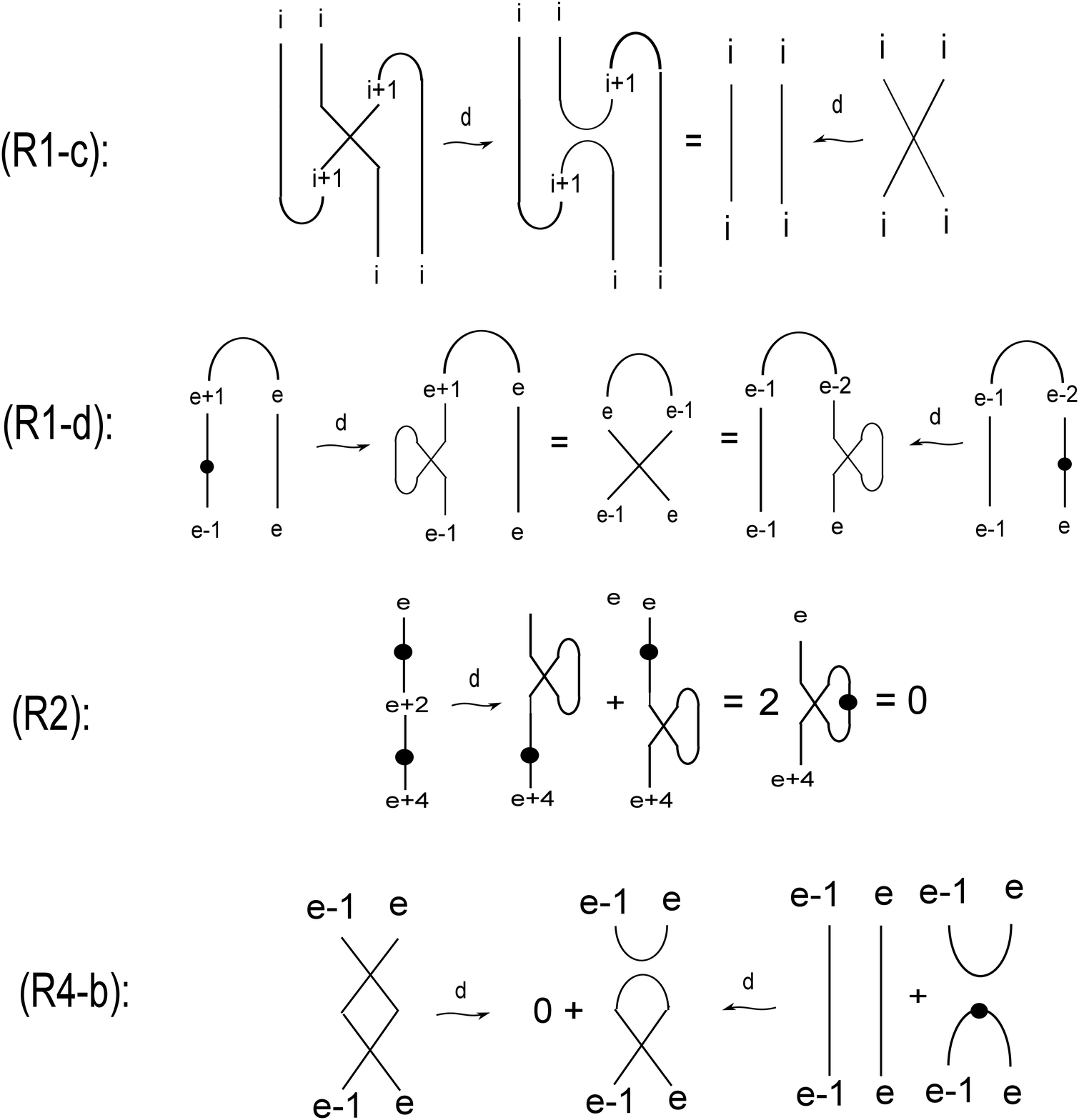}
\end{overpic}
\caption{Relations under differential}
\label{9}
\end{figure}
\end{proof}

\begin{lemma}
$d^2=0$.
\end{lemma}
\begin{proof}
This follows from $d^2=0$ on the elementary diagrams.
\end{proof}

\begin{rmk}
It is not clear to the author how to compute $\F$-bases for any $\Hom(\mf{a},\mf{b})$.
For instance, it is hard to unlink a family of linked circles in $\Hom ((\es),(\es))$ since we can perform a dot slide only under certain conditions.
In particular, we do not know if $\Hom ((\es),(\es))$ is a $1$-dimensional $\F$-vector space generated by $id_{\es}$.
We will show that $\cl$ does not collapse by constructing a nontrivial categorical action in Section 6.
\end{rmk}

\subsection{Definition of $\cl$}
We enlarge the elementary objects to {\em one-sided twisted complexes} of elementary objects following \cite{BK}.
Let $\mf{a}[m]$ denote an elementary object $\mf{a}$ with a cohomological grading shifted by $m$,
i.e.,
$$\Hom^i(\mf{a}[m],\mf{b})\cong \Hom^{i-m}(\mf{a},\mf{b}), \quad \Hom^i(\mf{a},\mf{b}[m])\cong \Hom^{i+m}(\mf{a},\mf{b}).$$
The objects of the DG category $\cl$ are {\em one-sided twisted complexes} of elementary objects:
$$\{(\{\mf{a}^i[m_i]\}_{i=1}^{n}, f=\sum_{i < j}f^j_i) ~|~ f^j_i \in \Hom^1(\mf{a}^i[m_i],\mf{a}^j[m_j]), \sum_k (f^j_k \circ f^k_i) + d(f^j_i)=0 ~\mbox{for all}~ i,j\}.$$

\begin{rmk}
The term ``twisted" refers to the condition as a deformation of $\sum_k (f^j_k \circ f^k_i)=0$ in ordinary complexes.
The term ``one-sided" refers to the condition $f^j_i=0$ for $i\geq j$.
The notion of {\em one-sided twisted complex} was introduced in \cite[Section 4, Definition 1]{BK}.
\end{rmk}

The morphism set $\Hom((\{\mf{a}^i[m_i]\}, f), (\{\mf{b}^j[n_j]\}, g))$ is an $\F$-vector space spanned by $\{h_i^j \in \Hom(\mf{a}^i[m_i], \mf{b}^j[n_j])\}$ with a differential $d_{cx}$:
$$d_{cx}(h_i^j)=\sum_{i'} (h_i^j \circ f_{i'}^i) + \sum_{j'} (g^{j'}_j \circ h^j_i) + d(h_i^j).$$

\begin{example}
An object of $\cl$ is given by a complex of three terms twisted by an additional morphism $f_1^3$:
\begin{gather*}
(\{(0,0)[2], (0,0)[1], (0,0)\}, f=f_1^2+f_2^3+f_1^3), \\
f_1^2=id_{(0,0)} \in \Hom^1((0,0)[2], (0,0)[1]), \\
f_2^3=id_{(0,0)} \in \Hom^1((0,0)[1], (0,0)), \\
f_1^3=cr_{0,0} \in \Hom^1((0,0)[2], (0,0)), \\
\mbox{such that}~ f_2^3 \circ f_1^2 + d(f_1^3) = id_{(0,0)} \circ id_{(0,0)} + d(cr_{0,0}) =0.
\end{gather*}
$$\xymatrix{
(0,0) \ar[r]^{f_1^2} \ar@/_1pc/[rr]_{f_1^3} & (0,0) \ar[r]^{f_2^3} & (0,0)
}
$$
\end{example}

The monoidal structure on the elementary objects extends to one-sided twisted complexes in the usual way:
$$\begin{array}{cccccc}
\ot: & \cl & \times & \cl & \ra & \cl \\
& (\{\mf{a}^i[m_i]\}, f) &,& (\{\mf{b}^j[n_j]\}, g) & \mapsto & (\{\mf{a}^i \ot \mf{b}^j[m_i+n_j]\}, f\ot id + id \ot g).
\end{array}$$

We finally have a strict monoidal DG category $\cl$.

\subsection{The Grothendieck group $K_0(\cl)$}
We refer to \cite{Ke} for an introduction to DG categories and their homology categories.
Let $H(\cl)$ denote the $0$th homology category of $\cl$, which is a triangulated category by \cite[Section 4, Proposition 2]{BK}.
Let $K_0(\cl)$ denote the Grothendieck group of $H(\cl)$.
The induced monoidal functor $\ot: H(\cl) \times H(\cl) \ra H(\cl)$ is bi-exact and hence descends to a multiplication $K_0(\ot): K_0(\cl) \times K_0(\cl) \ra K_0(\cl)$.

Let $[\mf{a}]$ and $[i] \in K_0(\cl)$ denote the classes of elementary objects $\mf{a}$ and $(i) \in \cl$ for $i \in \Z$.
Let $[i][j]$ denote the multiplication of $[i]$ and $[j]$.
Then $[\mf{a}]=[a_1]\cdots[a_n] \in K_0(\cl)$ for $\mf{x}=(a_1,\dots, a_n)$.
Hence, $K_0(\cl)$ is a $\Z$-algebra generated by $\{[i] ~|~ i \in \Z\}$ with unit $[\es]$.

\begin{lemma} \label{K0cl}
The following isomorphisms hold in the triangulated category $H(\cl)$:
\be
\item $(i,i) \simeq 0$.
\item $(e+1,e) \simeq \{(\es) \xra{cup_{e,e+1}} (e,e+1)\}$, where the term $(\es)$ is in cohomological grading $0$.
\item $(e-1,e) \simeq \{(e,e-1) \xra{cap_{e,e-1}} (\es)\}$, where the term $(\es)$ is in cohomological grading $0$.
\item $(j,i) \simeq (i,j)[\op{gr}(cr_{i,j})]$ if $|i-j|>1$.
\ee
\end{lemma}
\begin{proof}
\n(1) The identity $id_{(i,i)}=d(cr_{i,i}) \in \Hom_{\cl}((i,i),(i,i))$ implies that
$$id_{(i,i)}=0 \in \Hom_{H(\cl)}((i,i),(i,i)).$$
Hence, the zero morphisms give isomorphisms $(i,i) \simeq 0$ in $H(\cl)$.

\vspace{.1cm}
\n(2) We construct two morphisms $f \in \Hom((e+1,e), ~\{(\es) \xra{cup_{e,e+1}} (e,e+1)\})$ and $g \in \Hom(\{(\es) \xra{cup_{e,e+1}} (e,e+1)\}, ~(e+1,e))$.
The two compositions of $f$ and $g$ are given in Figure \ref{10}.
It is easy to check that $df=dg=0$.
Hence $f$ and $g$ represent morphisms in $H(\cl)$.
The compositions are identities by (R4-a,b) and Lemma \ref{curl}.
\begin{figure}[h]
\begin{overpic}
[scale=0.25]{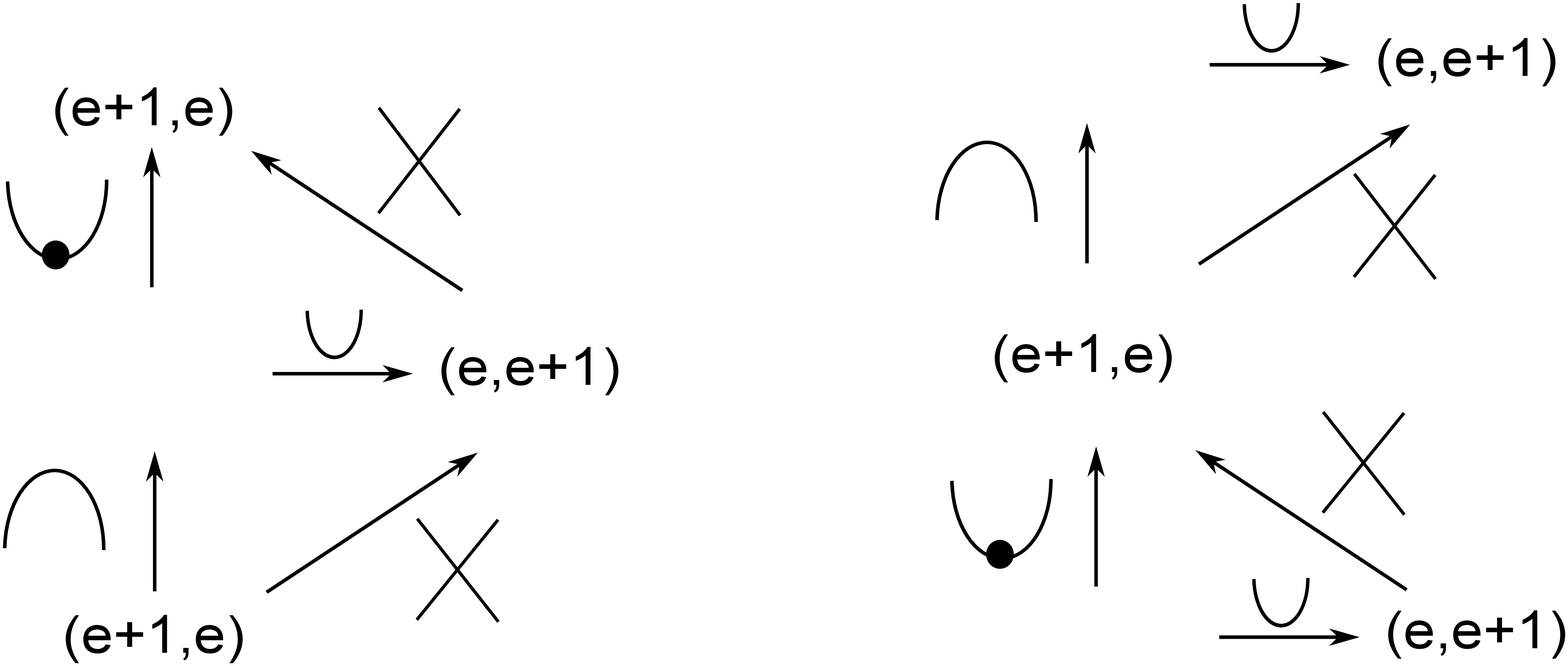}
\put(8,17){$(\es)$}
\put(68,0){$(\es)$}
\put(66,38){$(\es)$}
\end{overpic}
\caption{Isomorphism between $(e+1,e)$ and $\{(\es) \xra{cup_{e,e+1}} (e,e+1)\}$}
\label{10}
\end{figure}

\vspace{.1cm}
\n(3) Similar to (2).

\vspace{.1cm}
\n(4) The crossings $cr_{i,j}$ and $cr_{j,i}$ between $(j,i)$ and $(i,j)[\op{gr}(cr_{i,j})]$ are isomorphisms by (R4-d).
\end{proof}

\begin{cor} \label{relation}
The following relations hold in $K_0(\cl)$:
\begin{gather*}
[i]^2=0;\\
[i][j]=-[j][i] \hspace{.2cm} \mbox{if} \hspace{.2cm} |i-j|>1;\\
[i][i+1]+[i+1][i]=[\es]=1.
\end{gather*}
\end{cor}

\begin{proof}[Proof of Proposition \ref{K0clsur}]
Corollary \ref{relation} gives a ring homomorphism:
$$\begin{array}{cccc}
\gamma: & \Cl & \ra & K_0(\cl) \\
& a_i & \mapsto & [i],
\end{array}$$
which is surjective.
\end{proof}
We will show that $\gamma$ is an isomorphism in Section 6.

\subsection{Connection to contact topology}
The diagrammatic category $\cl$ can be viewed as a DG realization of the contact category $\Ccl$ which was described in Section 1.2.

Elementary objects $(\es)$ and $(i)$'s of $\cl$ correspond to the distinguished dividing sets $[\es]$ and $[i]$'s of $\Ccl$ for $i \in \Z$ as in Figure \ref{in1}.
The correspondence on elementary objects in general is built up on the monoidal structures on $\cl$ and $\Ccl$.

Each elementary diagram which is $d$-closed (i.e. all five diagrams in the first two columns of Figure \ref{6}) corresponds to a tight contact structure between the corresponding dividing sets.
Other elementary diagrams which are not $d$-closed give DG realizations of morphisms of $\Ccl$.
For instance,
\be
\item $d(cr_{i,i})=id_{(i,i)}$ implies that $(i,i)$ is isomorphic to the zero object of $H(\cl)$ as in Lemma \ref{K0cl}(1); the corresponding dividing set $[i][i]$ is defined as the zero object of $\Ccl$.
\item The isomorphisms which contain crossings $cr_{i,i+1}$ and dotted diagrams in Lemma \ref{K0cl}(2),(3) are DG realizations of the distinguished triangles in $\Ccl$ as in Figure \ref{in2}.
\ee

Some of relations in $\cl$, for instance (R1-b), (R4-d) and (R-5), represent isotopies of tight contact structures in $\Ccl$.

Although we do not know how to compute $\F$-bases for any $\Hom_{\cl}(\mf{a},\mf{b})$, we may get some hints about the cohomology of $\Hom_{\cl}(\mf{a},\mf{b})$ from the contact topology.
By a theorem of Eliashberg \cite{E} on the uniqueness of tight contact structure on the $3$-dimensional ball, $\op{dim}_{\F}\Hom_{\Ccl}(\g_1, \g_2)\leq 1$ for any dividing sets $\g_1$ and $\g_2$.
A dividing set $\g_2$ is called {\em stackable} over $\g_1$ if $\Hom_{\Ccl}(\g_1, \g_2)$ is nonzero.
Let $\overline{\mf{a}}$ denote the dividing set in $\Ccl$ corresponding to $\mf{a} \in \E(\cl)$.
We conclude this section with the following conjecture:

\begin{conj} \label{conjcl}
For $\mf{a}, \mf{b} \in \E(\cl)$,
$$ \op{dim}_{\F}H^*(\Hom_{\cl}(\mf{a},\mf{b}))= \left\{
\begin{array}{cl}
1 & \mbox{if}~~~ \overline{\mf{b}} ~~\mbox{is stackable over}~~ \overline{\mf{a}}; \\
0 & \mbox{otherwise}.
\end{array}\right.
$$
\end{conj}

\section{The linear representation $V$ of $Cl_{\Z}$}
In this section we define the Fock space representation $V$ of $Cl_{\Z}$ and show that it is faithful.
The basis of $V$ will be used to construct another diagrammatic category $\cv$ in Section 4.

The representation $V$ of $Cl_{\Z}$ is the pullback of the Fock space representation $F$ of $Cl$ under the homomorphism $\theta$:
$$\begin{array}{cccc}
\theta: & \Cl \ot \mathbb{C} & \ra & Cl \\
 & a_{2i} & \mapsto & \psi_i^* \\
 & a_{2i-1} & \mapsto & \psi_i+\psi_{i-1},
\end{array}$$
In suitable completions of the algebras, $\theta$ has an inverse map which involves infinite sums:
$$\begin{array}{ccc}
Cl & \ra & \Cl \ot \mathbb{C} \\
 \psi_i^* & \mapsto & a_{2i} \\
 \psi_i & \mapsto & a_{2i-1}-a_{2i-3}+\cdots.
\end{array}$$

We give a $\Z$-basis of $V$ which is equivalent to an integral version of the $\mathbb{C}$-basis of $F$ as follows.
Let $V$ be a free abelian group with a basis of semi-infinite increasing monomials of even integers:
$$\hb = \{x_1 \we x_2 \we \cdots ~|~ x_1 < x_2 < \cdots ~\mbox{are even integers}, x_n=x_{n-1}+2 ~\mbox{for}~ n\gg0\}.$$
There is a decomposition $\hb = \sqcup_{k} \hb_{k}$ for $k \in \Z$, where
$$\hb_k = \{x_1 \we x_2 \we \cdots \in \cb ~|~ x_n=2n+2k ~\mbox{for}~ n\gg0\}.$$
Let $\overline{|k\ran}=(2k+2) \we (2k+4) \we (2k+6) \cdots \in \hb_{k}$ be the {\em vacuum state} in $\hb_k$.
For $j\in \Z$, let $\eta_j$ and $\eta_j^*$ be the {\em wedging} and {\em contracting} operators on $V$:
\begin{align*}
\eta_j(x_1 \we x_2 \we \cdots)=& \left\{
\begin{array}{ll}
0 & \mbox{if} \hspace{0.3cm} 2j=x_i ~\mbox{for some}~ i; \\
(-1)^{i}x_1 \we \cdots \we x_i \we (2j) \we x_{i+1} \we \cdots & \mbox{if} \hspace{0.3cm} x_i < 2j < x_{i+1}.
\end{array}\right. \\
\eta_j^*(x_1 \we x_2 \we \cdots)=& \left\{
\begin{array}{ll}
0 & \quad \mbox{if} \hspace{0.3cm} 2j\neq x_i ~\mbox{for all}~ i; \\
(-1)^{i-1}x_1 \we \cdots \we x_{i-1} \we x_{i+1} \we \cdots & \quad \mbox{if} \hspace{0.3cm} 2j= x_i.
\end{array}\right. \\
\end{align*}

Now we define the action $G: \Cl \times V \ra V$ by:
$$a_{2j}:=\eta_j, \quad\quad a_{2j-1}:=\eta_j^*+\eta_{j-1}^*.$$
It is easy to verify that this is well-defined. For instance,
$$a_{2j}a_{2j-1}+a_{2j-1}a_{2j}=\eta_j(\eta_j^*+\eta_{j-1}^*)+(\eta_j^*+\eta_{j-1}^*)\eta_j=1.$$

\begin{rmk}
The correspondence between the representation $V$ and the Fock space $F$ is simply given by $V \ni \overline{|k\ran} \longleftrightarrow |k\ran \in F$, where $\overline{|k\ran}$ is the complement of $|k\ran$ multiplied by $2$.
\end{rmk}

Since any element in $\hb$ can be obtained from some vacuum state $\overline{|k\ran}$ via wedging operators, we have the following easy description of $V$ which will be used in Section 6.
\begin{lemma} \label{linear}
The representation $V$ of $Cl_{\Z}$ is uniquely determined by:
\begin{align*}
a_{2j}\cdot v =\eta_j(v) \quad\quad & \mbox{for}~~ v \in \hb, \\
a_{2j-1}\cdot \overline{|k\ran}=0 \quad\quad & \mbox{for}~~ k \geq j.
\end{align*}
\end{lemma}

Moreover, in order to show that $\gamma: \Cl \ra K_0(\cl)$ is injective we need the following lemma.
\begin{lemma} \label{faithful}
The representation $V$ of $Cl_{\Z}$ is faithful.
\end{lemma}
\begin{proof}
We first reduce the infinite-dimensional case to a finite-dimensional one.
Suppose that there exists $a \in \Cl$ such that $a\cdot v =0$ for all $v \in V$.
Let $\Cl(n)$ be a subalgebra of $Cl_{\Z}$ generated by $\{a_i ~|~ 1 \leq i \leq 2n\}$.
We may assume $a \in \Cl(n)$ for some $n$ without loss of generality since $a$ is a finite sum and $Cl_{\Z}$ has a $\Z$-translation.
Consider a $\Cl(n)$-module
$$V(n)=\{a\cdot \overline{|n\ran} \in V ~|~ a \in \Cl(n)\},$$
where $a_{2i-1}(\overline{|n\ran})=0$ for $1 \leq i \leq n$.
It suffices to show that $V(n)$ is faithful as a $\Cl(n)$-module.

Let $Cl_n$ be a $\Z$-algebra generated by $\eta_i, \eta_i^*$ for $1 \leq i \leq n$ with relations:
$$\{\eta_i, \eta_j^*\}=\delta_{ij}, \quad \{\eta_i, \eta_j\}=0, \quad \{\eta_i^*, \eta_j^*\}=0.$$
The $\Z$-algebras $\Cl(n)$ and $Cl_n$ are isomorphic via the following map:
$$\begin{array}{ccl}
\Cl(n) & \ra & Cl_n \\
 a_{2i} & \mapsto & \eta_i \\
 a_{2i-1} & \mapsto & \eta_i^*+\eta_{i-1}^* \quad \mbox{for}~~ i>1 \\
 a_{1} & \mapsto & \eta_1^*.
\end{array}$$
Let $V_n=\we ^* W_n$ be the exterior algebra of $W_n$, where $W_n$ is a free abelian group with a basis $\{e_i ~|~ 1 \leq i \leq n\}$.
Then $V_n$ is a representation of $Cl_n$, where $\eta_i$ and $\eta_i^*$ are wedging and contracting with $e_i$.
Moreover, $V_n$ is isomorphic to $V(n)$ under the isomorphism between $Cl_n$ and $\Cl(n)$.
It suffices to show that $V_n$ is faithful as a $Cl_n$-module.

We show that $Cl_n \cong \op{End}_{\Z}(V_n)$ as $\Z$-algebras by induction on $n > 0$.
It is true for $n=1$.
Suppose that $Cl_{n-1} \cong \op{End}_{\Z}(V_{n-1})$.
Consider a direct sum of abelian groups $V_n=V_{n-1} \oplus V_{n-1}'$, where $V_{n-1}$ is identified with a subspace of $V_n$ consisting of exterior products of $\{e_i ~|~ 1 \leq i \leq n-1\}$, and $V_{n-1}'=\{e_n \we v ~|~ v \in V_{n-1}\}$.
Then we have isomorphisms of $\Z$-algebras:
\[Cl_n \cong \left(
\begin{array}{cc}
Cl_{n-1}\cdot \eta_n^*\eta_n & Cl_{n-1}\cdot \eta_n^*\\
Cl_{n-1}\cdot \eta_n & Cl_{n-1}\cdot \eta_n\eta_n^*
\end{array}
\right)
\cong \left(
\begin{array}{cc}
\op{End}(V_{n-1}) & \Hom(V_{n-1}',V_{n-1})\\
\Hom(V_{n-1},V_{n-1}') & \op{End}(V_{n-1}')
\end{array}
\right).\]
Hence, $Cl_n \cong \op{End}_{\Z}(V_n)$ and $V_n$ is faithful.
\end{proof}

\section{The diagrammatic DG category $\cal{V}$}

To construct a categorical action of $\cl$ which lifts the linear action of $Cl_{\Z}$ on $V$, we need categorify $V$ first.
Recall from Section 3 that a basis of $V$ is given by semi-infinite monomials of even integers.
We mimic the construction of $\cl$ to define another diagrammatic category $\cv$, whose elementary objects are semi-infinite sequences of even integers.

In Section 4.1, we define the elementary objects and morphism sets as $\F$-vector spaces.
In Section 4.2, we define a differential on morphisms which is motivated from the differential in $\cl$.
In Section 4.3, we use define {\em canonical forms} of nonzero morphisms and give $\F$-bases of morphism sets generated by those canonical forms.
In Section 4.4, we define the DG category $\cv$ as one-sided twisted complexes of elementary objects.

\subsection{The elementary objects and morphisms}
In this subsection, we define the elementary objects and morphism sets as $\F$-vector spaces.

\vspace{.2cm}
\n $\bullet$ {\bf Elementary objects:} the set of elementary objects $\E(\cv)$ of $\cv$ consists of semi-infinite sequences of even integers:
$$\{\mf{x}=(x_1, x_2, \dots) ~|~ x_i ~\mbox{even}, ~ x_n=x_{n-1}+2 ~\mbox{for}~ n\gg0 \}.$$

\n $\bullet$ {\bf Morphisms:} the morphism set $\Hom_{\cv}(\mf{x}, \mf{y})$ for $\mf{x},\mf{y} \in \E(\cv)$ is a complex of $\F$-vector spaces generated by a set $D(\mf{x},\mf{y})$ of dotted labeled diagrams from the label $\mf{x}$ to the label $\mf{y}$, modulo local relations.
The differential will be defined in Section 4.2.

\vspace{.2cm}
\n $\bullet$ {\bf $D(\mf{x},\mf{y})$:}
For $\mf{x}=(x_1, x_2, \dots)$ and $\mf{y}=(y_1, y_2, \dots)$ in $\E(\cv)$, any diagram $f \in D(\mf{x},\mf{y})$ is obtained by vertically stacking finitely many {\em generating diagrams} in the strip $\R \times [0,1]$ such that endpoints of $f$ are located at $\{i\} \times \{0\}$ with labels $x_i$ and at $\{j\} \times \{1\}$ with labels $y_j$ for $i,j \in \Z_+$.
Each generating diagram is a horizontal stacking of one {\em elementary diagram} with infinitely many trivial vertical strands.
In particular, for any $f \in D(\mf{x},\mf{y})$ there exists $N(f)\gg0$ such that $f$ consists of trivial vertical strands in the region $\{x \in \R ~|~ x\geq N(f)\} \times [0,1]$.

\vspace{.2cm}
\n $\bullet$ {\bf Elementary diagrams:} the elementary diagrams are given in Figure \ref{v1}, where all labels are even throughout this section:
\be
\item a vertical strand in $\Hom((i),(i))$;
\item a crossing $cr_{i,j} \in \Hom((j,i), (i,j))$;
\item a dotted strand $dot_{i} \in \Hom((i+2), (i))$.
\ee
\begin{figure}[h]
\begin{overpic}
[scale=0.2]{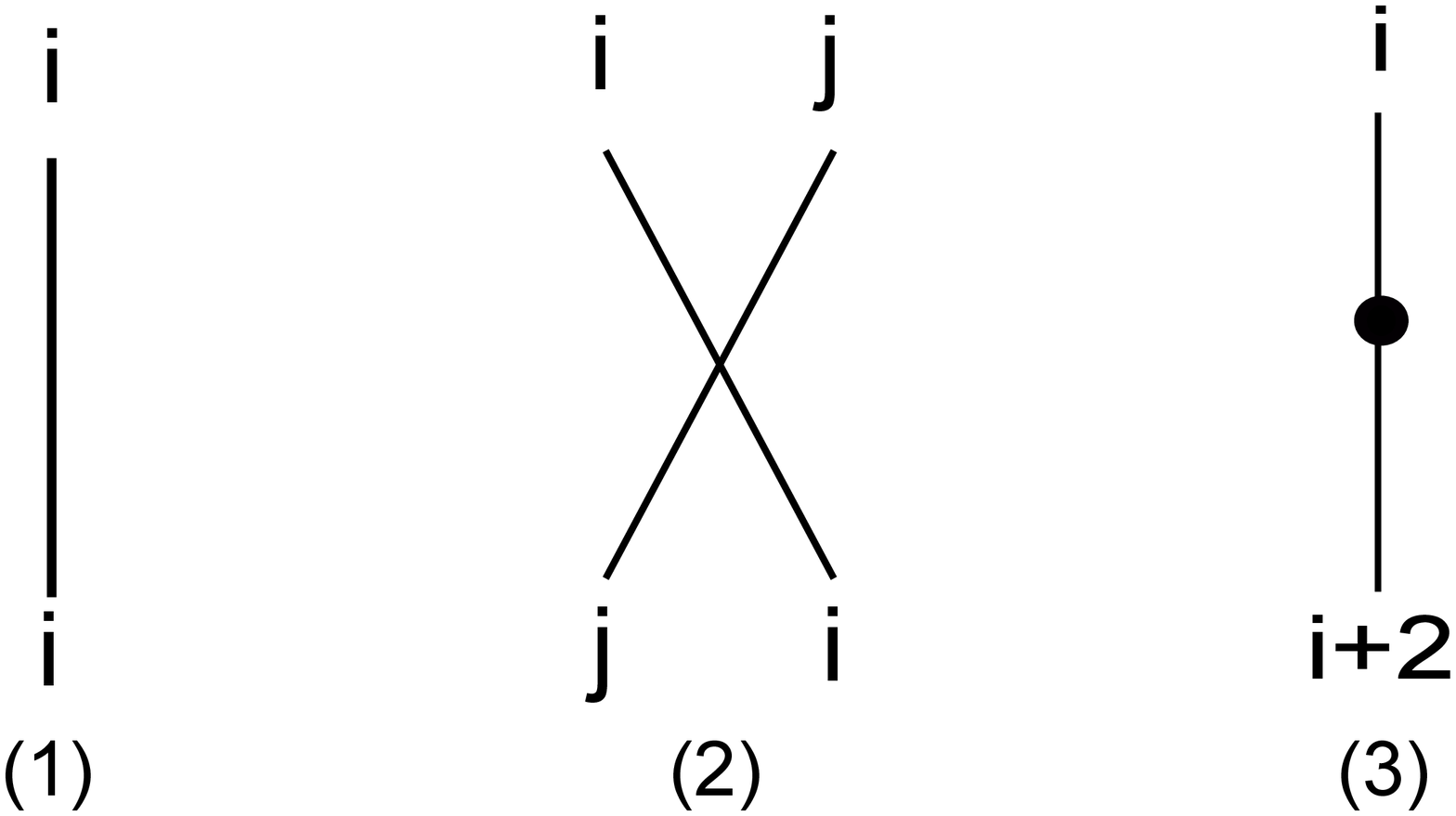}
\end{overpic}
\caption{Elementary diagrams of $\cv$}
\label{v1}
\end{figure}

\vspace{.2cm}
\n $\bullet$ {\bf Generating diagrams:}
let $\mf{a}$ denote a finite sequence of even integers, and $\mf{a} \cdot \mf{x} \in \E(\cv)$ denote a horizontal stacking of $\mf{a}$ and $\mf{x} \in \E(\cv)$.
The generating diagrams are given in Figure \ref{v4}:
\be
\item an identity $id_{\mf{x}} \in \Hom_{\cv}(\mf{x}, \mf{x})$;
\item a crossing $cr_{i,j}(\mf{a}, \mf{x}) \in \Hom_{\cv}(\mf{a}\cdot (j,i)\cdot \mf{x}, ~\mf{a}\cdot (i,j)\cdot \mf{x})$;
\item a dotted strand $dot_{i}(\mf{a}, \mf{x}) \in \Hom_{\cv}(\mf{a}\cdot (i+2)\cdot \mf{x}, ~\mf{a}\cdot (i)\cdot \mf{x})$.
\ee
\begin{figure}[h]
\begin{overpic}
[scale=0.25]{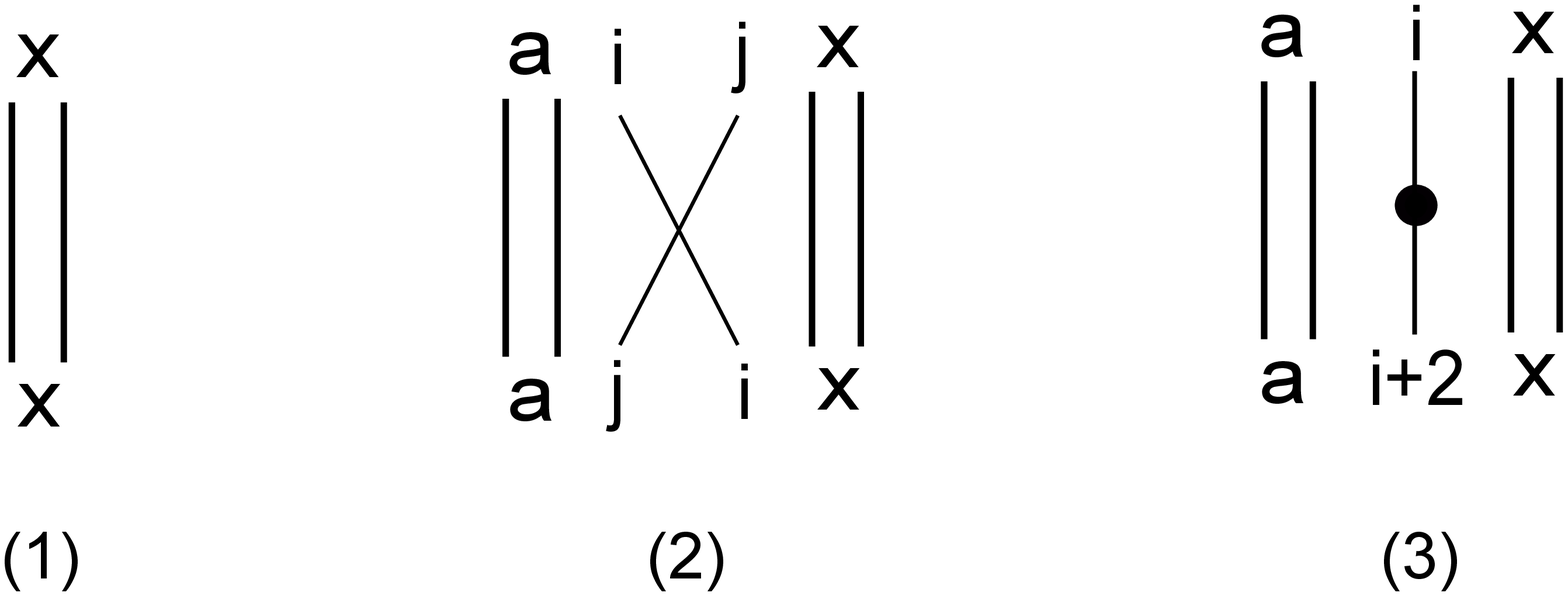}
\end{overpic}
\caption{Generating diagrams}
\label{v4}
\end{figure}

\n $\bullet$ {\bf Local relations:} the relations $\cal{L}$ consist of $5$ groups:

\vspace{.1cm}
\n {\bf (L1)} Isotopy relation:
\begin{description}
\item[(L1-a)] a vertical strand as an idempotent;
\item[(L1-b)] isotopy of disjoint diagrams.
\end{description}

\n {\bf (L2)} Double dot relation: a strand with a double dot is zero.

\vspace{.1cm}
\n {\bf (L3)} Double crossing relation:
\begin{description}
\item[(L3-a)] a double crossing in $\Hom((i,i),(i,i))$ is zero;
\item[(L3-b)] a double crossing in $\Hom((i,j),(i,j))$ is the identity if $i \neq j$.
\end{description}

\n {\bf (L4)} Triple intersection moves for all labels:

\vspace{.1cm}
\n {\bf (L5)} Dot slide relations.
\begin{figure}[h]
\begin{overpic}
[scale=0.2]{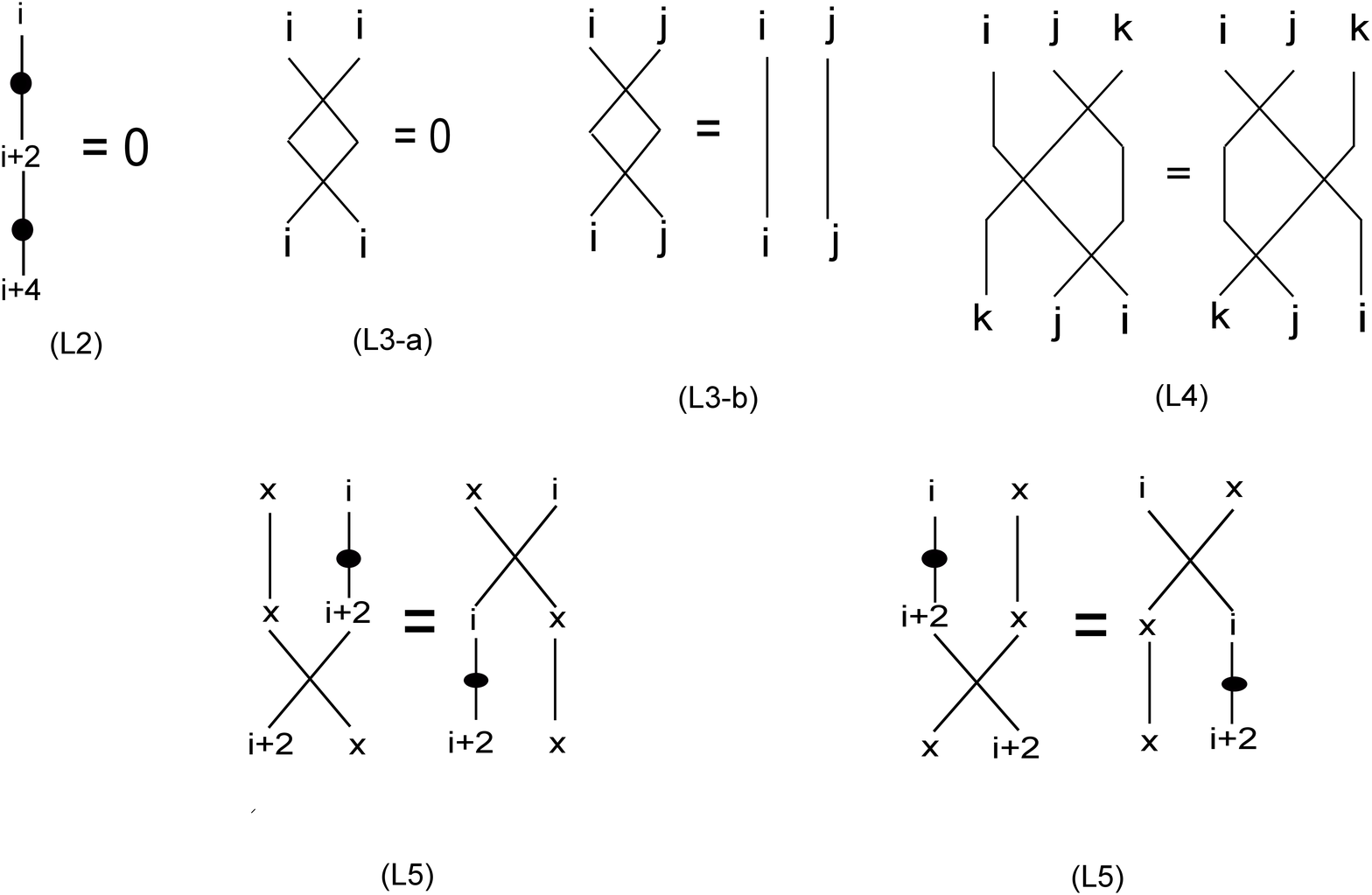}
\put(48,40){$i \neq j$}
\put(25,5){$x \neq i$}
\put(70,5){$x \neq i+2$}
\end{overpic}
\caption{Relations of $\cv$}
\label{v2}
\end{figure}

\vspace{.2cm}
\n $\bullet$ {\bf Grading on morphisms:} the cohomological grading $\op{gr}$ is defined in the same way as in $\cl$.

\begin{rmk}
(1) The elementary diagrams in $\cv$ do not include cups or caps since only even labels are allowed.
Moreover, there is no relation as the isotopy relations (R1-b,c,d) of $\cl$.
Strictly speaking, the morphisms in $\cv$ are rigid diagrams, not isotopy classes of diagrams.

\n(2) The relations of $\cv$ are much simpler than those of $\cl$.
For instance, any double crossing is either identity or zero depending on whether the labels are the same or not.
\end{rmk}

\n $\bullet$ {\bf More about (L5):} there are two cases when we cannot slide a dot to the left of a crossing to right using (L5).
The corresponding two diagrams in which the dots are to the left of the crossings are given in Figure \ref{v6}.
\begin{figure}[h]
\begin{overpic}
[scale=0.22]{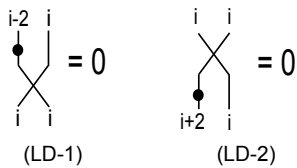}
\end{overpic}
\caption{Relations (LD)}
\label{v6}
\end{figure}

\begin{lemma} \label{LD}
The relations (LD) in Figure \ref{v6} hold in $\cv$.
\end{lemma}
\begin{proof}
We prove (LD-1) in Figure \ref{v7}. The proof for (LD-2) is similar.
\begin{figure}[h]
\begin{overpic}
[scale=0.25]{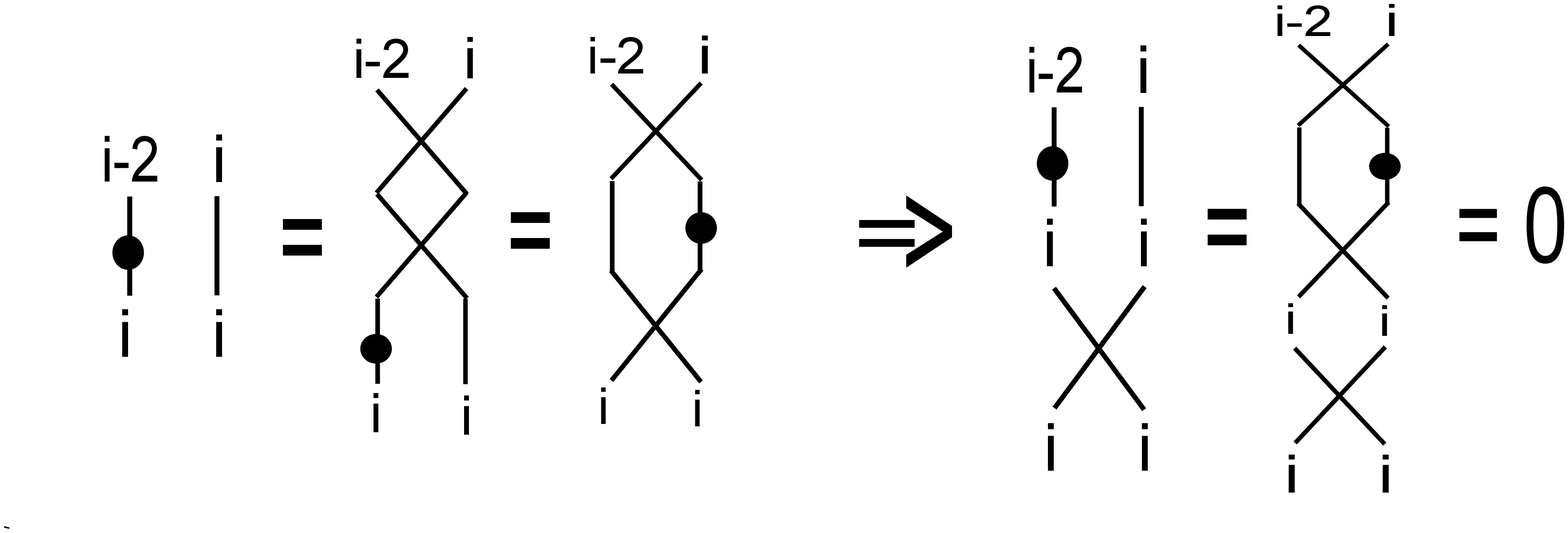}
\end{overpic}
\caption{}
\label{v7}
\end{figure}
\end{proof}

A diagram $f \in D(\mf{x},\mf{y})$ is {\em nonzero} if $[f]\neq 0 \in \Hom(\mf{x},\mf{y})$.
The following lemma is an easy consequence of the relations (LD).
\begin{lemma}\label{LRP}
Left-to-Right Principle of dot slide (LRP) holds: in a nonzero diagram we can always perform a dot slide from left to right in $\cv$.
\end{lemma}

\begin{rmk}
There are two cases when we cannot slide a dot to the right of a crossing to left using (L5).
We will prove that the corresponding two diagrams are actually nonzero in Proposition \ref{hom}.
Hence it is not always true that a dot can be slid from right to left in $\cv$.
\end{rmk}

\n $\bullet$ {\bf Charge and energy:}
Consider a grading $\op{ch}: \E(\cv) \ra \Z$ given by $\op{ch}(\mf{x})=k$ if $x_n=2n+2k$ for $n \gg 0$, where $\mf{x}=(x_1, x_2, \dots)$.
The grading $\op{ch}$ is called the {\em charge} following \cite{Kac}.
There is a corresponding decomposition:
$$\E(\cv)=\bigsqcup_k \E(\cv)_k,$$
where $\E(\cv)_k$ consists of elementary objects of charge $k$.
Since any diagram in a morphism set consists of trivial vertical strands after finitely many terms, we
have an easy observation:

\begin{lemma} \label{charge}
The charge gives a decomposition of the category $\cv$:
$$\Hom_{\cv}(\mf{x},\mf{y})=0 \quad \mbox{if} \quad \op{ch}(\mf{x})\neq \op{ch}(\mf{y}).$$
\end{lemma}

\begin{rmk}
Recall that a dividing set divides the surface into positive and negative regions.
The {\em Euler number} of a dividing set is the Euler characteristic of the positive region minus the Euler characteristic of the negative region.
The charge of $\mf{x} \in \E(\cv)$ is related to the Euler number of the corresponding dividing set in $\Cv$.
\end{rmk}

There is a distinguished object $\es(k)=(2k+2, 2k+4, \dots) \in \E(\cv)_k$, called {\em vacuum object} of charge $k$.
The {\em energy} $\op{E}_k$ is defined on $\E(\cv)_k$ by
$$\begin{array}{cccc}
\op{E}_k:& \E(\cv)_k & \ra & \Z \\
& (x_1, x_2, \dots)  & \mapsto & \frac{1}{2} \sum\limits_{i=1}^{\infty}(2k+2i-x_i),
\end{array}$$
where only finitely many terms are nonzero in the sum.
The energy measures the total difference between $\es(k)$ and $\mf{x} \in \E(\cv)_k$.
In particular, the vacuum object has energy zero.
Note that $\op{E}_k(\mf{x})<0$ implies that $\mf{x}$ contains repetitive numbers: $x_i=x_j$ for some $i \neq j$ and we will show that such $\mf{x}$ is isomorphic to the zero object in the homology category of $\cv$.

\vspace{.2cm}
\n $\bullet$ {\bf Monotonicity:} the key feature of $\cv$ is that any nontrivial morphism either retains or increases the energy.
\begin{lemma} \label{energy}
If the morphism set $\Hom_{\cv}(\mf{x},\mf{y}) \neq 0$ for $\mf{x},\mf{y} \in \E(\cv)_k$, then
$$\op{E}_k(\mf{x})\leq \op{E}_k(\mf{y}).$$
\end{lemma}
\begin{proof}
It follows from the fact that any elementary diagram either retains or increases the energy by $1$.
\end{proof}
This monotonicity with respect to the energy makes $\cv$ much easier to understand than $\cl$.
We will give an algebraic formulation of $\cv$ based on this property.

\subsection{The differential}

\subsubsection{Motivation}
We try to mimic the differential in $\cl$.
The only difficulty is that a differential of a dotted strand contains a cup and a cap which do not exist in $\cv$.
We modify the definition as follows.

Consider a finite reduction of $\cv$ first, where labels of strands are in $\{0,2,\dots, 2n-2\}$.
Let $U(n)$ be a free abelian group generated by increasing sequences of even integers in $\{0,2,\dots, 2n-2\}$.
Let $Cl(n)$ be a subalgebra of $Cl_{\Z}$ generated by $\{a_i ~|~ 0 \leq i \leq 2n-1\}$, and $J(n)$ be a left ideal of $Cl(n)$ generated by $\{a_i ~|~ i ~\mbox{odd}, ~0 \leq i \leq 2n-1\}$.
Then $U(n)$ is isomorphic to a quotient $Cl(n) / J(n)$ as $Cl(n)$-modules.
Under this isomorphism, any sequence of integers in $\{0,1,\dots,2n-1\}$ with an odd term in the last is in $J(n)$ and equals zero in $U(n)$.
Therefore, it is reasonable to set the sequence as the zero object in a diagrammatic category which lifts $U(n)$.

We now apply this idea to a few examples.
For instance in Figure \ref{v3}, the differential a dotted strand $d(dot_0)$ is set to be zero since the intermediate state $(2,0,1)$ is the zero object.
\begin{figure}[h]
\begin{overpic}
[scale=0.2]{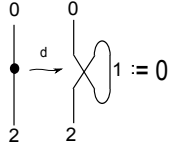}
\end{overpic}
\caption{The differential is zero since the label of the rightmost strand is odd}
\label{v3}
\end{figure}

If there are some strands to the right of the dotted strand, we use the double crossing relation (R4) and triple intersection relation (R5) to transport the strand with odd label from left to right and straighten the diagram whenever it is possible.
See Figure \ref{v3-1} for more examples.
\begin{figure}[h]
\begin{overpic}
[scale=0.2]{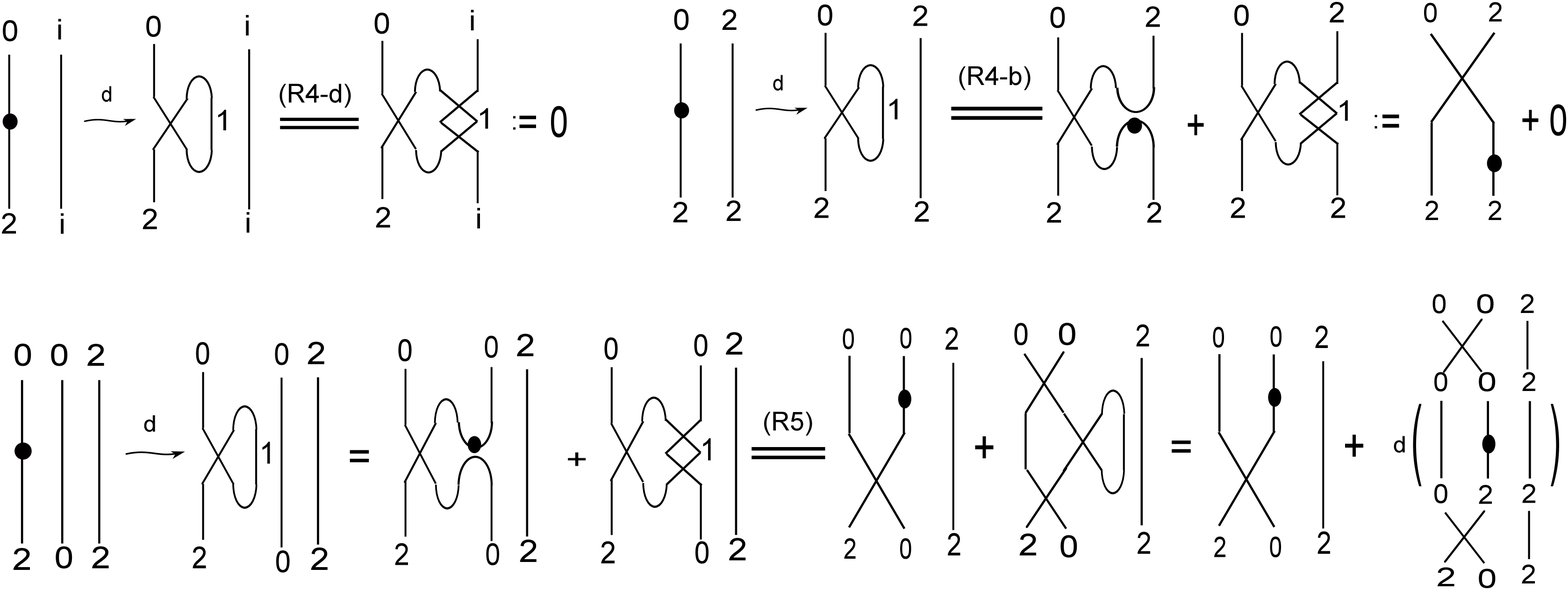}
\put(12,20){$i \neq 0,2$}
\end{overpic}
\caption{Transporting an odd strand from left to right using (R4) and (R5)}
\label{v3-1}
\end{figure}

Note that the strand with label $1$ in the differential of $dot_0$ commutes with any strand to its right with label $i \neq 0,2$ as in the top left diagram of Figure \ref{v3-1}.
For the semi-infinite case, any strand with an odd label commutes with all but finitely many strands since $x_n=2n+2k$ for $n\gg0$.
We define the differential as zero if the odd strand commutes with all strands to its right.

\subsubsection{Definition of differential}
We define the differential on the generating diagrams in the following and extend it to general diagrams by Leibniz's rule with respect to the vertical stacking.

\vspace{.2cm}
\n(Case 1) For an identity and a crossing, the definition is given in Figure \ref{v5}.

\vspace{.2cm}
\n(Case 2, Step 1) For a dotted strand $dot_{i}(\mf{a}, \mf{x})$, the differential passes through the trivial strands $id_{\mf{a}}$ as in Figure \ref{v5}.
Hence we reduce to the case where $\mf{a}=(\es)$.

\vspace{.1cm}
\n(Case 2, Step 2) For a dotted strand $dot_{i}((\es), \mf{x})$, the differential is defined inductively on the number $\beta(\mf{x},i)=|\{k \in \Z ~|~ x_k=i ~\mbox{or}~ i+2\}|$:
\be
\item if $\beta(\mf{x},i)=0$, then let $d(dot_{i}((\es), \mf{x}))=0$;
\item if $\beta(\mf{x},i)\neq0$, then write $\mf{x}=(j)\cdot\mf{y}$ and define $d(dot_{i}((\es), \mf{x}))$ in terms of $d(dot_{i}((\es), \mf{y}))$ as in Figure \ref{v5}.
\ee
The induction terminates after finitely many steps since $x_n=2n+2k$ for $n\gg0$.
\begin{figure}[h]
\begin{overpic}
[scale=0.22]{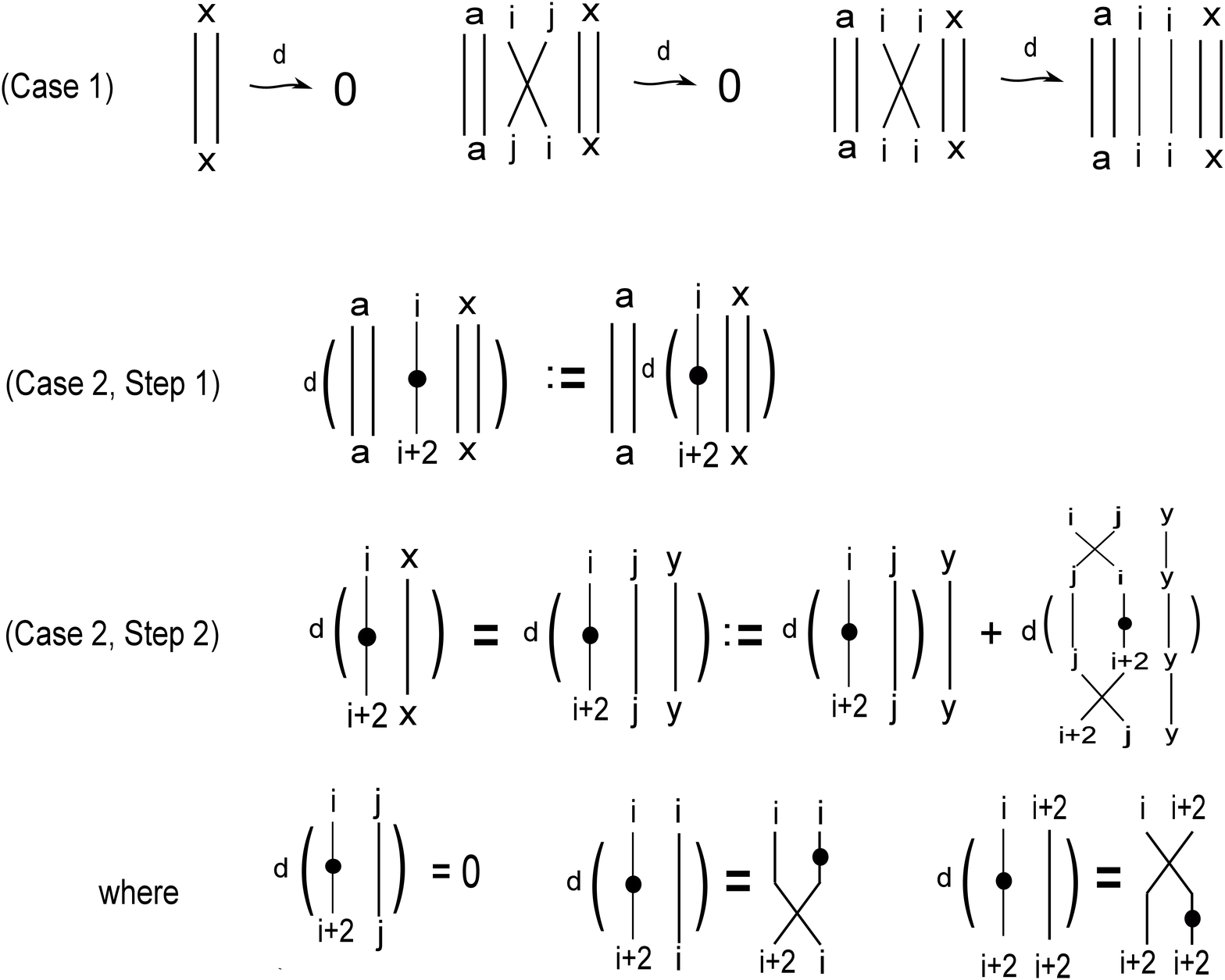}
\put(40,62){$i \neq j$}
\put(21,-2){$j \neq i,i+2$}
\end{overpic}
\caption{Differential on the generating diagrams}
\label{v5}
\end{figure}
\begin{rmk}
The differential does not satisfy Leibniz's rule with respect to the horizontal stacking.
\end{rmk}

\subsubsection{Well-defined differential}
We show that (1) the relations of $\cv$ are preserved under $d$; and (2) $d^2=0$.

\begin{lemma}
The relations of $\cv$ are preserved under $d$.
\end{lemma}
\begin{proof}
The proofs for relations without dots are similar to those of $\cl$ and are left to the reader.
It remains to verify the following relations:
\be
\item double dot relation;
\item dot slide relation;
\item isotopy of disjoint diagrams.
\ee
Since the differential can pass through finitely many trivial strands to the left of the dots, we may assume that the dots are on the leftmost strand with labels $0,2$ or $4$.
The calculation is done by induction on $\beta(\mf{x},0)$ or $\beta(\mf{x},2)$.
Here, we only present the calculation for (1) in Figure \ref{v8}.
The proofs for (2) and (3) are similar and are left to the reader.
\begin{figure}[h]
\begin{overpic}
[scale=0.23]{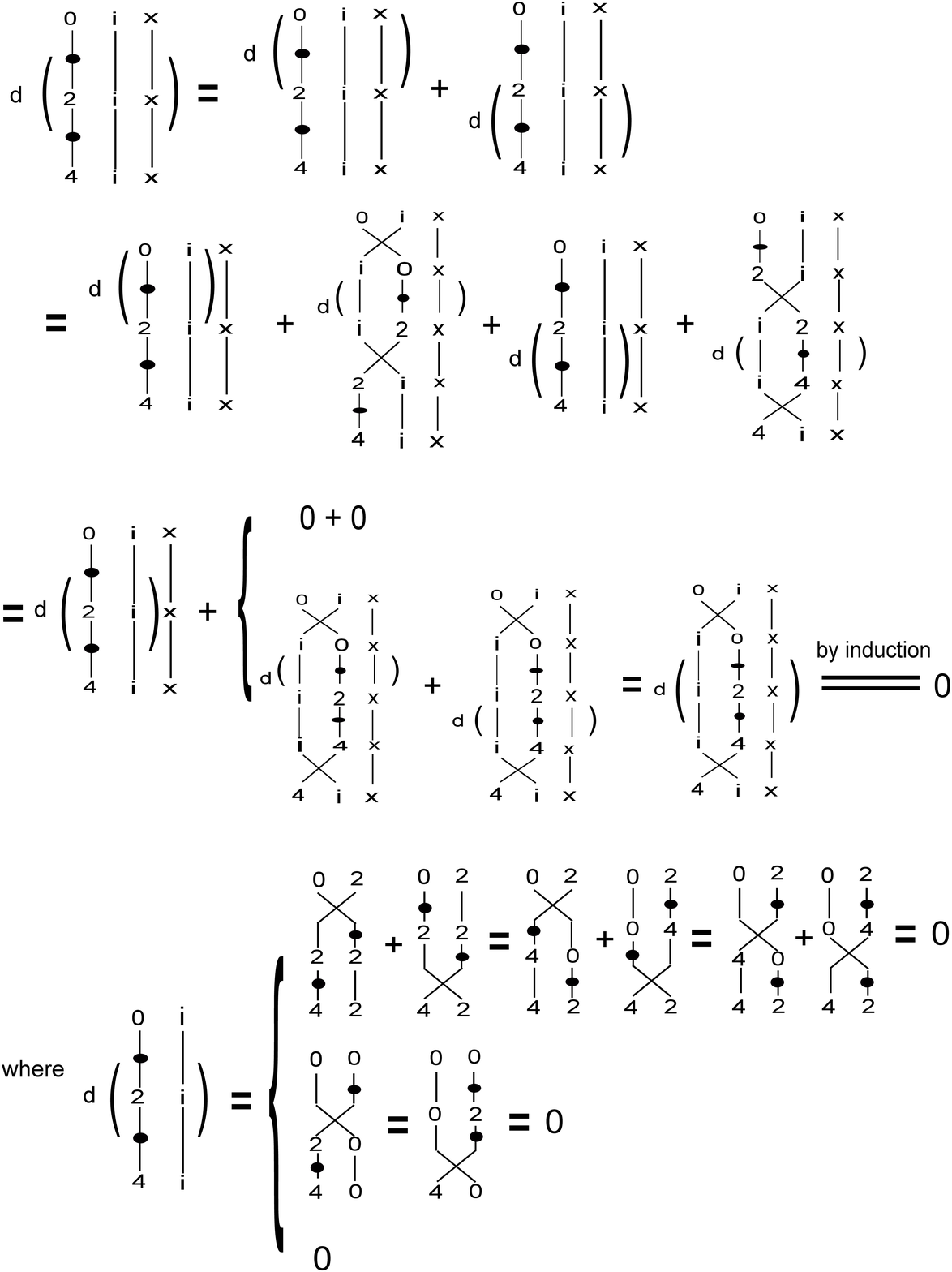}
\put(40,59){if $i=2$}
\put(77,46){if $i\neq 2$}
\put(77,25){if $i = 2$}
\put(50,10){if $i = 0$; $i = 4$ similar to $i=0$}
\put(40,0){if $i \neq 0,2,4$}
\end{overpic}
\caption{Double dot relation}
\label{v8}
\end{figure}
\end{proof}

\begin{lemma}
$d^2=0$.
\end{lemma}
\begin{proof}
It suffices to prove that $d^2=0$ for any generating diagram.
The nontrivial case is for a dotted strand $dot_i(\mf{a},\mf{x})$.
We may assume $\mf{a}=(\es)$ and $i=0$ as before.
The calculation is done by induction on $\beta(\mf{x},0)$ in Figure \ref{v9}.
\begin{figure}[h]
\begin{overpic}
[scale=0.25]{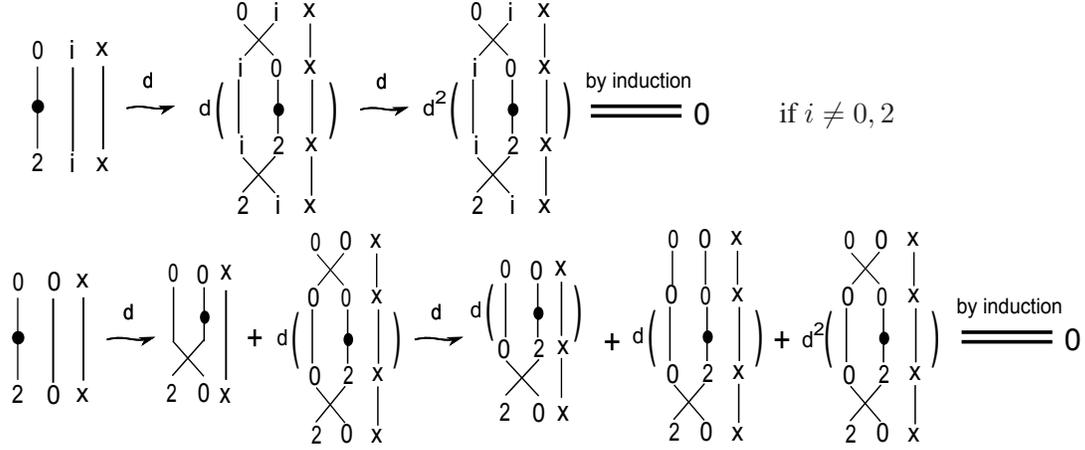}
\put(72,30){if $i \neq 0,2$}
\end{overpic}
\caption{The case $i=2$ is similar to $i=0$}
\label{v9}
\end{figure}
\end{proof}

This completes the description of $\Hom_{\cv}(\mf{x},\mf{y})$ as a complex of $\F$-vector spaces for $\mf{x}, \mf{y} \in \E(\cv)$.

\subsection{Canonical forms of nonzero diagrams}
We show that any morphism set $\Hom(\mf{x},\mf{y})$ is a finite-dimensional $\F$-vector space generated by {\em canonical forms} of nonzero diagrams.

Any diagram is determined by a configuration of strands and positions of dots.
We discuss the configuration first.
A {\em matching} $\sigma$ is a bijection $\sigma: \Z_+ \ra \Z_+$ such that $\sigma(n)=n$ for $n\gg0$.
Any diagram $f \in D(\mf{x},\mf{y})$ determines a matching $\sigma_f$, where $x_i$ is connected to $y_{\sigma_f(i)}$ in $f$.
Then $D(\mf{x},\mf{y})$ has a decomposition:
$$D(\mf{x},\mf{y})=\bigsqcup_{\sigma}D(\mf{x},\mf{y};\sigma),$$
where $D(\mf{x},\mf{y};\sigma)=\{f \in D(\mf{x},\mf{y}) ~|~ \sigma_f=\sigma\}$.
There is a corresponding decomposition of $\Hom(\mf{x},\mf{y})$, as a quotient of a free vector space generated by $D(\mf{x},\mf{y})$:
$$\Hom(\mf{x},\mf{y})=\bigoplus_{\sigma}\Hom(\mf{x},\mf{y};\sigma),$$
since any relation preserves the matching of diagrams.

\begin{lemma}\label{sigmafinite}
For $\mf{x}, \mf{y} \in \E(\cv)_k$, there are only finitely many $\sigma$ such that $D(\mf{x},\mf{y};\sigma)$ is nonempty.
\end{lemma}
\begin{proof}
There exists $N$ which only depends on $\mf{x}$ and $\mf{y}$ such that $x_n=y_n=2n+2k$ for $n\geq N$, and $x_i< 2N+2k$ for $i < N$.
We will prove that for any diagram $f \in D(\mf{x},\mf{y})$, $\sigma_f(n)=n$ for $n \geq N$.

Firstly, $y_n$ cannot be connected to $x_i$ in $f$ for $i<N\leq n$ because $y_n \geq 2N+2k > x_i$ and a label does not increase along a strand as we go from bottom to top.
Therefore, elements in subsets $\{x_n ~|~ n \geq N\}$ and $\{y_n ~|~ n \geq N\}$ are connected with each other.
Secondly, $x_N$ must be connected to $y_N$ since $y_n>x_N$ for $n>N$.
We inductively prove that $x_n$ is connected to $y_n$, i.e., $\sigma_f(n)=n$ for $n \geq N$.
In particular, there are only finitely many $\sigma$ such that $D(\mf{x},\mf{y};\sigma)$ is nonempty.
\end{proof}

We want to determine a configuration of a canonical diagram in each nonempty $D(\mf{x},\mf{y};\sigma)$ according to the following {\em crossing vector}.
\begin{defn}
Define the crossing vector $\op{CR}: D(\mf{x},\mf{y};\sigma) \ra \N^{\times \Z_+}$ as: $$\op{CR}(f)=(\op{CR}(f)_1, \op{CR}(f)_2, \dots),$$ where $\op{CR}(f)_i$ is the number of crossings of $f$ in the $i$th strip $[i,i+1] \times [0,1]$.
\end{defn}

\begin{defn}
Define an order ``$<$" on $\N^{\times \Z_+}$ as follows: we write $\mf{v}<\mf{w}$ if there exists some $i_0 \in \Z_+$ such that $v_i=w_i$ for $i<i_0$ and $v_{i_0}<w_{i_0}$, where $\mf{v}=(v_1, v_2, \dots)$ and $\mf{w}=(w_1, w_2, \dots)$ in $\N^{\times \Z_{+}}$.
\end{defn}

\begin{defn} \label{defcan}
For nonempty $D(\mf{x},\mf{y};\sigma)$, define the {\em canonical diagram} $can(\mf{x},\mf{y};\sigma) \in D(\mf{x},\mf{y};\sigma)$ as an element satisfying the following conditions:
\be
\item $CR(can(\mf{x},\mf{y};\sigma))=\op{min}\{CR(f)~|~ f \in D(\mf{x},\mf{y};\sigma)\}$ with respect to the order on $\N^{\times \Z_+}$;
\item no dot is to the left of any crossing in $can(\mf{x},\mf{y};\sigma)$.
\ee
\end{defn}

\begin{lemma} \label{unican}
The canonical diagram $can(\mf{x},\mf{y};\sigma)$ is unique in nonempty $D(\mf{x},\mf{y};\sigma)$.
\end{lemma}
\begin{proof}
The configuration of strands in $can(\mf{x},\mf{y};\sigma)$ is uniquely determined by the condition (1).
Moreover, for any strand in $can(\mf{x},\mf{y};\sigma)$, portions of the strand which attain local maximal values under the projection $\R \times [0,1] \ra \R$ have only one connected component.
All the dots on the strand must lie on this connected portion by the condition (2).
Hence $can(\mf{x},\mf{y};\sigma)$ is unique in $D(\mf{x},\mf{y};\sigma)$.
\end{proof}

\begin{rmk}\label{rmkcan}
Any canonical diagram has at most one crossing in the first strip $[1,2] \times [0,1]$.
\end{rmk}

\begin{lemma} \label{can}
(1) The canonical diagram $can(\mf{x},\mf{y};\sigma)$ has the maximal grading in $D(\mf{x},\mf{y};\sigma)$.

\n(2) Any nonzero diagram in $D(\mf{x},\mf{y};\sigma)$ is equivalent to $can(\mf{x},\mf{y};\sigma)$ under the relations of $\cv$.
\end{lemma}
\begin{proof}
(1) Starting with any diagram $f \in D(\mf{x},\mf{y};\sigma)$, we can perform local operations in Figure \ref{can1} to make $CR(f)$ smaller with respect to the order ``$<$" and eventually reach $can(\mf{x},\mf{y};\sigma)$.
The local operations of Figure \ref{can1} consist of two types:
\begin{description}
\item[Type a] those are not actual relations of $\cv$: a double crossing of $cr_{i,i}$ which is equal to zero by (L3-a); the two cases in (LD) where dot slide relation (L5) fail.
\item[Type b] those are relations of $\cv$.
\end{description}
The operations of Type (a) increase the grading by $2$, and those of Type (b) preserve the grading.
Hence $\op{gr}(can(\mf{x},\mf{y};\sigma)) \geq \op{gr}(f)$ for any $f \in D(\mf{x},\mf{y};\sigma)$.

\begin{figure}[h]
\begin{overpic}
[scale=0.22]{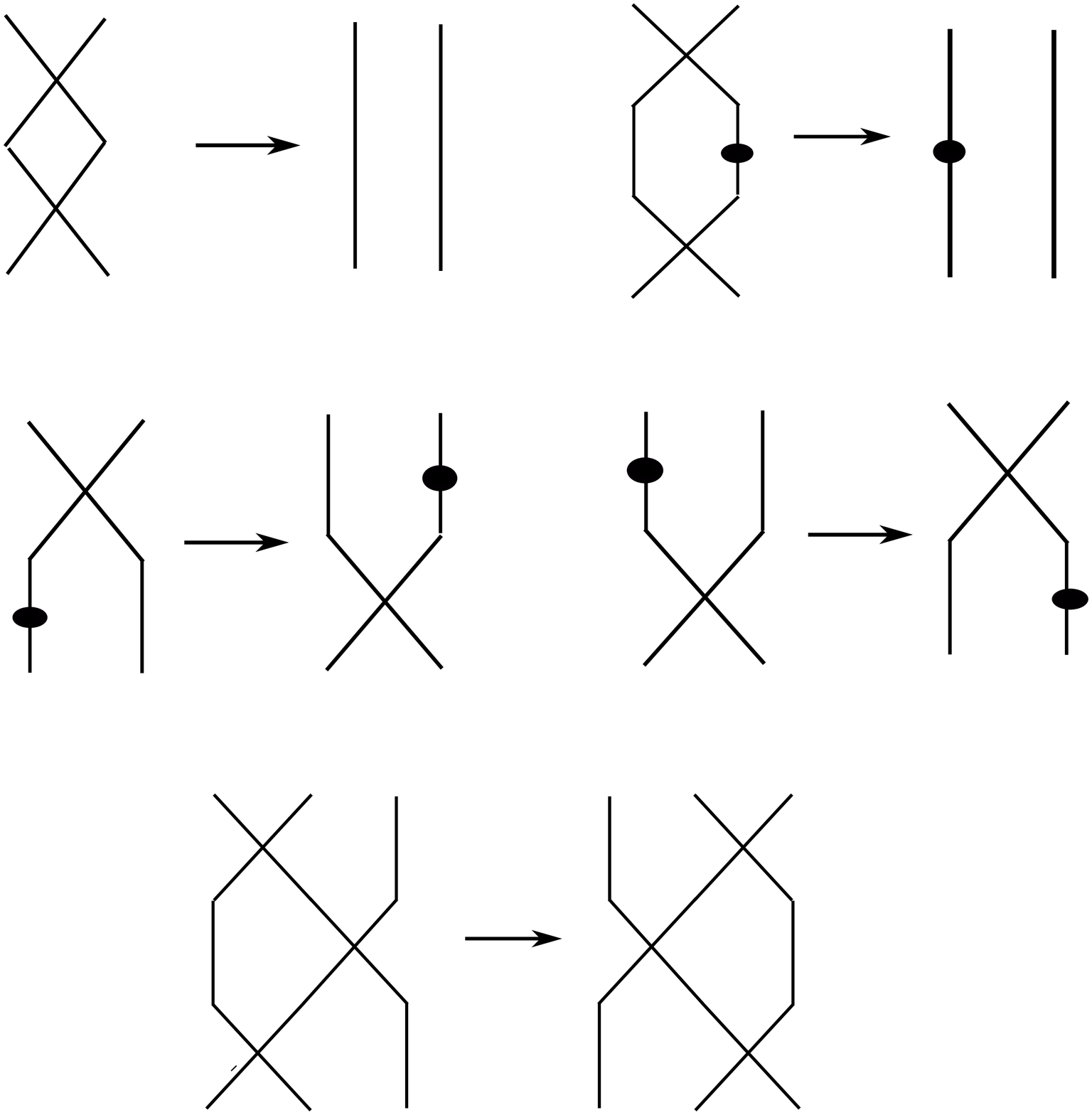}
\end{overpic}
\caption{Local operations in $D(\mf{x},\mf{y};\sigma)$ to obtain $can(\mf{x},\mf{y};\sigma)$.}
\label{can1}
\end{figure}
\n(2) Suppose that $f \in D(\mf{x},\mf{y};\sigma)$ is not equal to $can(\mf{x},\mf{y};\sigma)$ in $\Hom(\mf{x},\mf{y})$.
Consider a sequence of local operations which convert $f$ to $can(\mf{x},\mf{y};\sigma)$.
Then there exists at least one operation of Type (a) in the sequence.
Let $g$ be the diagram before performing the first operation of Type (a).
Firstly, $f=g \in \Hom(\mf{x},\mf{y})$ since all operations between $f$ and $g$ are of Type (b).
Secondly, $g=0 \in \Hom(\mf{x},\mf{y})$ since any diagram before performing each of the operations of Type (a) are zero by (L3-a) and (LD).
Hence, $f$ is not a nonzero diagram.
\end{proof}

A matching $\sigma$ is called {\em admissible} if $D(\mf{x},\mf{y};\sigma)$ is nonempty and there is at most one dot on each strand of any diagram in $D(\mf{x},\mf{y};\sigma)$.
Now we give a complete description of $\Hom(\mf{x},\mf{y};\sigma)$.

\begin{prop} \label{hom}
The morphism set
$$ \Hom(\mf{x},\mf{y};\sigma)= \left\{
\begin{array}{cl}
\F\{can(\mf{x},\mf{y};\sigma)\} & \mbox{if} \hspace{0.3cm} \sigma ~~\mbox{admissible}; \\
0 & \mbox{otherwise}.
\end{array}\right.
$$
\end{prop}
\begin{proof}
Suppose $\sigma$ is not admissible.
Arguing by contradiction, suppose that $\Hom(\mf{x},\mf{y};\sigma) \neq 0$. Let $f \in D(\mf{x},\mf{y};\sigma)$ be a nonzero diagram.
Then $f=can(\mf{x},\mf{y};\sigma) \in \Hom(\mf{x},\mf{y};\sigma)$ by Lemma \ref{can}(2).
By the proof of Lemma \ref{unican}, all the dots on each strand in $can(\mf{x},\mf{y};\sigma)$ must lie on the rightmost portion of the strand which is connected.
Since $\sigma$ is not admissible, $can(\mf{x},\mf{y};\sigma)=0$ by the double dot relation which is a contradiction.

\vspace{.1cm}
Suppose $\sigma$ is admissible.
We first show that $can(\mf{x},\mf{y};\sigma)$ is a nonzero diagram.
Arguing by contradiction.
We can perform all relations of $\cv$ but the double dot relation (L2) or the double crossing of $cr_{i,i}$ (L3-a) on $can(\mf{x},\mf{y};\sigma)$ to obtain $f$  such that $f$ contains either a double dot or a double crossing of $cr_{i,i}$.
In particular, $\op{gr}(can(\mf{x},\mf{y};\sigma))=\op{gr}(f)$.
Since $\sigma$ is admissible and any relation except (L2) preserves the number of dots on each strand, $f$ cannot have double dots.
Hence $f$ contains a double crossing of $cr_{i,i}$.
Let $g$ denote the diagram obtained from $f$ by removing the double crossing.
Then $g$ is still in $D(\mf{x},\mf{y};\sigma)$, but $\op{gr}(g)=\op{gr}(f)+2=\op{gr}(can(\mf{x},\mf{y};\sigma))+2$.
This contradicts with $can(\mf{x},\mf{y};\sigma)$ having the maximal grading in $D(\mf{x},\mf{y};\sigma)$ by Lemma \ref{can}(1).
Hence $can(\mf{x},\mf{y};\sigma)$ is a nonzero diagram.

By Lemma \ref{can}(2), any nonzero diagram in $D(\mf{x},\mf{y};\sigma)$ is equivalent to $can(\mf{x},\mf{y};\sigma)$.
Therefore, $\Hom(\mf{x},\mf{y};\sigma)$ is a $1$-dimensional $\F$-vector space generated by $can(\mf{x},\mf{y};\sigma)$.
\end{proof}

The following corollary is an easy consequence of Lemma \ref{sigmafinite} and Proposition \ref{hom}.
\begin{cor}\label{homfinite}
Any morphism set $\Hom(\mf{x},\mf{y})$ is finite-dimensional.
\end{cor}

\subsection{The DG category $\cv$}
We use the same enlargement from the elementary objects to one-sided twisted complexes as in Section 2.2.
Define the DG category $\cv$ as the category of one-sided twisted complexes of elementary objects in $\E(\cv)$.
Let $H(\cv)$ be its $0$th homology category which is a triangulated category.
We will show that $K_0(\cv)$ is isomorphic to $V$ in Section 5.

We give some isomorphisms between objects in $\E(\cv)$.
Let $\op{N}(\mf{x})$ denote a rearrangement of $\mf{x} \in \E(\cv)$ which is non-decreasing.
It is called the {\em normalization} of $\mf{x}$.
Let $\eta(\mf{x})=|\{(i,j)\in \Z^2 ~|~ i<j, ~x_i>x_j\}|$ denote the number of decreasing pairs in $\mf{x}$.

\begin{lemma} \label{normal}
Any elementary object $\mf{x}$ is isomorphic to its normalization in $\cv$ up to a grading shift:
$$\mf{x} \simeq \op{N}(\mf{x})[\eta(\mf{x})].$$
\end{lemma}
\begin{proof}
By induction on $\eta(\mf{x})$, it suffices to prove the lemma for $\eta(\mf{x})=1$.
When $\eta(\mf{x})=1$, there exists an adjacent decreasing pair $(i,i+1)$ with $x_i > x_{i+1}$.
The crossings $cr_{x_i,x_{i+1}}$ and $cr_{x_{i+1},x_{i}}$ between $\mf{x}$ and $\op{N}(\mf{x})[1]$ are isomorphisms by the double crossing relation (L3-b).
\end{proof}

Consider a collection of the non-decreasing sequences in $\cv$:
$$\cb=\bigsqcup_k\cb_k=\bigsqcup_k\{\mf{x} \in \E(\cv)_k ~|~ x_1 \leq x_2 \leq \cdots\}.$$
Let $\cv'$ be the DG category of one-sided twisted complexes of objects in $\cb$.
\begin{rmk} \label{vv'}
The categories $\cv$ and $\cv'$ are equivalent.
\end{rmk}

By the argument from Lemma \ref{K0cl}(1), any sequence with repetitive numbers is isomorphic to the zero object in the homology category $H(\cv)$.
Consider a collection of increasing sequences in $\cv$:
$$\hb=\bigsqcup_k\hb_k=\bigsqcup_k\{\mf{x} \in \E(\cv)_k ~|~ x_1 < x_2 < \cdots\}.$$
Note that $\hb$ coincides with the basis of the representation $V$ in Section 3.
We will build up a DG algebra $R$ from $\cb$ and its cohomology $H(R)$ from $\hb$ in Section 5.

\section{The algebraic formulation $DGP(R)$}
We give an equivalent description of $\cv$ in terms of projective DG $R$-modules, where $R$ is a DG algebra $R=\bigoplus_k R_k$:
$$R_k = \bigoplus_{\mf{x},\mf{y}\in \cb_k}\Hom_{\cv}(\mf{x},\mf{y}).$$
The multiplication and differential on $R$ are given by the composition and differential in $\cv$.
We will omit $\cv$ in $\Hom_{\cv}(\mf{x},\mf{y})$ from now on.

In Section 5.1, we show that the cohomology $H(R)$ is generated by crossingless diagrams between objects in $\hb$.
It can be represented as a quotient of a path algebra of a quiver whose vertex set is $\hb$.
In Section 5.2, we define a DG category $DGP(R)$ generated by some projective DG $R$-modules and show that $\cv$ is equivalent to $DGP(R)$.
We also show that the Grothendieck group of $H(DGP(R))$ is isomorphic to $V$.

\subsection{The DG algebra $R$}
We compute the cohomology $H(R)$ and show that $R$ is {\em formal}, i.e., $R$ and $H(R)$ are quasi-isomorphic.

An important observation is that $R$ is nonzero since the morphism set of $\cv$ is monotone with respect to the energy by Lemma \ref{energy}.
In particular, the infinite-dimensional algebra $R$ contains a family of mutually orthogonal idempotents
$e(\mf{x}):= id_{\mf{x}}$ for $\mf{x} \in \cb$.

Although the size of $R$ is very large, its cohomology $H(R)$ is much simpler.
We first look at cohomology of $\Hom(\mf{x},\mf{y})$ if $\mf{x}$ or $\mf{y}$ contains repetitive numbers.

\begin{lemma} \label{HR1}
$H(\Hom(\mf{x},\mf{y}))=0$ if either $\mf{x}$ or $\mf{y}$ in $\cb \bs \hb$.
\end{lemma}
\begin{proof}
Assume that $\mf{x} \in \cb \bs \hb$ such that $x_i=x_{i+1}=a$ for some $i$.
Let $f \in \Hom(\mf{x},\mf{x})$ be a diagram with a single crossing $cr_{a,a}$ so that $df=id_{\mf{x}}$.
Hence $[id_{\mf{x}}]=0 \in H(\Hom(\mf{x},\mf{y}))$.
Then for any class $[g] \in H(\Hom(\mf{x},\mf{y}))$, we have
$[g]=[g \circ id_{\mf{x}}]=[g][id_{\mf{x}}]=0$.
\end{proof}

As an $\F$-vector space, $H(R_k)$ is generated by morphisms between elements in $\hb$:
$$H(R_k)=\bigoplus_{\mf{x},\mf{y}\in \hb_k}H(\Hom(\mf{x},\mf{y})).$$
In order to obtain the algebra structure on $H(R_k)$, we want to find generators of $\Hom(\mf{x},\mf{y})$ whose cohomology is nonzero.
The following result is about the morphism set between two objects of the same energy.
\begin{lemma} \label{HR2}
For $\mf{x}, \mf{y} \in \hb_k$, if $\op{E}_k(\mf{x})=\op{E}_k(\mf{y})$, then
$$ \Hom(\mf{x},\mf{y})=H(\Hom(\mf{x},\mf{y}))= \left\{
\begin{array}{cl}
\F\{id_{\mf{x}}\} & \mbox{if} \hspace{0.3cm} \mf{x}=\mf{y}; \\
0 & \mbox{otherwise}.
\end{array}\right.
$$
\end{lemma}
\begin{proof}
Suppose that $D(\mf{x},\mf{y})$ is nonempty.
Any diagram in $D(\mf{x},\mf{y})$ contains no dots since energy increases by $1$ along a dotted strand as we go from bottom to top.
Hence $\mf{y}$ is a rearrangement of $\mf{x}$.
We have $\mf{x}=\mf{y}$ since they are both increasing sequences.
Then the only admissible matching is the trivial one $id$.
The lemma follows from Proposition \ref{hom}.
\end{proof}

Now we give a complete description of $\Hom(\mf{x},\mf{y})$ for $\mf{x}, \mf{y} \in \hb_k$.
Let $\cal{I}$ be a finite index set of positive integers which is possibly empty.
For $\mf{x}, \mf{y} \in \hb_k$, we write $\mf{x} \xra{\cal{I}} \mf{y}$ if $x_i=y_i+2$ for $i \in \cal{I}$, and $x_i=y_i$ otherwise.

\begin{lemma} \label{HR3}
For $\mf{x}, \mf{y} \in \hb_k$, we have
$$ \op{dim}_{\F}\Hom(\mf{x},\mf{y})= \left\{
\begin{array}{cl}
1 & \mbox{if} \hspace{0.3cm} \mf{x} \xra{\cal{I}} \mf{y} ~~\mbox{for some}~~ \cal{I}; \\
0 & \mbox{otherwise}.
\end{array}\right.
$$
\end{lemma}
\begin{proof}
Suppose that $\Hom(\mf{x},\mf{y})\neq0$.
We prove that any nonzero diagram must connect $x_1$ to $y_1$.

\vspace{.1cm}
\n {\bf Claim 1}: $y_1 \leq x_1$.
Arguing by contradiction, suppose $y_1 > x_1$.
We have $y_k \geq y_1 > x_1$ for all $k$.
Then the strand from $x_1$ has no term to connect in $\mf{y}$ since a label does not increase along any strand as we go from bottom to top.

\vspace{.1cm}
\n {\bf Claim 2}: $y_1 \geq x_1-2$.
Arguing by contradiction, suppose $y_1 < x_1-2$.
We have $x_k \geq x_1 > y_1+2$ for all $k$.
Then the strand to $y_1$ contains at least two dots so that the diagram is zero by Proposition \ref{hom}.

\vspace{.1cm}
Hence we have either $y_1=x_1$ and a strand connecting them; or $y_1=x_1-2$ and a dotted strand connecting them.
We inductively prove that any nonzero diagram connects $x_k$ to $y_k$ for all $k$.
Moreover, there exists a finite index set $\cal{I}$ such that $\mf{x} \xra{\cal{I}} \mf{y}$.
In other words, the trivial matching $id$ is the only admissible one.
The lemma follows from Proposition \ref{hom}.
\end{proof}

\begin{rmk}
The connection between $\cv$ and the contact category $\Cv$ is better understood than that between $\cl$ and $\Ccl$ as in Section 2.4.
In particular, the analogue of Conjecture \ref{conjcl} for $\cv$ is true:
$$ \op{dim}_{\F}H^*(\Hom_{\cv}(\mf{x},\mf{y}))= \left\{
\begin{array}{cl}
1 & \mbox{if}~~~ \overline{\mf{y}} ~~\mbox{is stackable over}~~ \overline{\mf{x}}; \\
0 & \mbox{otherwise}.
\end{array}\right.
$$
Here $\overline{\mf{x}}$ and $\overline{\mf{y}}$ are dividing sets of $\Cv$ corresponding to $\mf{x}$ and $\mf{y}$ in $\hb_k$, respectively.
\end{rmk}

Let $r(\mf{x},\mf{y};i)$ denote the generator in $\Hom(\mf{x},\mf{y})$ if $\mf{x} \xra{\cal{I}} \mf{y}$ for $\cal{I}=\{i\}$.
It is represented by $can(\mf{x},\mf{y}; id)$, a crossingless diagram with one dot on the $i$th strand.
We have the following easy description of the cohomology.

\begin{prop} \label{HR}
The cohomology $H(R_k)$ is generated by $\{r(\mf{x},\mf{y};i) ~|~ \mf{x} \xra{\{i\}} \mf{y} ~\mbox{for}~ \mf{x},\mf{y} \in \hb_k\}$ with relations:
\begin{gather*}
r(\mf{x},\mf{y};i)\cdot r(\mf{y},\mf{z};i)=0;\\
r(\mf{x},\mf{y};i)\cdot r(\mf{y},\mf{w};j)=r(\mf{x},\mf{z};j) \cdot r(\mf{z},\mf{w};i), ~~\mbox{for}~~ i\neq j.
\end{gather*}
Moreover, the inclusion $H(R_k) \hookrightarrow R_k$ is a quasi-isomorphism.
\end{prop}
\begin{proof}
Since any nonzero element in $H(\Hom(\mf{x},\mf{y}))$ is given by a crossingless diagram, it is a composition of $r(\mf{x},\mf{y};i)$'s.
The first relation comes from the double dot relation; the second one comes from the isotopy of two disjoint dotted strands.
\end{proof}

Let $\g_k$ be a quiver with vertex set $V(\g_k) = \hb_k$ and an arrow from $\mf{x}$ to $\mf{y}$ if $\mf{x} \xra{\{i\}} \mf{y}$.
Then $H(R_k)$ is a quotient of the path algebra $\F\g_k$ of $\g_k$.

The quiver $\g_k$ can also be understood in terms of partitions as follows.
Recall that $\es(k)=(2k+2,2k+4,\dots)$ is the vacuum object in $\hb_k$.
There is a one-to-one correspondence between $V(\g_k) = \hb_k$ and the set of all partitions $\op{Par}$:
$$\begin{array}{ccc}
\hb_k & \ra & \op{Par} \\
\mf{x} & \mapsto & \frac{1}{2}(\es(k)-\mf{x}).
\end{array}$$
There is an arrow from $\lambda$ to $\mu$ in $\g_k$ if $\lambda$ is a sub-partition of $\mu$ and $|\lambda|=|\mu|-1$.
The relations on arrows are:
\be
\item a composition of two arrows is zero if both arrows add one on the same entry of a partition;
\item any oriented square commutes.
\ee
A part of the quiver $\g_k$ is given as follows:
$$
\xymatrix{
            &           &                  &(3) \\
            &                    &    (2) \ar[dr]\ar[ur] & \\
(0)\ar[r]   & (1) \ar[ur]\ar[dr] \ar@/^2pc/[uurr]^{= 0}&                & (2,1) \\
            &                    &   (1,1) \ar[ur]\ar[dr]& \\
             &            &                & (1,1,1)
}$$

\subsection{The DG category $DGP(R)$}
We refer to \cite[Section 10]{BL} for an introduction to DG modules and {\em projective} DG modules, and to \cite{Ke} for an introduction to DG categories and their homology categories.
For the DG algebra $R$, let $DG(R)$ denote the DG category of DG left $R$-modules.
The cohomology $H(R)$ is viewed as a DG algebra with a trivial differential.
We are interested in full subcategories of $DG(R)$ and $DG(H(R))$ generated by some projective DG modules.
We refer to \cite{Tian1} for more detail about the construction of subcategories.

\begin{defn}
\be
\item Let $DGP(R)$ be the smallest full subcategory of $DG(R)$ which contains the projective DG $R$-modules $\{P(\mf{x})=R\cdot e(\mf{x}) ~|~ \mf{x} \in \cb\}$ and is closed under the cohomological grading shift functor $[1]$ and taking mapping cones.
\item Let $DGP(H(R))$ be the smallest full subcategory of $DG(H(R))$ which contains the projective DG $H(R)$-modules $\{PH(\mf{x})=H(R)\cdot e(\mf{x}) ~|~ \mf{x} \in \hb\}$ and is closed under the cohomological grading shift functor $[1]$ and taking mapping cones.
\ee
\end{defn}

Since $R$ and $H(R)$ are quasi-isomorphic by Proposition \ref{HR}, their $0$th homology categories $H(DG(R))$ and $H(DG(H(R)))$ are equivalent as triangulated categories.
By the same method in \cite[Lemma 2.29]{Tian1}, we have the following equivalence between their subcategories:

\begin{lemma} \label{equi}
The triangulated categories $H(DGP(R))$ and $H(DGP(H(R)))$ are equivalent.
Hence, their Grothendieck groups are isomorphic.
\end{lemma}

Since $H(R)$ has a trivial differential, the $0$th homology category $H(DGP(H(R)))$ is just the homotopy category of bounded complexes of projective modules in $\{PH(\mf{x}) ~|~ \mf{x} \in \hb\}$.
We use the path algebra interpretation of $H(R_k)$ to compute the Grothendieck group of $H(DGP(H(R)))$, where $H(R)=\bigoplus_kH(R_k)$.

\begin{lemma} \label{k0}
The Grothendieck group of $H(DGP(H(R)))$ is isomorphic to $V=\Z \lan \hb \ran$.
\end{lemma}
\begin{proof}
It is obvious that the Grothendieck group is generated by the classes $\{[PH(\mf{x})] ~|~ \mf{x} \in \hb\}$.
It suffices to show that there are not other relations.
Suppose that there is a distinguished triangle $\Delta$ in $H(DGP(H(R_k)))$.
Then $\Delta$ only involves $\{PH(\mf{x})\}$ for finitely many $\mf{x} \in \hb$.
Let $m=\op{max}\{\op{E}_k(\mf{x}) ~|~ PH(\mf{x}) ~\mbox{in}~ \Delta\}$ denote the maximal energy in $\Delta$ which is a finite number.

Let $\hb_k^m=\{\mf{x} \in \hb_k, \op{E}(\mf{x})\leq m\}$ consist of objects with energy bounded above by $m$.
Consider a finite dimensional subalgebra $H(R_k)^m$ of $H(R_k)$ generated by morphisms between objects in $\hb_k^m$.
Similarly, consider a full sub-quiver $\g_k^m$ of $\g_k$ with its vertex set $V(\g_k^m)=\hb_k^m$.
Then the subalgebra $H(R_k)^m$ is a quotient of the path algebra $\F\g_k^m$ of a finite quiver $\g_k^m$.
Since $H(DGP(H(R_k)^m))$ is a full subcategory of $H(DGP(H(R_k)))$, the distinguished triangle $\Delta$ can be viewed in $H(DGP(H(R_k)^m))$.
From the classical result in representation theory of quivers \cite{ASS}, the Grothendieck group of $H(DGP(H(R_k)^m))$ is isomorphic to the free abelian group over $\hb_k^m$.
Hence, there is no more relation in the Grothendieck groups of $H(DGP(H(R_k)^m))$ as well as $H(DGP(H(R_k)))$.
\end{proof}

\begin{lemma} \label{dgpv}
The DG categories $\cv$ and $DGP(R)$ are equivalent. In particular, $K_0(\cv)\cong V$.
\end{lemma}
\begin{proof}
Recall that $\cv$ and $\cv'$ are equivalent by Remark \ref{vv'}.
We show that $\cv'$ is isomorphic to $DGP(R)$.
Note that $DGP(R)$ is actually the category of one-sided twisted complexes of projective modules in $\{P(\mf{x}) ~|~ \mf{x} \in \cb\}$ since taking mapping cones in $DGP(R)$ is the same as constructing one-sided twisted complexes.

A functor $\chi: \cv' \ra DGP(R)$ is defined on the elementary objects as:
\be
\item $\chi(\mf{x})=P(\mf{x})$ for $\mf{x} \in \cb$;
\item For $f \in \Hom_{\cv'}(\mf{x},\mf{y})$, $\chi(f)$ is the right multiplication with $f \in R: P(\mf{x}) \xra{\times f} P(\mf{y})$, where $P(\mf{x})$ and $P(\mf{y})$ are viewed as subspaces of $R$.
\ee
and extended to one-sided twisted complexes.
The functor $\chi$ gives $\Hom_{\cv'}(\mf{x},\mf{y}) \cong \Hom_{DGP(R)}(P(\mf{x}), P(\mf{y}))$.
Hence $\chi$ is an isomorphism of the categories.
\end{proof}

\section{The categorical action of $\cl$ on $DGP(R)$}
We construct a family of DG $R$-bimodules $T_i \in DG(R\ot R^{op})$ for $i \in \Z$.
We show that there is a functor $\tau: \cl \ra DG(R\ot R^{op})$ of monoidal DG categories satisfying:
\be
\item $\tau(\es)=R$ as an $R$-bimodule;
\item $\tau(i)=T_i$ for $i \in \Z$;
\item $\tau$ maps morphisms of $\cl$ to homomorphisms of bimodules;
\item $\tau$ maps the monoidal functor in $\cl$ to the functor of tensoring $R$-bimodules over $R$;
\item $\tau$ maps relations of $\cl$ to identities of homomorphisms;
\item $\tau$ maps differentials of morphisms to differentials of homomorphisms.
\ee
In Section 6.1, we give the definition of $T_i$.
In Section 6.2, we define the homomorphisms for the elementary diagrams in $\cl$ and show that some of the relations in $\cl$ are mapped to identities of homomorphisms.
Cases for other relations are proved in Section 6.3.
In Section 6.4, we show that tensoring with the bimodule $T_i$ maps projectives in $DGP(R)$ to $DGP(R)$.
Hence, the functor $\tau$ defines a categorical action $\cl \times DGP(R) \ra DGP(R)$.
The induced action on their homology categories descends to the linear action $Cl \times V \ra V$ on the level of Grothendieck groups.

\subsection{DG $R$-bimodules}
In this subsection we define the DG $R$-bimodules $T_i$ for $i \in \Z$.

Recall that $(2i)\cdot\mf{x}$ denotes the horizontal stacking of $(2i)$ and $\mf{x} \in \E(\cv)$.
It is easy to see that $(2i)\cdot\mf{x} \in \E(\cv)_{k-1}$ for $\mf{x} \in \E(\cv)_k$.
There is a map $$\imath_{i}: \Hom(\mf{x},\mf{x'}) \rightarrow \Hom((2i)\cdot\mf{x}, ~(2i)\cdot\mf{x'}),$$
given by adding a vertical strand with label $2i$ to the left.
Note that $\imath_{i}$ is an inclusion since canonical diagrams generate morphism sets by Proposition \ref{hom}.

\begin{defn}
For $i \in \Z$, let $\tau(2i), \tau(2i-1)$ be the DG $R$-bimodules:
\begin{gather*}
\tau(2i):=T_{2i}=\bigoplus_{k=-\infty}^{+\infty}T_{2i}(k)=\bigoplus_{k=-\infty}^{+\infty}\left(\bigoplus_{\mf{x}\in\cb_{k-1}, \mf{y}\in\cb_k}\Hom(\mf{x},~(2i)\cdot\mf{y})\right),\\
\tau(2i-1):=T_{2i-1}=\bigoplus_{k=-\infty}^{+\infty}T_{2i-1}(k)=\bigoplus_{k=-\infty}^{+\infty}\left(\bigoplus_{\mf{z}\in\cb_{k+1}, \mf{y}\in\cb_k}\Hom((2i)\cdot\mf{z},~\mf{y})\right),
\end{gather*}
where the $T_{2i}(k)$ are $(R_{k-1},R_k)$-bimodules and $T_{2i-1}(k)$ are $(R_{k+1},R_k)$-bimodules via the composition in $\cv$ and the inclusions $\imath_{i}$.
\end{defn}

\begin{rmk}
The definition of $\tau(2i)$ is motivated from a categorical action on $\cv$ which is given by horizontally stacking a vertical strand with label $2i$ to the left of morphisms of $\cv$.
See Lemma \ref{tensor} for the effect of tensoring with $\tau(2i)$ on $DGP(R)$.
Then $\tau(2i-1)$ is right adjoint to $\tau(2i)$ as endofunctors on $DGP(R)$.
\end{rmk}

\subsection{Homomorphisms of $R$-bimodules}
In this subsection we define homomorphisms $\tau(f)$ between the bimodules, where $f$ are the elementary diagrams in $\cl$.
The definitions of $\tau(f)$ will be given on direct summands of the bimodules according to the decomposition $R=\oplus_k R_k$.
The main tool to understand the bimodules is the calculation on canonical diagrams.

\subsubsection{Morphisms between $(\es)$, $(2i,2i-1)$ and $(2i-1,2i)$}
We will define $\tau(f)$ for $f$ in a set $\cal{ED}_1$ of some elementary diagrams
$$\cal{ED}_1=\{cr_{2i,2i-1}, ~cr_{2i-1,2i}, ~dap_{2i-1,2i}, ~cup_{2i-1,2i}, ~cap_{2i,2i-1}\}.$$

\vspace{.2cm}
\n $\bullet$ {\bf Notation.}
For objects in $\cb$, we fix the notation $\mf{x} \in \cb_{k-1}, \mf{y} \in \cb_{k}, \mf{z} \in \cb_{k+1}$.
Let $\ot_{k}=\ot_{R_k}$ for simplicity.
We use the notation $f*g$ for compositions in $\cv$ or $DG(R\ot R^{op})$:
$$\begin{array}{ccccc}
\Hom(\mf{x},\mf{y}) &\times &\Hom(\mf{y},\mf{z}) &\ra &\Hom(\mf{x},\mf{z}) \\
f & , & g & \mapsto & f * g
\end{array}
$$

\n $\bullet$ {\bf Preparation.}
The following lemma is motivated from the isomorphism in $\cl$: $(2i-1,2i) \simeq \{(2i,2i-1) \xra{cap_{2i,2i-1}} (\es)\}$ in Lemma \ref{K0cl}.
\begin{lemma} \label{12=21+es}
There is an isomorphism of $R_k$-bimodules:
$$T_{2i-1}(k-1) \ot_{k-1} T_{2i}(k) \cong T_{2i}(k+1) \ot_{k+1} T_{2i-1}(k)[1] ~\oplus~ R_k.$$
\end{lemma}
\begin{proof}
By definition,
\begin{align*}
& T_{2i-1}(k-1) \ot_{k-1} T_{2i}(k) \\
= & \left(\bigoplus_{\mf{y}\in\cb_{k}, \mf{x}\in\cb_{k-1}}\Hom((2i)\cdot\mf{y},~\mf{x})\right) \ot_{k-1} \left(\bigoplus_{\mf{x'}\in\cb_{k-1}, \mf{y'}\in\cb_{k}}\Hom(\mf{x'},~(2i)\cdot\mf{y'})\right) \\
= & \left(\bigoplus_{\mf{y}\in\cb_{k}, \mf{y'}\in\cb_{k}}\Hom((2i)\cdot\mf{y},~(2i)\cdot\mf{y'})\right),
\end{align*}
which is generated by all canonical diagrams from $(2i)\cdot\mf{y}$ to $(2i)\cdot\mf{y'}$.
Similarly,
\begin{align*}
& T_{2i}(k+1) \ot_{k+1} T_{2i-1}(k) \\
= & \left(\bigoplus_{\mf{y}\in\cb_{k}, \mf{z}\in\cb_{k+1}}\Hom(\mf{y},~(2i)\cdot\mf{z})\right) \ot_{k+1} \left(\bigoplus_{\mf{z'}\in\cb_{k+1}, \mf{y'}\in\cb_{k}}\Hom((2i)\cdot\mf{z'},~\mf{y'})\right).
\end{align*}

By Remark \ref{rmkcan} the set $D$ of canonical diagrams in $\Hom((2i)\cdot\mf{y}, ~(2i)\cdot\mf{y'})$ is a disjoint union of two subsets as shown in Figure \ref{c1}:
$$\begin{array}{rl}
 D_1=&\{ ~\mbox{canonical diagrams without crossing in the first strip}~ \} \\
 =&\{id_{2i}\cdot f ~~|~~ \mbox{canonical}~~ f \in \Hom(\mf{y},\mf{y'})\}, \\
D_2=&\{ ~\mbox{canonical diagrams with one crossing in the first strip}~\} \\
=&\left\{(id_{2i}\cdot g)* (cr_{2i,2i}\cdot id_{\mf{z}}) * (id_{2i}\cdot h) ~~\left|~~ \begin{array}{c} \mbox{canonical}~~g \in \Hom(\mf{y},(2i)\cdot\mf{z}),\\ ~h \in \Hom((2i)\cdot\mf{z},\mf{y'}),~~\mbox{for all}~~\mf{z}\end{array} \right.\right\}/\sim,
\end{array}
$$
where ``$\sim$" is modulo the relation of transporting diagrams along $\mf{z}$.

\begin{figure}[h]
\begin{overpic}
[scale=0.2]{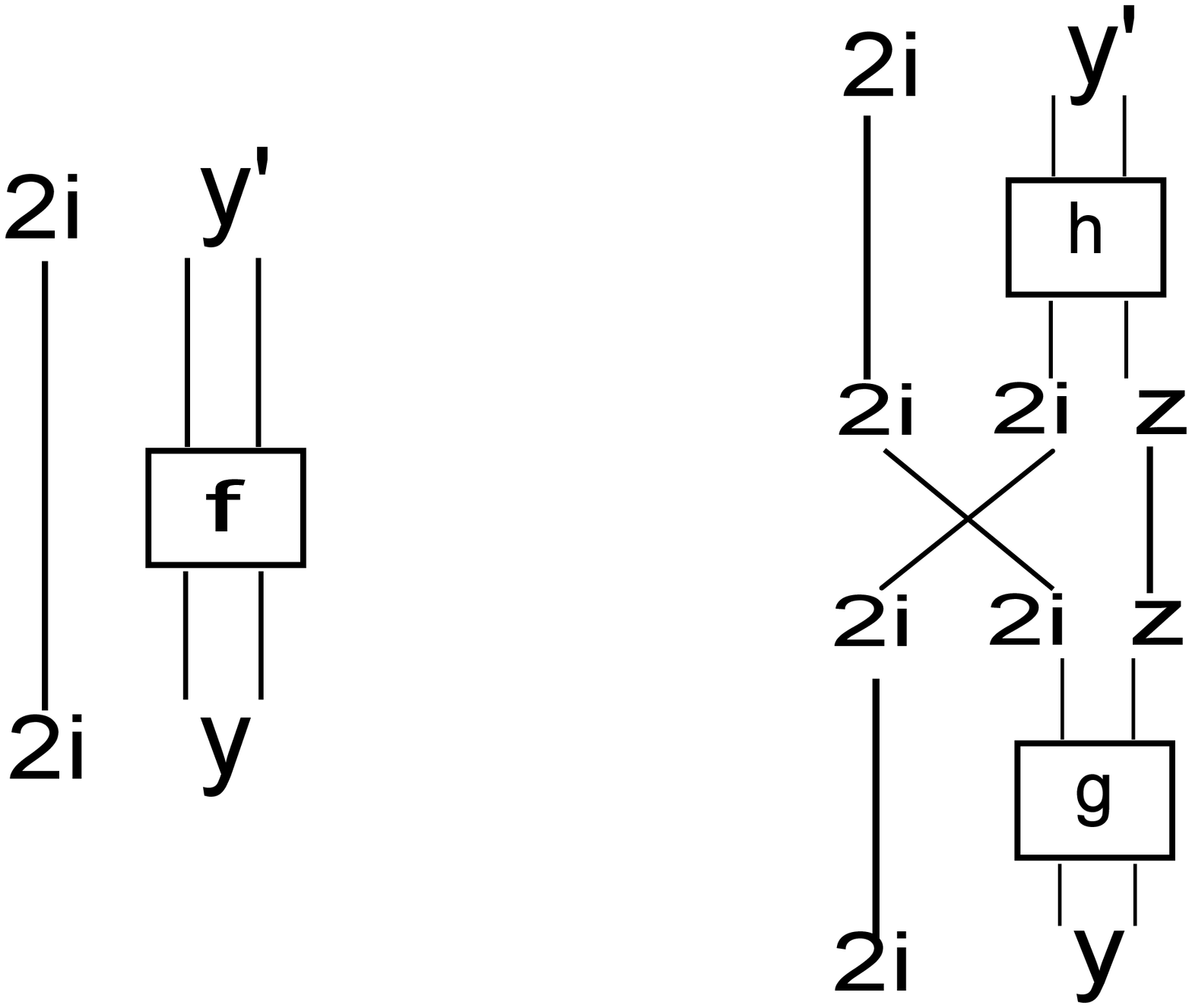}
\end{overpic}
\caption{Two types of canonical diagrams from $(2i)\cdot\mf{y}$ to $(2i)\cdot\mf{y'}$.}
\label{c1}
\end{figure}
It is easy to see that the sub-bimodule generated by $D_1$ is isomorphic to the algebra $R_k$ as $R_k$-bimodules.
On the other hand, the sub-bimodule generated by $D_2$ is isomorphic to $T_{2i}(k+1) \ot_{k+1} T_{2i-1}(k)[1]$ by removing the crossing $cr_{2i,2i}\cdot id_{\mf{z}}$.
The disjoint union $D=D_1 \sqcup D_2$ gives the corresponding decomposition of $T_{2i-1}(k-1) \ot_{k-1} T_{2i}(k)$.
\end{proof}

\begin{rmk}
This is not a direct sum of DG bimodules since there is a nontrivial differential between the two summands.
\end{rmk}

\n $\bullet$ {\bf Definition of $\tau(f)$, Part I.}
For $f \in \cal{ED}_1 \backslash \{cap_{2i,2i-1}\}$, $\tau(f)$ is one of the projections or inclusions with respect to the isomorphism in Lemma \ref{12=21+es}.
We use the notation in $D=D_1 \sqcup D_2$ as the generators of the bimodules.

\n(1) The homomorphism for $cr_{2i,2i-1} \in \Hom((2i-1,2i),(2i,2i-1))$ is a projection:
$$\begin{array}{cccc}
\tau(cr_{2i,2i-1})=\op{pr}_2: & T_{2i-1}(k-1) \ot_{k-1} T_{2i}(k) & \ra & T_{2i}(k+1) \ot_{k+1} T_{2i-1}(k) \\
& id_{2i}\cdot f & \mapsto & 0 \\
& (id_{2i}\cdot g)* (cr_{2i,2i}\cdot id_{\mf{z}}) * (id_{2i}\cdot h) & \mapsto & g \ot h.
\end{array}
$$

\vspace{.1cm}
\n(2) The homomorphism for $cr_{2i-1,2i} \in \Hom((2i,2i-1),(2i-1,2i))$ is an inclusion:
$$\begin{array}{cccc}
\tau(cr_{2i-1,2i})=\op{\imath}_2: & T_{2i}(k+1) \ot_{k+1} T_{2i-1}(k) & \ra & T_{2i-1}(k-1) \ot_{k-1} T_{2i}(k) \\
& g \ot h & \mapsto & (id_{2i}\cdot g)* (cr_{2i,2i}\cdot id_{\mf{z}}) * (id_{2i}\cdot h).
\end{array}
$$

\vspace{.1cm}
\n(3) The homomorphism for $dap_{2i-1,2i} \in \Hom((2i-1,2i),(\es))$ is a projection:
$$\begin{array}{cccc}
\tau(dap_{2i-1,2i})=\op{pr}_1: & T_{2i-1}(k-1) \ot_{k-1} T_{2i}(k) & \ra & R_k \\
& id_{2i}\cdot f & \mapsto & f \\
& (id_{2i}\cdot g)* (cr_{2i,2i}\cdot id_{\mf{z}}) * (id_{2i}\cdot h) & \mapsto & 0.
\end{array}
$$

\vspace{.1cm}
\n(4) The homomorphism for $cup_{2i-1,2i} \in \Hom((\es),(2i-1,2i))$ is an inclusion:
$$\begin{array}{cccc}
\tau(cup_{2i-1,2i})=\op{\imath}_1: & R_k & \ra & T_{2i-1}(k-1) \ot_{k-1} T_{2i}(k) \\
& f & \mapsto & id_{2i}\cdot f.
\end{array}
$$

\vspace{.1cm}
\n(5) The homomorphism for $cap_{2i,2i-1} \in \Hom((2i,2i-1),(\es))$ is a multiplication:
$$\begin{array}{cccc}
\tau(cap_{2i,2i-1}): & T_{2i}(k+1) \ot_{k+1} T_{2i-1}(k) & \ra & R_k \\
& g \ot h & \mapsto & g * h .
\end{array}
$$

\vspace{.2cm}
The definitions (1)-(4) are summarized in Figure \ref{c2}.
\begin{figure}[h]
\begin{overpic}
[scale=0.25]{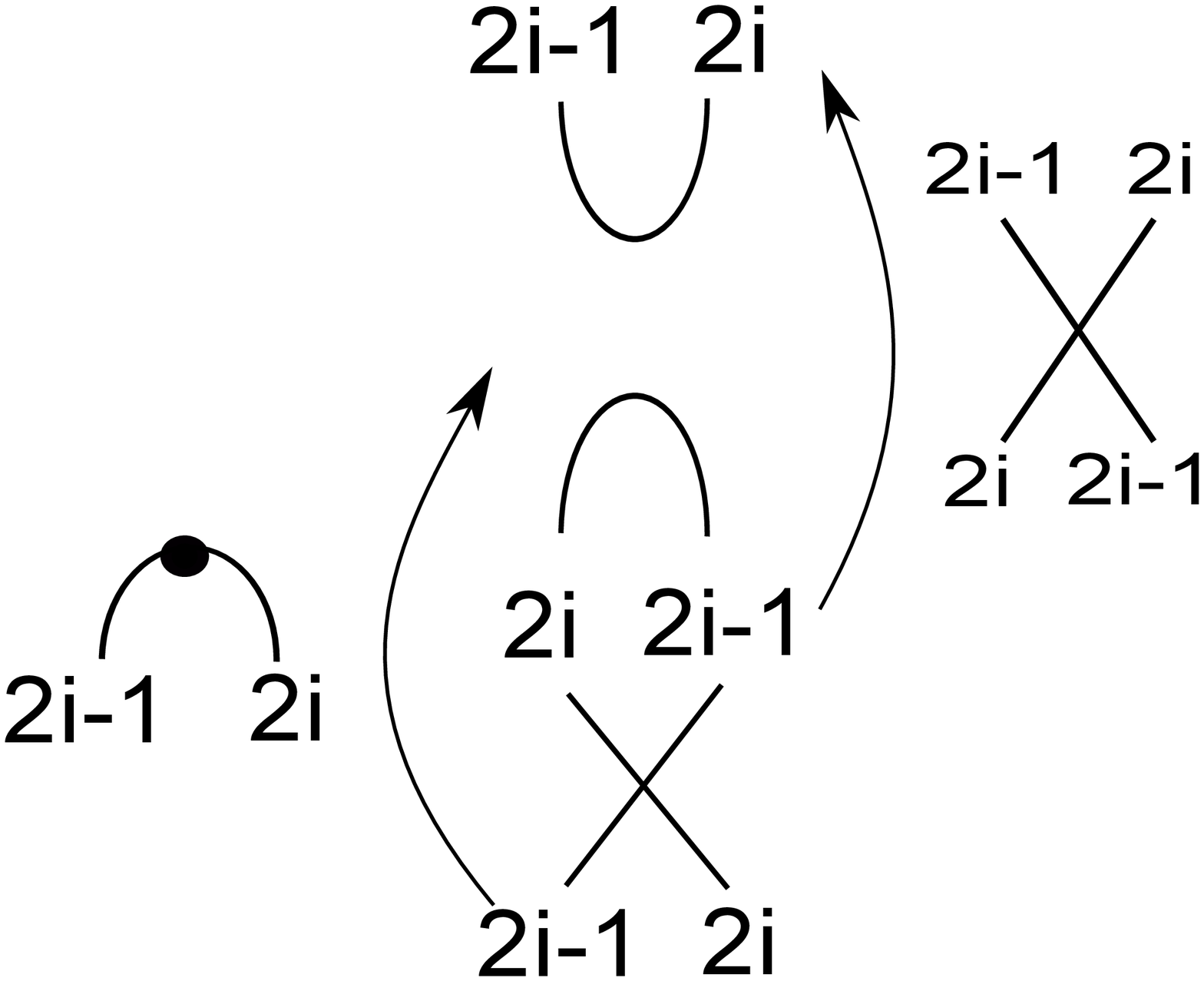}
\put(65,15){$=\op{pr}_2$}
\put(100,52){$=\op{\imath}_2$}
\put(50,52){$\es$}
\put(-15,30){$\op{pr}_1=$}
\put(30,65){$\op{\imath}_1=$}
\end{overpic}
\caption{The projections and inclusions, where the bimodule corresponding to $(2i-1,2i)$ is the sum of the two bimodules corresponding to $(2i,2i-1)$ and $(\es)$.}
\label{c2}
\end{figure}

\begin{lemma} \label{relation12}
The functor $\tau$ maps (R3), (R4-a,b) for $f \in \cal{ED}_1$ to identities, and maps differentials of $f \in \cal{ED}_1$ to differentials of the homomorphisms.
\end{lemma}
\begin{proof}
All the relations in Figure \ref{c3} are mapped to identities which are given by compositions of the projections and inclusions.
\begin{figure}[h]
\begin{overpic}
[scale=0.18]{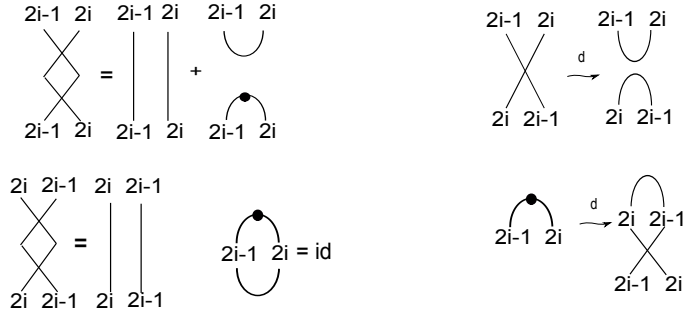}
\end{overpic}
\caption{Relations and nontrivial differentials for morphisms in $\cal{ED}_1$}
\label{c3}
\end{figure}

For $d(cr_{2i-1,2i})=cap_{2i,2i-1} * cup_{2i-1,2i}$, we want to prove that
$$d\circ \tau(cr_{2i-1,2i})+\tau(cr_{2i-1,2i})\circ d=d(\tau(cr_{2i-1,2i}))=\tau(cap_{2i,2i-1} * cup_{2i-1,2i}).$$
The calculation is done in Figure \ref{c4}.
The proof for $d(dap_{2i-1,2i})$ is similar.
\begin{figure}[h]
\begin{overpic}
[scale=0.25]{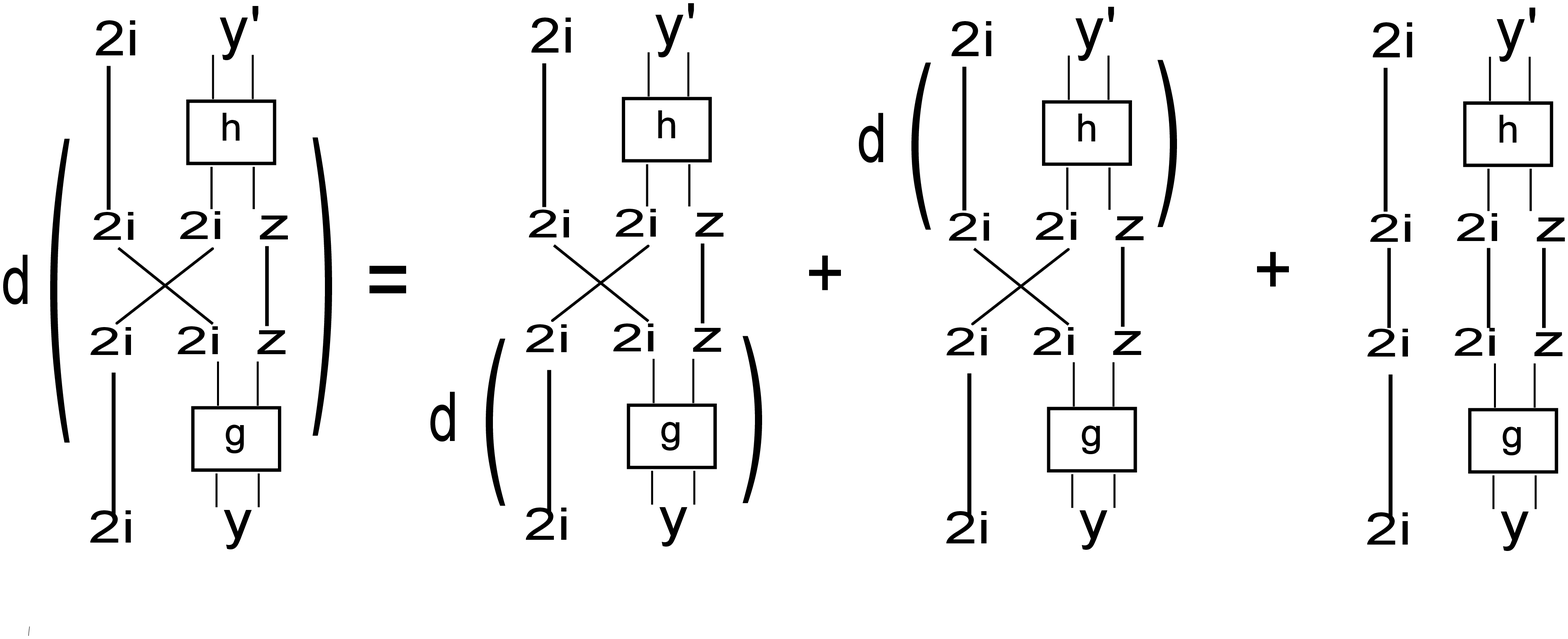}
\put(-12,0){$d(\tau(cr_{2i-1,2i})(g\ot h))=\quad \tau(cr_{2i-1,2i})(d(g\ot h))~+~ \tau(cap_{2i,2i-1} * cup_{2i-1,2i})(g\ot h)$}
\end{overpic}
\caption{The differential of $\tau(cr_{2i-1,2i})$.}
\label{c4}
\end{figure}
\end{proof}

\subsubsection{Morphisms between $(\es)$, $(2i,2i+1)$ and $(2i+1,2i)$}
We will define $\tau(f)$ for $f$ in a set $\cal{ED}_2$ of some elementary diagrams
$$\cal{ED}_2=\{cap_{2i+1,2i}, ~dup_{2i+1,2i}, ~cr_{2i,2i+1}, ~cr_{2i+1,2i}, ~cup_{2i,2i+1}\}.$$

\n $\bullet$ {\bf Preparation.}
The following lemma is motivated from the isomorphism in $\cl$: $(2i+1,2i) \simeq \{(\es) \xra{cup_{2i,2i+1}} (2i,2i+1)\}$ in Lemma \ref{K0cl}.

\begin{lemma} \label{32=23+es}
There is an isomorphism of $R_k$-bimodules:
$$T_{2i+1}(k-1) \ot_{k-1} T_{2i}(k) \cong  R_k ~\oplus~ T_{2i}(k+1) \ot_{k+1} T_{2i+1}(k)[-1].$$
\end{lemma}
\begin{proof}
By definition,
\begin{align*}
& T_{2i+1}(k-1) \ot_{k-1} T_{2i}(k) \\
= & \left(\bigoplus_{\mf{y}\in\cb_{k}, \mf{x}\in\cb_{k-1}}\Hom((2i+2)\cdot\mf{y},~\mf{x})\right) \ot_{k-1} \left(\bigoplus_{\mf{x'}\in\cb_{k-1}, \mf{y'}\in\cb_{k}}\Hom(\mf{x'},~(2i)\cdot\mf{y'})\right) \\
= & \left(\bigoplus_{\mf{y}\in\cb_{k}, \mf{y'}\in\cb_{k}}\Hom((2i+2)\cdot\mf{y},~(2i)\cdot\mf{y'})\right),
\end{align*}
\begin{align*}
& T_{2i}(k+1) \ot_{k+1} T_{2i+1}(k) \\
= & \left(\bigoplus_{\mf{y}\in\cb_{k}, \mf{z}\in\cb_{k+1}}\Hom(\mf{y},~(2i)\cdot\mf{z})\right) \ot_{k+1} \left(\bigoplus_{\mf{z'}\in\cb_{k+1}, \mf{y'}\in\cb_{k}}\Hom((2i+2)\cdot\mf{z'},~\mf{y'})\right).
\end{align*}
By Remark \ref{rmkcan}, the set $D'$ of canonical diagrams in $\Hom((2i+2)\cdot\mf{y}, ~(2i)\cdot\mf{y'})$ is a disjoint union of two subsets as shown in Figure \ref{c5}:
$$\begin{array}{rl}
 D'_1=&\{ ~\mbox{canonical diagrams without crossing in the first strip}~ \} \\
 =&\{dot_{2i}\cdot f ~~|~~ \mbox{canonical}~~ f \in \Hom(\mf{y},\mf{y'})\}, \\
D'_2=&\{ ~\mbox{canonical diagrams with one crossing in the first strip}~\} \\
=&\left\{(id_{2i+2}\cdot g)* (cr_{2i,2i+2}\cdot id_{\mf{z}}) * (id_{2i+2}\cdot h) ~~\left|~~ \begin{array}{c} \mbox{canonical}~~g \in \Hom(\mf{y},(2i)\cdot\mf{z}),\\ ~h \in \Hom((2i+2)\cdot\mf{z},\mf{y'}),~~\mbox{for all}~~\mf{z}\end{array} \right.\right\}/\sim,
\end{array}
$$
where ``$\sim$" is modulo the relation of transporting diagrams along $\mf{z}$.
It gives the corresponding decomposition of $T_{2i+1}(k-1) \ot_{k-1} T_{2i}(k)$.
\end{proof}
\begin{figure}[h]
\begin{overpic}
[scale=0.2]{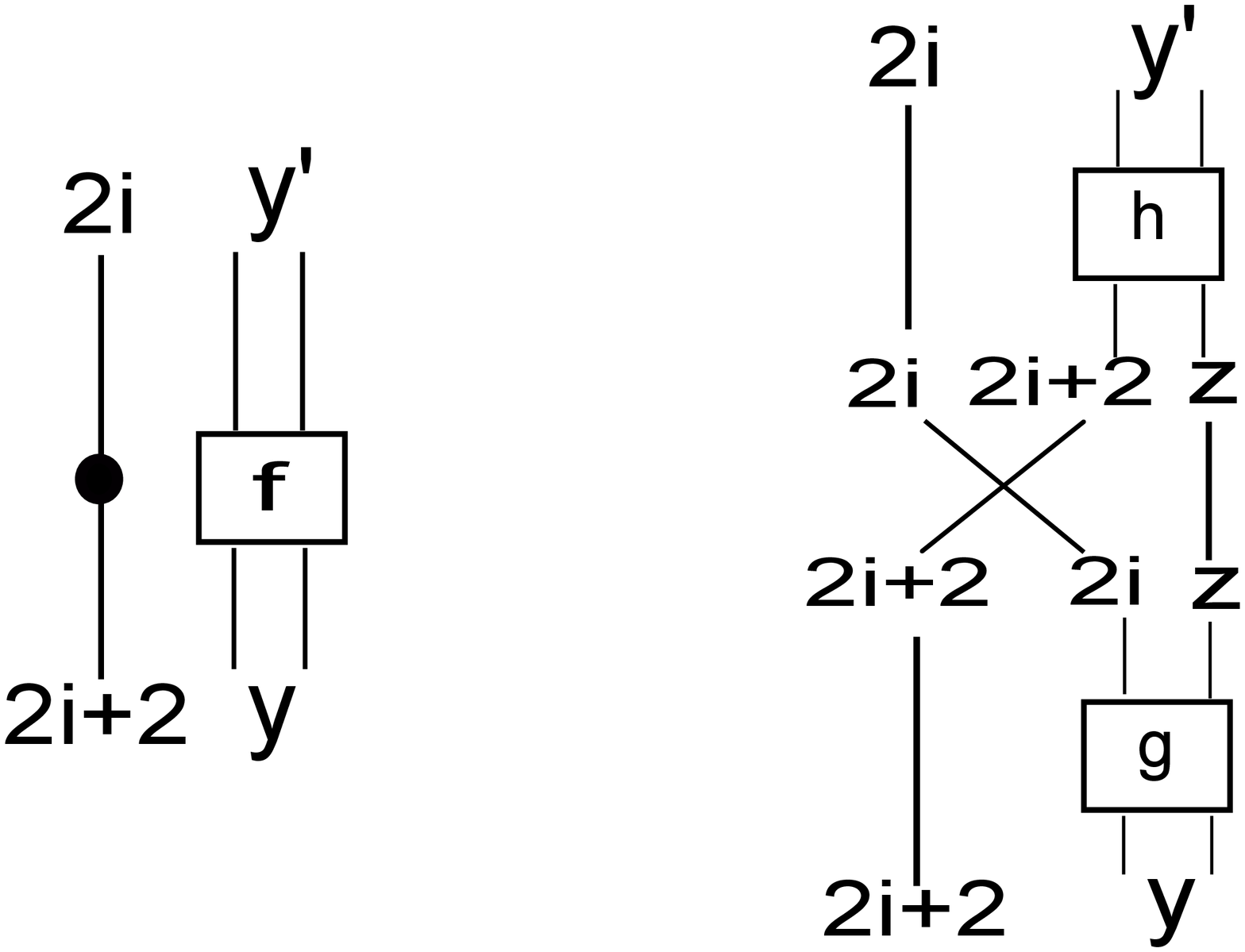}
\end{overpic}
\caption{Two types of canonical diagrams from $(2i+2)\cdot\mf{y}$ to $(2i)\cdot\mf{y'}$}
\label{c5}
\end{figure}

\begin{rmk}
This is not a direct sum of DG bimodules since there is a nontrivial differential between the two summands.
\end{rmk}

\n $\bullet$ {\bf Definition of $\tau$, Part II.}
For $f \in \cal{ED}_2 \backslash \{cup_{2i,2i+1}\}$, $\tau(f)$ is one of the projections or inclusions with respect to the isomorphism in Lemma \ref{32=23+es}.
We use the notation in $D'=D'_1 \sqcup D'_2$ as the generators of the bimodules.

\n(1) The homomorphism for $cap_{2i+1,2i} \in \Hom((2i+1,2i),(\es))$ is a projection:
$$\begin{array}{cccc}
\tau(cap_{2i+1,2i})=\op{pr}_1: & T_{2i+1}(k-1) \ot_{k-1} T_{2i}(k) & \ra & R_k \\
& dot_{2i}\cdot f & \mapsto & f \\
& (id_{2i+2}\cdot g)* (cr_{2i,2i+2}\cdot id_{\mf{z}}) * (id_{2i}\cdot h) & \mapsto & 0.
\end{array}
$$

\vspace{.1cm}
\n(2) The homomorphism for $dup_{2i+1,2i} \in \Hom((\es),(2i+1,2i))$ is an inclusion:
$$\begin{array}{cccc}
\tau(dup_{2i+1,2i})=\op{\imath}_1: & R_k & \ra & T_{2i+1}(k-1) \ot_{k-1} T_{2i}(k) \\
& f & \mapsto & dot_{2i}\cdot f.
\end{array}
$$

\vspace{.1cm}
\n(3) The homomorphism for $cr_{2i,2i+1} \in \Hom((2i+1,2i),(2i,2i+1))$ is a projection:
$$\begin{array}{cccc}
\tau(cr_{2i,2i+1})=\op{pr}_2: & T_{2i+1}(k-1) \ot_{k-1} T_{2i}(k) & \ra & T_{2i}(k+1) \ot_{k+1} T_{2i+1}(k) \\
& dot_{2i}\cdot f & \mapsto & 0 \\
& (id_{2i+2}\cdot g)* (cr_{2i,2i+2}\cdot id_{\mf{z}}) * (id_{2i}\cdot h) & \mapsto & g \ot h.
\end{array}
$$

\vspace{.1cm}
\n(4) The homomorphism for $cr_{2i+1,2i} \in \Hom((2i,2i+1),(2i+1,2i))$ is an inclusion:
$$\begin{array}{cccc}
\tau(cr_{2i+1,2i})=\op{\imath}_2: & T_{2i}(k+1) \ot_{k+1} T_{2i+1}(k) & \ra & T_{2i+1}(k-1) \ot_{k-1} T_{2i}(k) \\
& g \ot h & \mapsto & (id_{2i+2}\cdot g)* (cr_{2i,2i+2}\cdot id_{\mf{z}}) * (id_{2i}\cdot h).
\end{array}
$$

\vspace{.1cm}
\n(5) The homomorphism for $cup_{2i,2i+1} \in \Hom((\es),(2i,2i+1))$ is a composition:
$$\begin{array}{cccc}
\tau(cup_{2i,2i+1}): & R_k & \ra & T_{2i}(k+1) \ot_{k+1} T_{2i-1}(k) \\
& f & \mapsto & \op{pr}_2(d(\op{\imath}_1(f))) .
\end{array}
$$

\vspace{.2cm}
The definitions (1)-(4) are summarized in Figure \ref{c5-1}.
\begin{figure}[h]
\begin{overpic}
[scale=0.25]{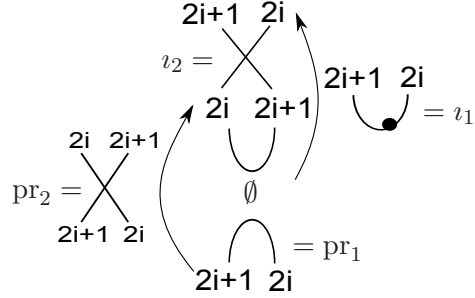}
\put(-10,25){$\op{pr}_2=$}
\put(30,60){$\op{\imath}_2=$}
\put(52,24){$\es$}
\put(65,10){$=\op{pr}_1$}
\put(100,46){$=\op{\imath}_1$}
\end{overpic}
\caption{The projections and inclusions, where the bimodule corresponding to $(2i+1,2i)$ is the sum of the two bimodules corresponding to $(\es)$ and $(2i, 2i+1)$.}
\label{c5-1}
\end{figure}

\begin{lemma} \label{relation32}
The functor $\tau$ maps (R3), (R4-a,b) for $f \in \cal{ED}_2$ to identities, and maps differentials of $f \in \cal{ED}_2$ to differentials of the homomorphisms.
\end{lemma}
\begin{proof}
It is easy to see that $\tau$ maps all relations for $f \in \cal{ED}_2$ to identities which are given by compositions of the projections and inclusions.

For $d(dup_{2i+1,2i})=cup_{2i,2i+1} * cr_{2i+1,2i}$, we want to prove that
$$d\circ \tau(dup_{2i+1,2i})+\tau(dup_{2i+1,2i})\circ d=d(\tau(dup_{2i+1,2i}))=\tau(cup_{2i,2i+1} * cr_{2i+1,2i}).$$
The calculation is done in Figure \ref{c6}.
The proof for $d(cr_{2i,2i+1})$ is similar.
\begin{figure}[h]
\begin{overpic}
[scale=0.25]{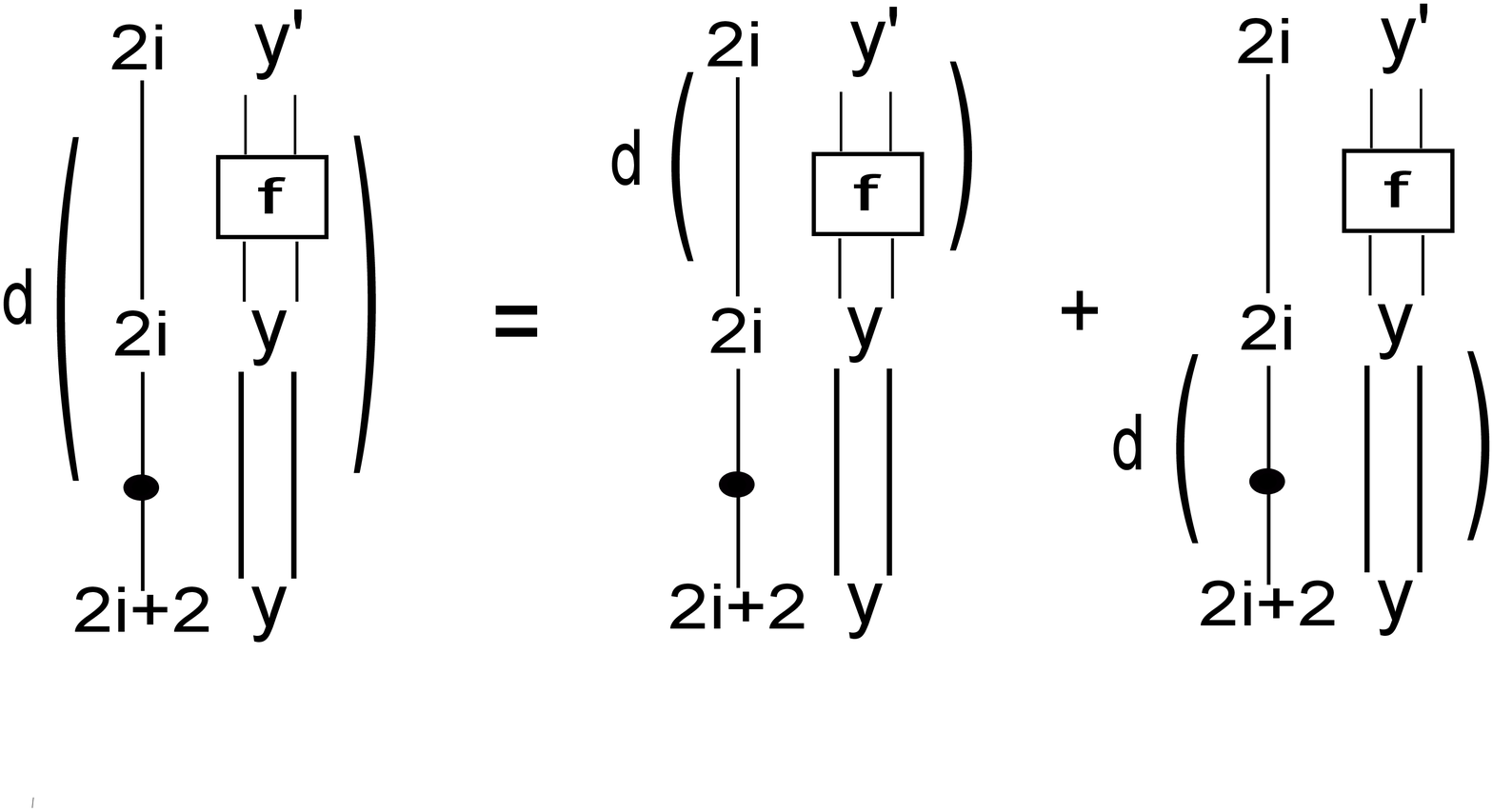}
\put(-15,-2){$d(\tau(dup_{2i+1,2i})(f))=\quad \tau(dup_{2i+1,2i})(d(f))~+~ \tau(cup_{2i,2i+1} * cr_{2i+1,2i})(f)$}
\put(10,5){$d(\op{\imath}_1(f))=\quad\quad \op{\imath}_1(d(f))\quad+ \quad\op{\imath}_2(\op{pr}_2(d(\op{\imath}_1(f))))$}
\end{overpic}
\caption{The differential of $\tau(dup_{2i+1,2i})$.}
\label{c6}
\end{figure}
\end{proof}

\subsubsection{Morphisms of dotted strands}
We define $\tau(f)$ for dotted strands $f$ in
$$\cal{ED}_3=\{dot_{2i}, dot_{2i+1} ~|~ i \in \Z\}.$$

\n $\bullet$ {\bf Definition of $\tau$, Part III.}

\n(1) The homomorphism for $dot_{2i} \in \Hom((2i+2),(2i))$ is a composition with $(dot_{2i}\cdot id_{\mf{y}})$:
$$\begin{array}{cccc}
\tau(dot_{2i}): & T_{2i+2}(k) & \ra & T_{2i}(k) \\
&  f & \mapsto & f * (dot_{2i}\cdot id_{\mf{y}}),
\end{array}
$$
for $f \in \Hom(\mf{x},~(2i+2)\cdot\mf{y})$.

\vspace{.2cm}
\n(2) The homomorphism for $dot_{2i+1} \in \Hom((2i-1),(2i+1))$ is a pre-composition with $(dot_{2i}\cdot id_{\mf{z}})$:
$$\begin{array}{cccc}
\tau(dot_{2i+1}): & T_{2i-1}(k) & \ra & T_{2i+1}(k) \\
&  f & \mapsto &  (dot_{2i}\cdot id_{\mf{z}}) * f ,
\end{array}
$$
for $f \in \Hom((2i)\cdot\mf{z},~\mf{y})$.

\vspace{.2cm}
We discuss two relations about dotted strands: isotopy relation (R1-d); double dot relation (R2).
\begin{lemma} \label{relationdot}
The functor $\tau$ maps the two relations above to identities.
\end{lemma}
\begin{proof}
(1) The lemma is true for (R2) since the corresponding homomorphism is a composition with a diagram which contains a double dot. See Figure \ref{c7}.
\begin{figure}[h]
\begin{overpic}
[scale=0.25]{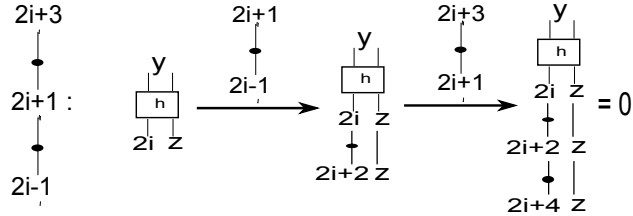}
\end{overpic}
\caption{The homomorphisms for two dotted strands are given by horizontal arrows.}
\label{c7}
\end{figure}

\n(2) For (R1-d), the calculation for $\tau(dap_{2i-1,2i})=\tau((id_{2i-1}\cdot dot_{2i-2}) * cap_{2i-1,2i-2})$ is done in Figure \ref{c8}.
The proofs for other cases are similar.
\begin{figure}[h]
\begin{overpic}
[scale=0.22]{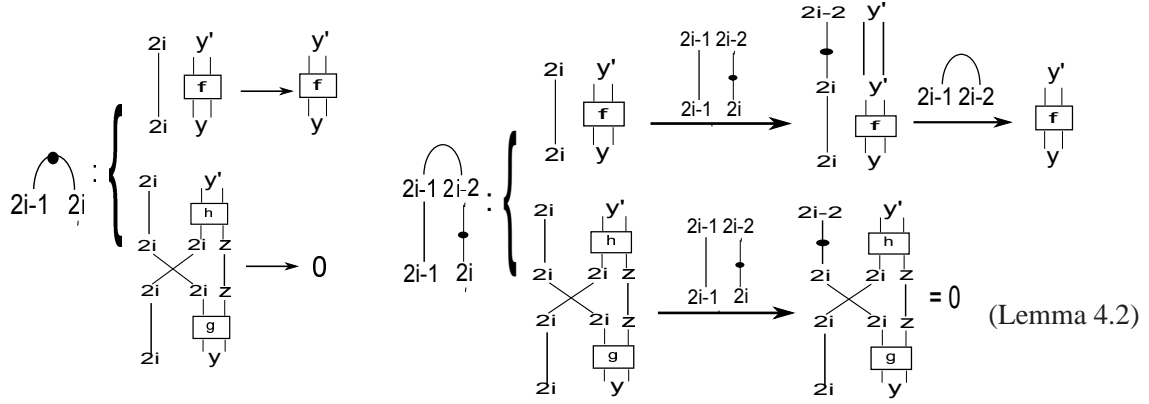}
\put(92,7){(Lemma \ref{LD})}
\end{overpic}
\caption{The homomorphisms for isotopy of a dot.}
\label{c8}
\end{figure}
\end{proof}

\subsubsection{Morphisms of crossings $cr_{i,j}$ for $|i-j|\neq1$}
We define $\tau(f)$ for $f \in \cal{ED}_4$, where $\cal{ED}_4$ consists of crossings of $3$ types depending on the parity of labels:
\be
\item $cr_{2i,2j}$ for all $i,j$;
\item $cr_{2i-1,2j-1}$ for all $i,j$;
\item $cr_{2i,2j-1}$, $cr_{2j-1,2i}$ for $j-i\neq 0,1$.
\ee
It will complete the definitions of $\tau(f)$ for all elementary diagrams $f$ in $\cl$.

\n $\bullet$ {\bf Preparation.}
By definition,
\begin{align*}
& T_{2j}(k) \ot_{k} T_{2i}(k+1) \\
= & \left(\bigoplus_{\mf{x}\in\cb_{k-1}, \mf{y}\in\cb_{k}}\Hom(\mf{x},~(2j)\cdot\mf{y})\right) \ot_{k} \left(\bigoplus_{\mf{y'}\in\cb_{k}, \mf{z}\in\cb_{k+1}}\Hom(\mf{y'},~(2i)\cdot\mf{z})\right) \\
= & \left(\bigoplus_{\mf{x}\in\cb_{k-1}, \mf{z}\in\cb_{k+1}}\Hom(\mf{x},~(2j,2i)\cdot\mf{z})\right),
\end{align*}
\begin{align*}
& T_{2j-1}(k) \ot_{k} T_{2i-1}(k-1) \\
= & \left(\bigoplus_{\mf{z}\in\cb_{k+1}, \mf{y}\in\cb_{k}}\Hom((2j)\cdot\mf{z},~\mf{y})\right) \ot_{k} \left(\bigoplus_{\mf{y'}\in\cb_{k}, \mf{x}\in\cb_{k-1}}\Hom((2i)\cdot\mf{y'},~\mf{x})\right) \\
= & \left(\bigoplus_{\mf{z}\in\cb_{k+1}, \mf{x}\in\cb_{k-1}}\Hom((2i,2j)\cdot\mf{z},~\mf{x})\right),
\end{align*}

\vspace{.2cm}
\n $\bullet$ {\bf Definition of $\tau$, Part IV.}

\n(1) The homomorphism for $cr_{2i,2j} \in \Hom((2j,2i),(2i,2j))$ is a composition with $(cr_{2i,2j}\cdot id_{\mf{z}})$:
$$\begin{array}{cccc}
\tau(cr_{2i,2j}): & T_{2j}(k) \ot_{k} T_{2i}(k+1) & \ra & T_{2i}(k) \ot_{k} T_{2j}(k+1) \\
&  f & \mapsto & f * (cr_{2i,2j}\cdot id_{\mf{z}}),
\end{array}
$$
for $f \in \Hom(\mf{x},~(2j,2i)\cdot\mf{z})$.

\vspace{.2cm}
\n(2) The homomorphism for $cr_{2i-1,2j-1} \in \Hom((2j-1,2i-1),(2i-1,2j-1))$ is a pre-composition with $(cr_{2i,2j}\cdot id_{\mf{z}})$:
$$\begin{array}{cccc}
\tau(cr_{2i-1,2j-1}): & T_{2j-1}(k) \ot_{k} T_{2i-1}(k-1) & \ra & T_{2i-1}(k) \ot_{k} T_{2j-1}(k+1) \\
&  f & \mapsto &  (cr_{2i,2j}\cdot id_{\mf{z}}) * f,
\end{array}
$$
for $f \in \Hom((2i,2j)\cdot\mf{z},~\mf{x})$.

\vspace{.2cm}
\n(3) For $j-i \neq 0,1$, the homomorphism for $cr_{2i,2j-1} \in \Hom((2j-1,2i),(2i,2j-1))$ is:
$$\begin{array}{cccc}
\tau(cr_{2i,2j-1}): & T_{2j-1}(k-1) \ot_{k-1} T_{2i}(k) & \ra & T_{2i}(k+1) \ot_{k+1} T_{2j-1}(k) \\
& f  & \mapsto &  g \ot h,
\end{array}
$$
for $f=(id_{2j}\cdot g)* (cr_{2i,2j}\cdot id_{\mf{z}}) * (id_{2i}\cdot h) \in \Hom((2j)\cdot\mf{y},~(2i)\mf{y'})$, where $g \in \Hom(\mf{y},~(2i)\cdot\mf{z})$ and $h \in \Hom((2j)\cdot\mf{z},\mf{y'})$.

Any canonical diagram in $\Hom((2j)\cdot\mf{y},~(2i)\mf{y'})$ has exactly one crossing in the first strip since the label $2j$ cannot be connected to $2i$ if $j-i\neq 0,1$.
Hence the map $\tau(cr_{2i,2j-1})$ is an isomorphism.

\vspace{.2cm}
\n(4) For $j-i \neq 0,1$, the homomorphism for $cr_{2j-1,2i} \in \Hom((2i,2j-1),(2j-1,2i))$ is $$\tau(cr_{2j-1,2i}):=\tau(cr_{2i,2j-1})^{-1}.$$

\begin{figure}[h]
\begin{overpic}
[scale=0.2]{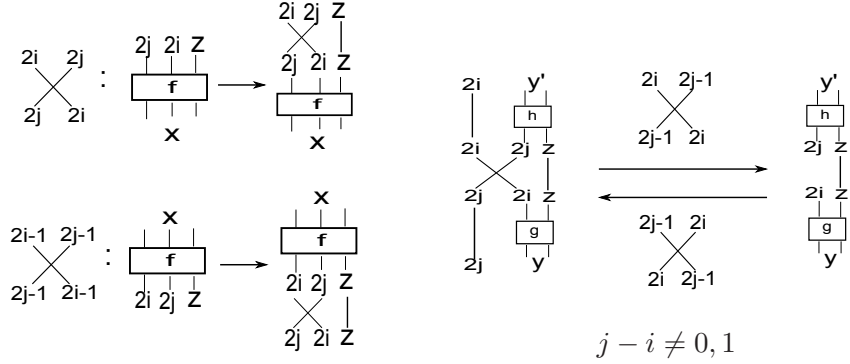}
\put(70,0){$j-i \neq 0,1$}
\end{overpic}
\caption{The homomorphisms for the crossings in $\cal{ED}_4$.}
\label{c9}
\end{figure}

By definition, $\tau$ maps double crossing relations (R4-c,d) for morphisms in $\cal{ED}_4$ to identities, and maps differentials of the morphisms to differentials of the homomorphisms.
The following lemma is an easy consequence of the double crossing relation (R4-d).
\begin{lemma}\label{13=31}
For $|i-j|>1$, there is an isomorphism of $R$-bimodules:
$$T_{j} \ot_{R} T_{i} \cong  T_{i} \ot_{R} T_{j}[\op{gr}(cr_{i,j})].$$
\end{lemma}

\subsection{More relations}
We checked that $\tau$ maps (R1-d), (R2), (R3) and (R4) of $\cl$ to the identities, and maps the differentials of all elementary diagrams to the differentials of the homomorphisms in Lemmas \ref{relation12}, \ref{relation32} and \ref{relationdot}.
The proofs for (R1-a) and (R1-e) easily follow from the definition of bimodules.
It remains to check for
\be
\item (R1-b): isotopy of a single strand;
\item (R1-c): isotopy of a crossing;
\item (R5): triple intersection moves;
\item (R6): dot slide.
\ee

We want to show that $\tau$ maps two sides of any relation to an identity of the homomorphisms.
It suffices to verify that the homomorphisms agree on generators of the bimodules.
We will check some examples using specific labels.
Cases in general are similar.

We present calculations for two examples on (R1-b) in Figures \ref{r1}, \ref{r2},
for one example on (R1-c) in Figure \ref{r3},
for two examples on (R6) in Figures \ref{r4}, \ref{r4-1},
and for one example on (R5) in Figure \ref{r5}.

\vspace{.4cm}
\begin{figure}[h]
\begin{overpic}
[scale=0.2]{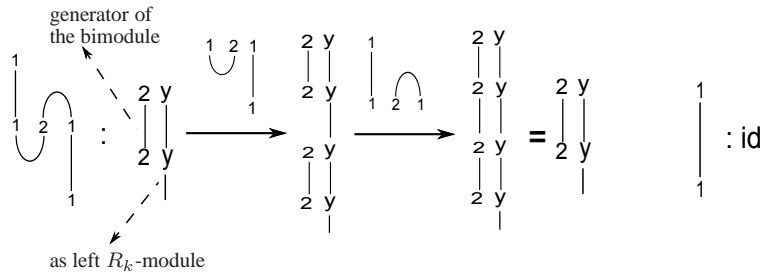}
\put(5,30){{\scriptsize generator of}}
\put(5,27){{\scriptsize the bimodule}}
\put(5,-3){{\scriptsize as left $R_k$-module}}
\end{overpic}
\caption{Isotopy of a single strand}
\label{r1}
\end{figure}

\begin{figure}[h]
\begin{overpic}
[scale=0.2]{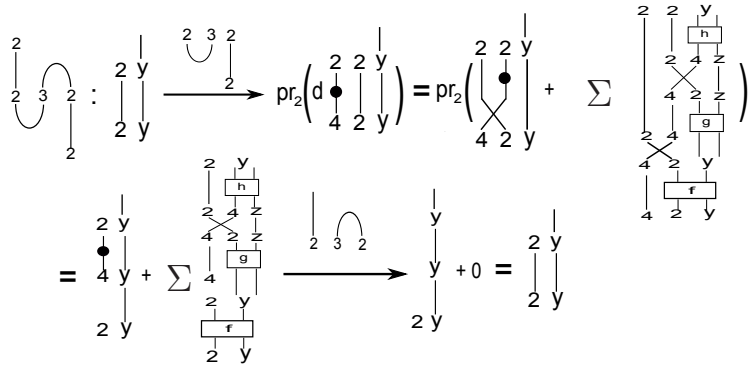}
\put(78,35){\Large{$\Sigma$}}
\put(21,10){\Large{$\Sigma$}}
\end{overpic}
\caption{Isotopy of a strand: the second homomorphism kills all terms but one.}
\label{r2}
\end{figure}

\begin{figure}[h]
\begin{overpic}
[scale=0.19]{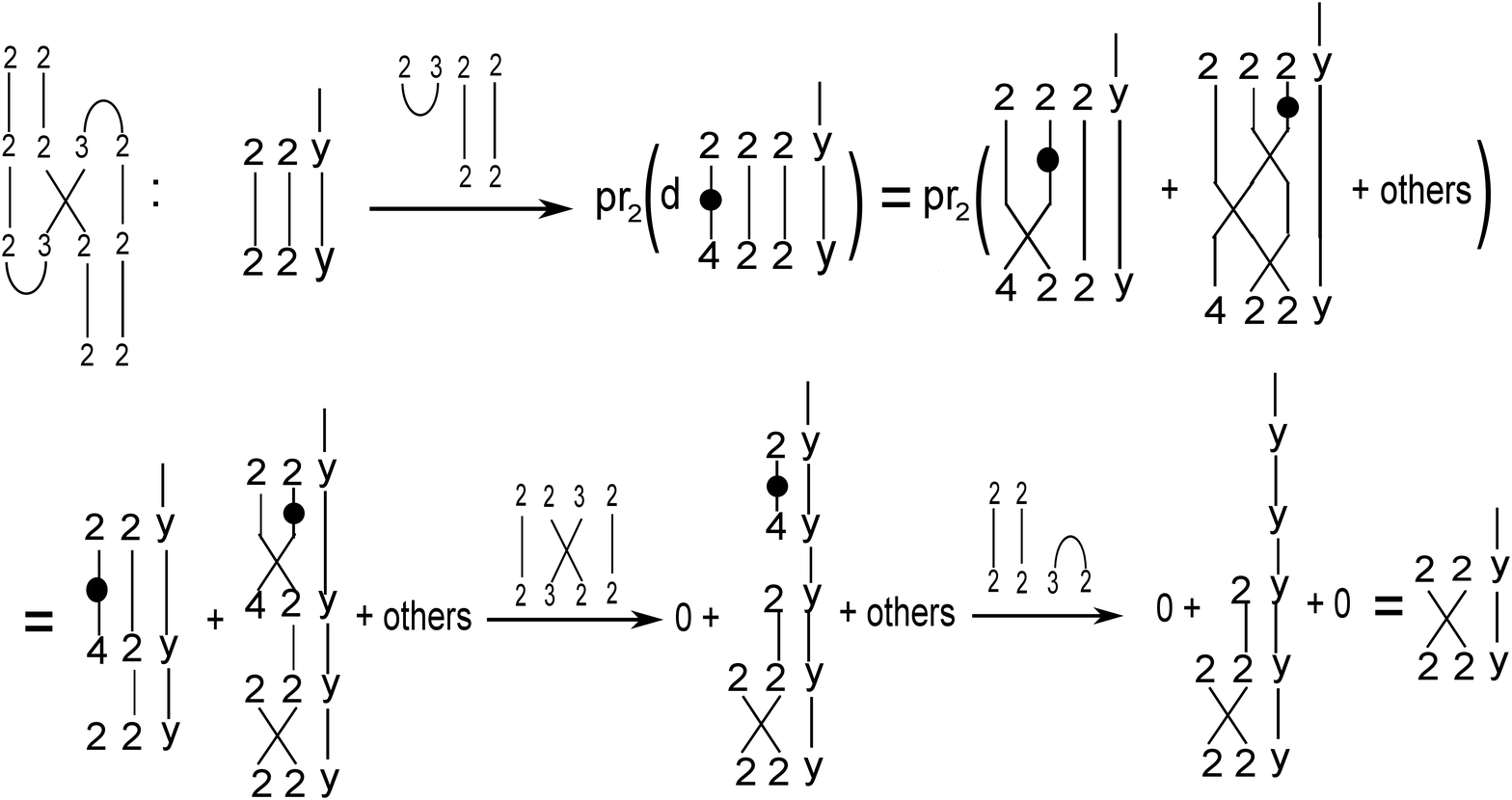}
\end{overpic}
\caption{Isotopy of a crossing: the second homomorphism kills the first term and the last homomorphism kills all other terms but one.}
\label{r3}
\end{figure}

\begin{figure}[h]
\begin{overpic}
[scale=0.22]{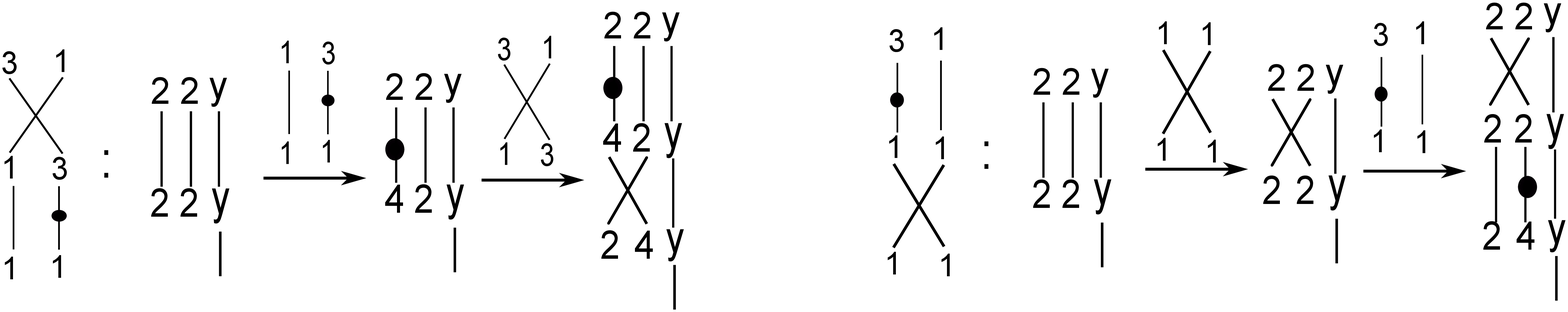}
\end{overpic}
\caption{Dot slide: (L5) of $\cv$ is used.}
\label{r4}
\end{figure}

\begin{figure}[h]
\begin{overpic}
[scale=0.2]{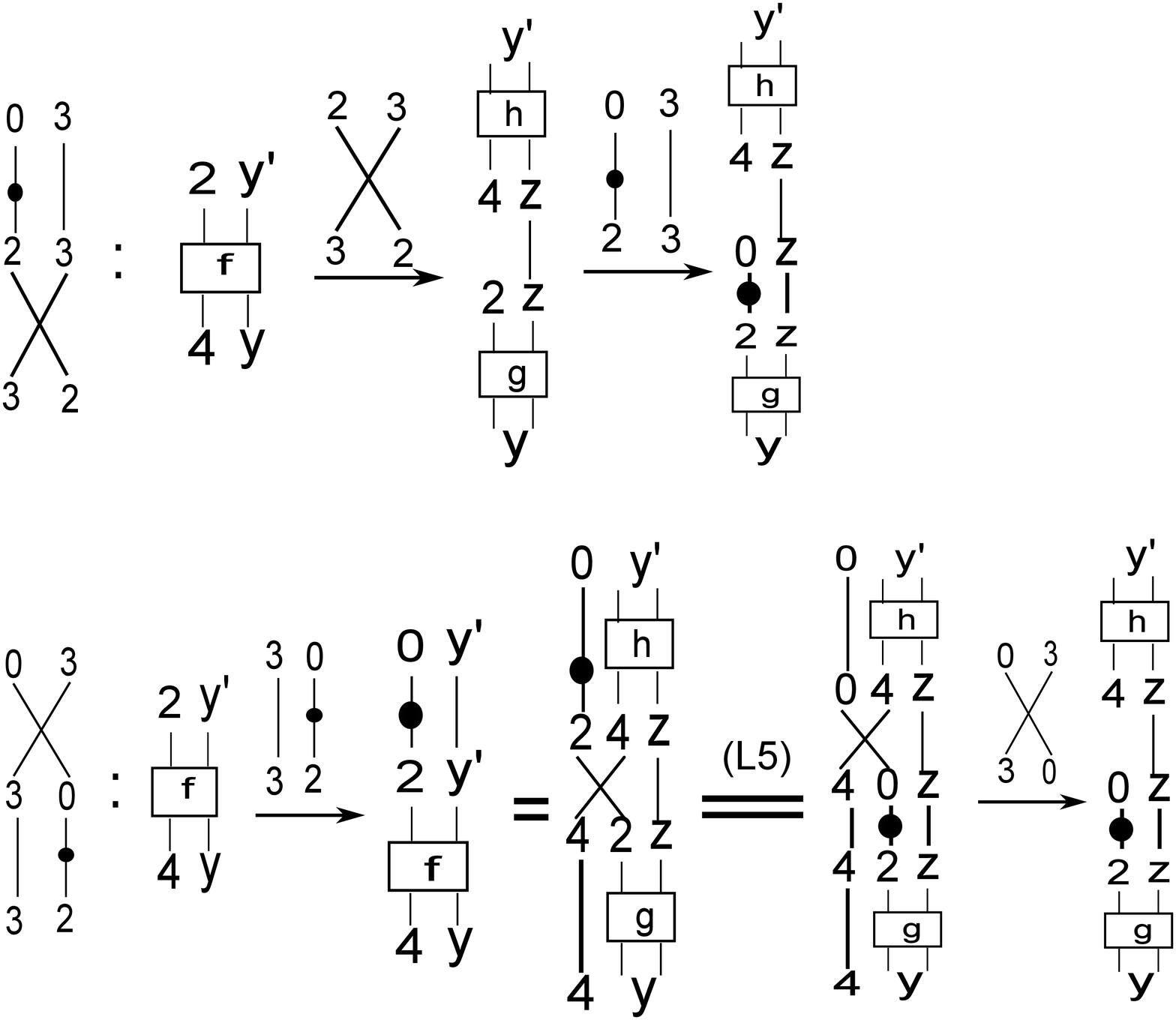}
\end{overpic}
\caption{Dot slide: (L5) of $\cv$ is used.}
\label{r4-1}
\end{figure}

\begin{figure}[h]
\begin{overpic}
[scale=0.22]{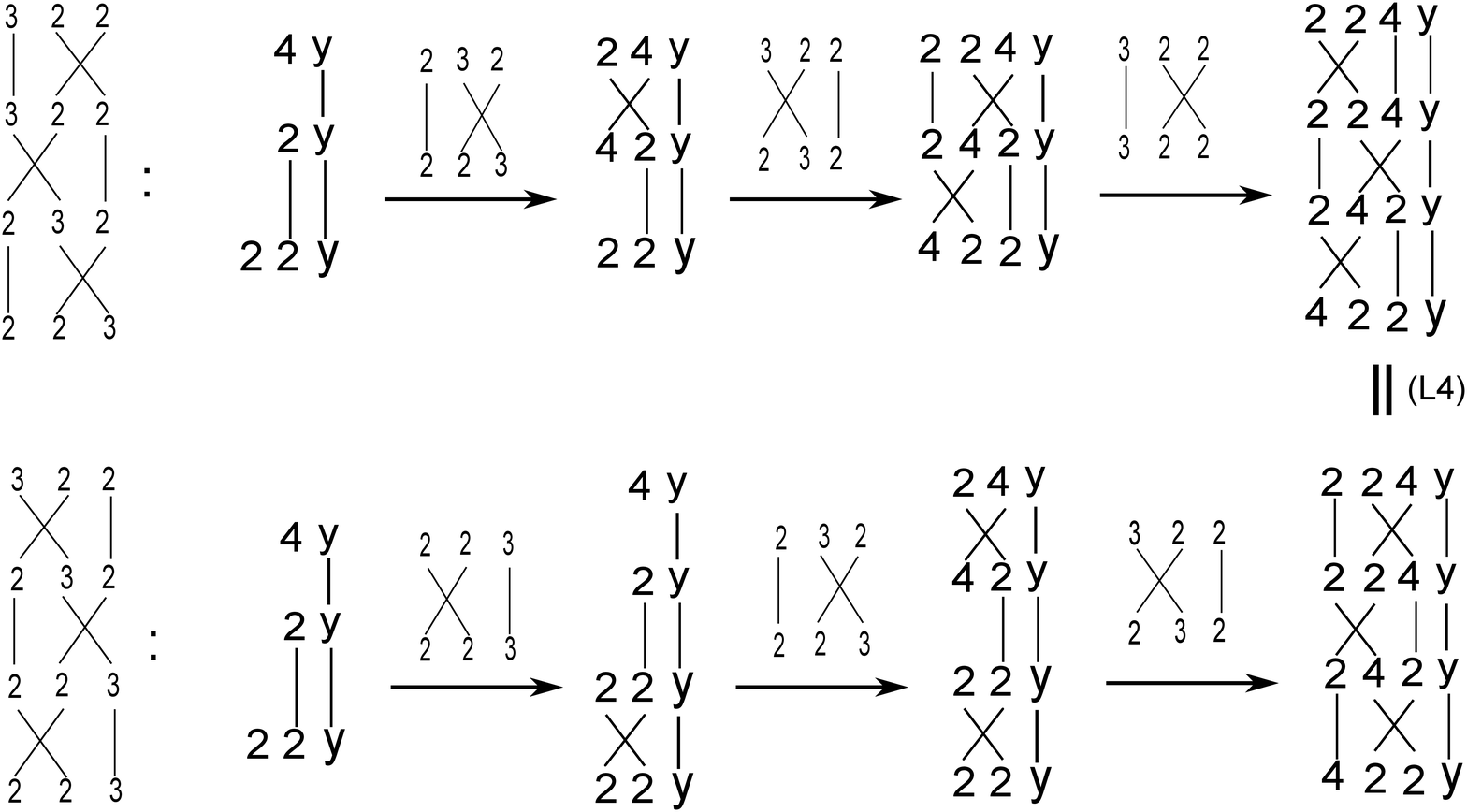}
\end{overpic}
\caption{Triple intersection moves: (L4) of $\cv$ is used.}
\label{r5}
\end{figure}

Finally, we proved that the functor $\tau: \cl \ra DG(R\ot R^{op})$ is well-defined.

\subsection{The categorical action of $\cl$}
We show that tensoring with $T_i$ over $R$ maps the projective DG (left) $R$-modules $P(\mf{x})$ into $DGP(R)$, for $\mf{x} \in \cb$.

\begin{lemma} \label{tensor}
(1) $T_{2i}(k) \ot_{k} P(\mf{y}) ~\cong ~ P(N((2i)\cdot \mf{y}))[\eta((2i)\cdot \mf{y}))]$, for $\mf{y}\in \cb_k$.

\n(2) $T_{2i-1}(k) \ot_{k} P(\es(k)) = 0$, if $i \leq k$.

\n(3) $T_{2i-1} \ot P(\mf{y}) \in DGP(R)$, for $\mf{y}\in \cb_k$.
\end{lemma}
\begin{proof}
(1) Recall that $P(\mf{y})=R_k\cdot e(\mf{y})$ is generated by all diagrams ending at $\mf{y}$ in $\cv$.
$$T_{2i}(k) \ot_{k} P(\mf{y})=\bigoplus_{\mf{x}\in\cb_{k-1}}\Hom(\mf{x},~(2i)\cdot\mf{y})$$
is generated by all diagrams ending at $(2i)\cdot \mf{y}$.
As DG left $R_{k-1}$-modules, it is isomorphic to $P(N((2i)\cdot \mf{y}))[\eta((2i)\cdot \mf{y}))]$ generated by all diagrams ending at the normalization $N((2i)\cdot \mf{y})$ because $(2i)\cdot \mf{y} \cong N((2i)\cdot \mf{y})[\eta((2i)\cdot \mf{y}))]$ in $\cv$ by Lemma \ref{normal}.

\vspace{.1cm}
\n(2) Recall that $\es(k)=(2k+2,2k+4,\dots) \in \cb_k$ is the vacuum object.
$$T_{2i-1}(k) \ot_{k} P(\es(k))=\bigoplus_{\mf{z}\in\cb_{k+1}}\Hom((2i)\cdot\mf{z},~\es(k))$$
is generated by all diagrams from $(2i)\cdot \mf{z}$ to $\es(k)$.
If $i \leq k$, then $2i < 2k+2 < 2k+4 < \cdots$.
In any diagram, the strand from $2i$ has no term to connect in $\es(k)$.
Hence the tensor product is zero.

\vspace{.1cm}
\n(3) We reduce the case to a combination of (1) and (2) by using the isomorphisms in Lemmas \ref{12=21+es}, \ref{32=23+es} and \ref{13=31}.
There exists $j\gg0$ such that $i \leq k+j$ and $\mf{y}=(y_1, \dots, y_j, \dots)$ satisfying $y_n=2n+2k$ for $n>j$.
Using (1) we can rewrite
$$P(\mf{y})=T_{y_1} \ot \cdots \ot T_{y_j} \ot P(\es(k+j)),$$
$$T_{2i-1} \ot P(\mf{y})=T_{2i-1} \ot T_{y_1} \ot \cdots \ot T_{y_j} \ot P(\es(k+j))=T \ot P(\es(k+j)),$$
where $T=T_{2i-1} \ot T_{y_1} \ot \cdots \ot T_{y_j}$.
By ignoring the grading, an $R$-bimodule $T_{2i-1} \ot T_{y_n}$ is
\be
\item either a direct sum of $T_{y_n} \ot T_{2i-1}$ and $R$ as in Lemmas \ref{12=21+es}, \ref{32=23+es};
\item or $T_{y_n} \ot T_{2i-1}$ as in Lemma \ref{13=31}.
\ee
We continuously exchange $T_{2i-1}$ with $T_{y_s}$ in $T$ for $s=1,\dots, j$ to obtain a decomposition of $T=T'\oplus T''$ as bimodules such that
\be
\item $T'$ is a direct sum of tensor products of $T_{y_s}$ for $1 \leq s \leq j$, which is possibly zero;
\item $T''= T_{y_1} \ot \cdots \ot T_{y_j}\ot T_{2i-1}$.
\ee

The tensor product $T \ot P(\es(k+j))$ is obtained from $$(T' \ot P(\es(k+j))) \oplus (T'' \ot P(\es(k+j)))$$ by adding some differentials,
where the first summand is in $DGP(R)$ by (1) and the second one is zero by (2).
Hence $T_{2i-1} \ot P(\mf{y})=T \ot P(\es(k+j))$ is in $DGP(R)$ as well.
\end{proof}

Since $DGP(R)$ is generated by the $P(\mf{x})$'s, we proved that tensoring with $T_i$ maps $DGP(R)$ to itself.
We then have the categorical action via the functor $\tau: \cl \ra DG(R\ot R^{op})$:
$$\mathcal{G}: \cl \times DGP(R) \xrightarrow{} DGP(R).$$
The induced map on their homology categories descends to a linear map $K_0(\cal{G}): K_0(\cl) \times V \ra V$ under the isomorphism in Lemma \ref{k0}.
The pullback of $K_0(\cal{G})$ under $\gamma: \Cl \ra K_0(\cl)$ in Proposition \ref{K0clsur} gives an action of $Cl_{\Z}$ on $V$.

\begin{lemma} \label{k0action}
The pullback action agrees with the Fock space representation $G: \Cl \times V \ra V$.
\end{lemma}
\begin{proof}
From Lemma \ref{linear}, it suffices to check that
\begin{align*}
a_{2j}\cdot v =\eta_j(v) \quad\quad & \mbox{for}~~ v \in \hb, \\
a_{2j-1}\cdot \overline{|k\ran}=0 \quad\quad & \mbox{for}~~ k \geq j.
\end{align*}
The action of $a_{i}$ is a decategorification of tensoring with $T_i$.
The equations follow from Lemma \ref{tensor}(1), (2).
\end{proof}

\begin{proof}[Proof of Theorems \ref{thmcl} and \ref{thmv}]
The action $G: \Cl \times V \ra V$ is faithful by Lemma \ref{faithful}.
Since $G$ is the pullback of $K_0(\cal{G})$ under $\gamma: Cl \ra K_0(\cl)$, the map $\gamma$ must be injective.
Together with $\gamma$ being surjective by Proposition \ref{K0clsur}, we proved that $K_0(\cl) \cong Cl$.
\end{proof}
The author would like to thank Anthony Licata for suggestions to prove the injectivity of $\gamma$ using the faithfulness of $V$.

\end{document}